\documentclass{amsart}
\pdfoutput=1

\usepackage[utf8]{inputenc}
\usepackage{amsfonts}
\usepackage{amsmath,enumerate}
\usepackage[usenames,dvipsnames,svgnames]{xcolor}

\newtheorem{thm}{Theorem}[section]
\newtheorem{lemma}[thm]{Lemma}
\newtheorem{prp}[thm]{Proposition}
\newtheorem{corollary} [thm]{Corollary}

\newtheorem{remark}[thm]{Remark}
\newtheorem{example}[thm]{Example}
\newtheorem{definition}[thm]{Definition}

\newcommand{\vc}[1]{\mathbf{#1}}

\newcommand{\symm}{\mathrm{symm}}

\newcommand{\vca}{\vc{a}}

\newcommand{\xx}{\vc{x}}
\newcommand{\yy}{\vc{y}}

\newcommand{\zz}{\vc{z}}

\newcommand{\vv}{\vc{v}}
\newcommand{\ww}{\vc{w}}

\newcommand{\ws}{\vc{w}_*}
\newcommand{\ys}{\yy_*}
\newcommand{\wS}{\vc{w}^*}
\newcommand{\yS}{\yy^*}

 \newcommand{\ue}{\mathrm{e}}

\newcommand{\ee}{\vc{e}}
\newcommand{\ff}{\vc{f}}

\newcommand{\vcg}{\vc{g}}

\newcommand{\Kki}{K^{(k)}_i}
\newcommand{\Kkj}{K^{(k)}_j}
\newcommand{\ek}{{\ee}^{(k)}}

\newcommand*{\TT}{{\mathbb{T}} }
\newcommand*{\RR}{{\mathbb{R}} }
\newcommand*{\NN}{{\mathbb{N}} }
\newcommand*{\ZZ}{{\mathbb{Z}} }
\newcommand*{\CC}{{\mathbb{C}} }

\newcommand*{\DD}{{\mathcal{D}} }

\newcommand*{\II}{{\mathbb{I}} }

\newcommand*{\downto}{\downarrow}
\newcommand*{\upto}{\uparrow}
\newcommand*{\trp}{\top}
\newcommand*{\Ball}{B}
\newcommand*{\mm}{\underline{m}}
\newcommand*{\MM}{\overline{m}}

\newcommand{\mkx}{\MM^{(k)}({\bf x})}
 \newcommand{\Fk}{F^{(k)}}

 \newcommand{\MMk}{\MM^{(k)}}

\newcommand*{\piros}{}
\newcommand*{\kek}{}

\newcommand*{\lila}{}

\DeclareMathOperator{\inter}{int}
\DeclareMathOperator{\intt}{int}

\DeclareMathOperator{\supp}{supp}

 \newcommand{\id}{\mathrm{id}}

 \newcommand{\Ce}{\mathrm{C}}

 \newcommand{\RRR}{\Bar{\RR}}
\def\qedhere{}
\newenvironment{iiv}{\begin{enumerate}[{\rm (i)}]}{\end{enumerate}}
\newenvironment{abc}{\begin{enumerate}[{\rm (a)}]}{\end{enumerate}}
\newenvironment{ABC}{\begin{enumerate}[{\rm (A)}]}{\end{enumerate}}
\definecolor{red}{rgb}{1,0,0}

\title{A  minimax problem for sums of translates on the torus}

\author{B\'alint Farkas, B\'ela Nagy and Szil\'ard Gy. R\'ev\'esz}

\thanks{This work was supported by the Hungarian Science Foundation, Grant \#'s K-100461, NK-104183, K-109789.}

\date{}

\begin{document}
\maketitle
\begin{abstract}
We extend some equilibrium-type results first conjectured by Ambrus,
Ball and Erdélyi, and then proved recently by Hardin,
Kendall and Saff. We work on the
torus $\TT\simeq[0,2\pi)$, but the motivation comes from
an analogous setup on the unit interval, investigated earlier by
Fenton.

The problem is to minimize---with respect to the arbitrary translates
$y_0=0, y_j\in \TT$, $j=1,\dots,n$---the maximum of the sum function
$F:=K_0 + \sum_{j=1}^n K_j(\cdot-y_j)$, where the $K_j$'s are certain
fixed ``kernel functions''. In our setting, the function $F$ has singularities at $y_j$’s, while in between these nodes it still behaves regularly. So one can consider the
maxima $m_i$ on each subinterval between the nodes $y_j$, and minimize $\max F = \max_i m_i$. Also the dual question of maximization of $\min_i m_i$ arises.

Hardin, Kendall and Saff considered one \emph{even} kernel,
$K_j=K$ for $j=0,\dots,n$, and Fenton considered the case of the
interval $[-1,1]$ with \emph{two} fixed kernels $K_0=J$ and $K_j=K$
for $j=1,\dots,n$. Here we build up a systematic treatment when \emph{all the kernel functions can be different}
without assuming them to be even.  As an application we generalize a
result of Bojanov about Chebyshev type polynomials with
prescribed zero order.

MSC2010 Primary: 49J35
Secondary: 26A51,
42A05,
90C47

\end{abstract}
\section{Introduction}
\label{sec:intro}

The present work deals with an ambitious extension of an equilibrium-type result, conjectured by  Ambrus,  Ball and Erd\'elyi \cite{ABE} and recently proved by Hardin, Kendall and Saff \cite{Saff}. To formulate this equilibrium result, it is convenient to identify the unit circle (or one dimensional torus) $\TT$, $\RR/2\pi\ZZ$ and $[0,2\pi)$, and call a function $K:\TT \to \RR\cup\{-\infty,\infty\}$ a \emph{kernel}.
The setup of \cite{ABE} and \cite{Saff}  requires that the kernel function is \emph{convex} and has values in $\RR\cup \{\infty\}$. However, due to historical reasons, {\piros described below,} we shall suppose that the kernels are \emph{concave} and have values in $\RR\cup\{-\infty\}$, the transition between the two settings is
a trivial multiplication by $-1$.
Accordingly, we take the liberty to reformulate the results of \cite{Saff} after a multiplication by $-1$, so in particular for concave kernels, see Theorem \ref{th:ABE-HKS} below.

\medskip The setup of our investigation is therefore that some \emph{concave} function $K:\TT\to \RR\cup\{-\infty\}$ is fixed, meaning that $K$ is concave on $[0,2\pi)$.
Then $K$ is necessarily either finite valued (i.e., $K:\TT\to \RR$) or it satisfies $K(0)=-\infty$ and $K:(0,2\pi)\to \RR$ (the degenerate
situation when $K$ is
constant $-\infty$ is excluded), and $K$ is upper semi-continuous on $[0,2\pi)$,
and continuous on $(0,2\pi)$.

\medskip
The kernel functions are extended periodically to $\RR$ and we consider
the sum of translates function
\[F(y_0,\dots,y_n,t):=\sum_{j=0}^n K(t-y_j).\]
The points $y_0,\dots,y_n$ are called \emph{nodes}.
Then we are interested in solutions of the minimax problem
\begin{align*}
\inf_{y_0,\dots, y_n\in [0,2\pi)}\sup_{t\in [0,2\pi)} \sum_{j=0}^n K(t-y_j)=\inf_{y_0,\dots, y_n\in [0,2\pi)}\sup_{t\in [0,2\pi)} F(y_0,\dots,y_n,t),
\end{align*}
and address questions concerning existence and uniqueness of solutions, as well as the distribution of the points $y_0,\dots, y_n$ (mod $2\pi$) in such extremal situations.

\medskip
In \cite{ABE} it was shown that for $K(t):=-|\ue^{it}-1|^{-2}=-\frac14 \sin^{-2}(t/2)$, 
(which comes from the Euclidean distance $|\ue^{it}-\ue^{is}|=2\sin((t-s)/2)$ 
between points of the unit circle on the complex plane), 
$\max F$ is minimized exactly for the 
regular, 
{\lila in other words,} 
equidistantly spaced, 
configuration of points, i.e., {\lila if we normalize by taking $y_0=0$, then $y_j=2\pi j/(n+1)$ for $j=0,\dots,n$. } 
(The authors in \cite{ABE} mention
that the concrete problem stems from a certain extremal problem, called ``strong polarization constant problem''
by
\cite{AmbrusPhD}.)

Based on this and natural heuristical considerations, Ambrus, Ball and Erd\'elyi conjectured that the same phenomenon should hold also when $K(t):=-|\ue^{it}-1|^{-p}$ ($p>0$), and, moreover, even when $K$ is any concave kernel (in the above sense). Next, this was proved for $p=4$ by Erd\'elyi and Saff \cite{ES}. Finally, in \cite{Saff} the full conjecture of Ambrus, Ball and Erd\'elyi was indeed settled for symmetric (even) kernels.

\begin{thm}[(Hardin, Kendall, Saff)]\label{th:ABE-HKS}
Let $K$ be any concave kernel function.
such that $K(t)=K(-t)$.
For any $ 0=y_0\le y_1\le \ldots \le y_n <2\pi$
write $\yy:=(y_1,\ldots,y_n)$ and
$F(\yy,t):= K(t)+ \sum_{j=1}^n K(t-y_j)$.
Let $\ee:=(\frac{2\pi}{n+1},\ldots,\frac{2\pi n} {n+1})$ (together with $0$ the equidistant node system in $\TT$).
\begin{abc}
\item Then
\[
\inf_{0=y_0\le y_1\le \ldots \le y_n <2\pi} \sup_{t\in\TT}
F(\yy,t) = \sup_{t\in\TT} F(\ee,t),
\]
i.e., the smallest supremum is attained at the equidistant configuration.

\item
Furthermore, if $K$ is strictly concave, then
the smallest supremum is attained at the equidistant configuration \emph{only}.
\end{abc}
\end{thm}

We thank the anonymous referee for drawing our attention to a
results of Erdélyi, Hardin and Saff \cite{ErdelyiHardinSaff}.
They reestablished Theorem \ref{th:ABE-HKS} 
with a different method and then they applied it in proving
an inverse Bernstein type inequality.

\medskip
Although this might seem as the end of the story, it is in fact not. The equilibrium phenomenon,
captured by this result, is indeed much more general, when we interpret it from a proper point of view. However, to generalize further, we should first analyze what more general situations we may address and what phenomena
we can expect to hold in the formulated more general situations. Certainly, regularity in the sense of the nodes $y_j$ distributed \emph{equidistantly} is a rather strong property, which is intimately connected to the use of one single and fixed kernel function $K$. 
However, this regularity obviously entails \emph{equality of the ``local maxima''} (suprema) $m_j$ {\piros on the arc between $y_j$ and $y_{j+1}$} for all $j=0,1,\dots,n$, and this is what is usually natural in such equilibrium questions.

We say that the configuration of points $0=y_0\leq y_1\leq \cdots \leq y_n\leq y_{n+1}=2\pi$ \emph{equioscillates}, if
\[
m_j(y_1,\dots, y_n):=\sup_{t\in [y_j,y_{j+1}]}F(y_1\dots,y_n,t)=\sup_{t\in [y_i,y_{i+1}]}F(y_1,\dots,y_n,t)=:m_i(y_1,\dots, y_n)\]
holds for all $i,j\in\{0,\dots,n\}$.
Obviously, with one single and fixed kernel $K$, if the nodes are equidistantly spaced, then the configuration equioscillates.
In the more general setup, this ---as will be seen from this work--- is a good replacement for the property that a point configuration is  equidistant.

\medskip
To give a perhaps enlightening example of what we have in mind, let us recall here a remarkable, but regrettably almost forgotten result of Fenton (see \cite{Fenton}), in the analogous, yet also somewhat different situation, when the underlying set is not the torus $\TT$, but the unit interval $\II:=[0,1]$. In this setting the underlying set is not a group, hence defining translation $K(t-y)$ of a kernel $K$ can only be done if we define the basic kernel function $K$ not only on $\II$ but also on $[-1,1]$. Then for any $y\in\II$ the translated kernel $K(\cdot-y)$ is well-defined on $\II$, moreover, it will have analogous properties to the above situation, provided we assume $K|_{\II}$ and also $K|_{[-1,0]}$ to be concave. Similarly, for any node systems the analogous sum $F$ will have similar properties to the situation on the torus.

From here one might derive that under the proper and analogous conditions, a similar regularity (i.e., equidistant node distribution) conclusion can be drawn also for the case of $\II$. But this is \emph{not the only} result of Fenton, who indeed did dig much deeper.

Observe that there is one rather special role, played by the \emph{fixed} endpoint(s) $y_0=0$ (and perhaps $y_{n+1}=1$), since perturbing a system of nodes the respective kernels are translated---but not the one belonging to $K_0:=K(\cdot-y_0)$, since $y_0$ is fixed. In terms of (linear) potential theory, $K=K(\cdot-y_0)=:K_0$ is a fixed \emph{external field}, while the other translated kernels play the role of a certain ``gravitational field'', as observed when putting (equal) point masses at the nodes. The potential theoretic interpretation is indeed well observed already in \cite{ES}, where it is mentioned that the \emph{Riesz potentials} with exponent $p$ on the circle correspond to the special problem of Ambrus, Ball and Erd\'elyi. From here, it is only a little step further to separate the role of the varying mass points, as generating the corresponding gravitational fields, from the stable one, which may come from a similar mass point and law of gravity---or may come from anywhere else.

Note that this potential theoretic external field consideration is far from being really new. To the contrary, it is the fundamental point of view of studying weighted polynomials (in particular, orthogonal polynomial systems with respect to a weight), which has been introduced by the breakthrough paper of Mhaskar and Saff \cite{MS} and developed into a far-reaching theory in \cite{ST} and several further treatises. So in retrospect we may interpret the factual result of Fenton as an early (in this regard, not spelled out and very probably not thought of) external field generalization of the equilibrium setup considered above.

\begin{thm}[(Fenton)]\label{th:Fenton} Let $K:[-1,1]\to \RR\cup\{-\infty\}$ be a kernel function in $\Ce^2(0,2\pi)$ which is concave and  which is
monotone
both on $(-1,0)$ and $(0,1)$ with $K''<0$ and $D_\pm K(0)=\pm \infty$
that is, the left- and right-hand side derivatives of $K$ at $0$ are $-\infty$ and $+\infty$, respectively.
Let $J:(0,1)\to \RR$ be a
concave
function and put $J(0):=\lim_{t\rightarrow 0}J(t)$, $J(1):=\lim_{t\rightarrow 1}J(t)$
which could be $-\infty$ as well.
For $\yy=(y_1,\ldots,y_n) \in [0,1]^n$ consider
\[
F(\yy,t):=J(t)+\sum_{j=0}^{n+1} K(t-y_j),
\]
where $y_0:=0$, $y_{n+1}:=1$.
Then the following are true:
\begin{abc}
\item There are
$0=w_0\leq w_1\leq \cdots\leq w_n\leq w_{n+1}=1 $
such that with $\ww=(w_1,\dots,w_n)$
\[
\inf_{0\leq y_1\leq\cdots\leq y_n\leq 1}\max_{j=0,\dots,n} \sup_{t\in [y_j,y_{j+1}]}F(\yy,t)=\sup_{t\in [0,1]} F(\ww,t).
\]
\item The
sum of translates function
of  $\ww$ equioscillates, i.e.,
\[
\sup_{t\in [w_j,w_{j+1}]}F(\ww,t)=\sup_{t\in [w_i,w_{i+1}]}F(\ww,t)
\]
for all $i,j\in \{0,\dots, n\}$.
\item We have
\[
\inf_{0\leq y_1\leq\cdots\leq y_n\leq 1} \max_{j=0,\dots,n} \sup_{t\in [y_j,y_{j+1}]}F(\yy,t)=
\sup_{0\leq y_1\leq\cdots\leq y_n\leq 1} \min_{j=0,\dots,n} \sup_{t\in [y_j,y_{j+1}]}F(\yy,t).
\]
\item If $0\leq z_1\leq \cdots\leq z_n\leq 1$ is a configuration
such that the  sum of translates function $F(\zz,\cdot)$
equioscillates, then $\ww=\zz$.
\end{abc}
\end{thm}

This gave us the first clue and impetus to the further, more general investigations, which, however,
{\piros have been executed for the torus setup here}. As regards Fenton's framework, i.e., similar questions on the interval, we plan to
return to them in a subsequent paper. 
The two setups are rather different in technical details, and we found it difficult to explain them simultaneously---while in principle they should indeed be the same. 
Such an equivalency is at least exemplified also in this paper, when we apply our results to the problem of Bojanov on so-called ``restricted Chebyshev polynomials'': In fact, the original result of Bojanov (and our generalization of it) is formulated on an interval. 
So in order to use our results, valid on the torus, we must work out both some corresponding (new) results on the torus itself, and also a method of transference (working well at least in the concrete Bojanov situation). 
The transference seems to work well in symmetric cases, but becomes intractable for non-symmetric ones. 
Therefore, it seems that to capture full generality, not the transference, but direct, analogous arguments should be used. 
This explains our decision to restrict current considerations to the case of the torus only. 
{\lila Let us also mention here a recent, interesting manuscript by D.~Benko, D.~Coroian, P.~D.~Dragnev and  R.~Orive \cite{BCDOarxiv} where the authors investigate a statistical problem which 
{\lila is a case of the interval setting of the minimax problem here.} 
}

\medskip
Nevertheless, as for generality of the results, the reader will see that we indeed make a further step, too. Namely, we will allow not only an external field (which, for the torus case, would already be an extension of Theorem \ref{th:ABE-HKS}, analogous to Theorem \ref{th:Fenton}), but we will study situations when \emph{all the kernels}, fixed or translated, may as well be different. (Definitely, this makes it worthwhile to work out subsequently the analogous questions also for the interval case.) 

{\piros
The following exemplifies one of the main results of this paper, formulated here without the convenient terminology developed in the later sections.  
It is stated again in Theorem \ref{thm:mainspecialcase} below in a more concise way, and  it is proved in Section \ref{sec:summary} by using the techniques developed in the forthcoming sections. 
\begin{thm}
\label{thm:0mainspecialcase}
	Suppose the $2\pi$-periodic functions $K_0,K_1,\dots,K_n:\RR\to [-\infty,0)$ are strictly concave on $(0,2\pi)$ and either all are continuously differentiable on $(0,2\pi)$ or for each $j=0,1,\dots,n$ 
\[
\lim_{t\upto 2\pi} D_+ K_j(t)=\lim_{t\upto 2\pi} D_- K_j(t)=-\infty,\quad\text{or}\quad \lim_{t\downto 0} D_- K_j(t)=\lim_{t\downto 0} D_+ K_j(t)= \infty,
\]
$D_{\pm}K_j$ denoting the (everywhere existing) one sided derivatives of the function $K_j$.
For any $ 0=y_0\le y_1\le \ldots \le y_n <2\pi$
write $\yy:=(y_1,\ldots,y_n)$ and
$F(\yy,t):= K_0(t)+ \sum_{j=1}^n K_j(t-y_j)$.
Then there are $w_1,\dots,w_n\in (0,2\pi)$ such that
\begin{align*}
M:=\inf_{\yy\in\TT^n}\sup_{t\in \TT}F(\yy,t)=\sup_{t\in \TT} F(\ww,t),
\end{align*}
and the following hold:
\begin{abc}
	\item The points $0,w_1,\dots,w_n$ are pairwise different and hence determine a permutation $\sigma:\{1,\dots,n\}\to \{1,\dots,n\}$ such that $0<w_{\sigma(1)}<w_{\sigma(2)}<\cdots<w_{\sigma(n)}<2\pi$.
 Denote by $S$ the set of points $(y_1,\dots, y_n)\in \TT^n$ with $0<y_{\sigma(1)}<y_{\sigma(2)}<\cdots<y_{\sigma(n)}<2\pi$. A point $\yy\in S$ together with $y_0:=0$ determines $n+1$ arcs on $\TT$, denote by $I_j(\yy)$ the one that starts at $y_j$ and goes in the counterclockwise direction ($j=0,1\dots,n$). We have
 \[
\sup_{t\in I_0(\ww)} F(\ww,t)=\cdots=\sup_{t\in I_n(\ww)} F(\ww,t),
 \]
 for which we say that $\ww$ is an \emph{equioscillation point}.	
 \item  With the set $S$ from {\upshape (a)} we have
\begin{align*}
\inf_{\yy\in S}\max_{j=0,\dots,n}\sup_{t\in I_j(\yy)} F(\yy,t)=M=\sup_{\yy\in S}\min_{j=0,\dots,n}\sup_{t\in I_j(\yy)} F(\yy,t).
\end{align*}
\item  For each $\xx,\yy\in S$
\[
\min_{j=0,\dots,n}\sup_{t\in I_j(\xx)} F(\xx,t)\leq M\leq \max_{j=0,\dots,n}\sup_{t\in I_j(\yy)} F(\yy,t).
\]
This is called the \emph{Sandwich Property}.
\end{abc}
\end{thm}
With the help of this result we shall prove a strengthening of Theorem \ref{th:ABE-HKS} in Corollary \ref{cor:HKSgen}.}

\medskip\noindent {\piros A particular connection of this problem with physics is the  field of
Calogero-Moser and the trigonometric Calogero-Moser-Sutherland systems
(of type A and BC).
In those models, there are $n$ particles on the unit circle and the
interaction potential
corresponds to the kernel $1/\sin^2(x)$.
Roughly speaking, if the particles are closer, then the repulsion
force among them is stronger.
 The positions of n particles
depend on time $t$. If one of the particles is fixed, and the others
are in pairs which are symmetric (say, the fixed particle is at $0$,
and the others are at $x$ and $2\pi-x$), then it is of BC type.
The equilibrium state means that the particles 
{\lila do not }
move, in some
sense it is a minimal energy configuration.
Then it is a simple fact that the equilibrium
configuration is the equidistant configuration only (see, e.g.
\cite{CalogeroPerelomov1978}, p. 110).
See also \cite{Calogero1977}, which is on the real line. We thank G\'abor Pusztai for informing us and providing references. In this application the kernels are the same so one can apply the result of Hardin, Kendall, Saff.
}

It is not really easy to interpret the situation {\piros of different kernels }
in terms of physics or potential theory
anymore. However, one may argue that in physics we do encounter some situations, e.g., in sub-atomic scales,
when different forces and laws can be observed simultaneously: strong kernel forces, electrostatic and gravitational forces etc. 
{\piros Also it can be that in the one-dimensional $n$-body problem though the potentials are the same, but the masses of the particles are different. This leads to our formulation with different kernels, more specifically to Theorem \ref{thm:mainspecialcase3} below, where $K_j=r_j K$ 
{\lila with } 
numbers $r_j>0$.} 

\medskip In any case, the reader will see that the generality here is clearly a powerful one: e.g., the above mentioned new solution (and generalization and extension to the torus) of Bojanov's problem of restricted Chebyshev polynomials requires this generality. Hopefully, in other equilibrium type questions the generality of the current investigation will prove to be of use, too.

\medskip
{\piros In this introduction it is not yet possible to formulate all the 
results of this paper,
because we need to discuss a couple of technical
details first, to be settled in Section \ref{sec:basdef}. One such, but not only technical, matter is the loss of symmetry with respect to the \emph{ordering} of the nodes, cf.{} the statement (a) of the previous Theorem \ref{thm:0mainspecialcase}.
Indeed,
in case of a fixed kernel to be translated (even if the external field is different), all permutations of the nodes $y_1,\dots,y_n$ are equivalent,
while
for different kernels $K_1,\dots,K_n$ we of course must distinguish between situations when the ordering of the nodes differ. Also, the original extremal problem can have \emph{different interpretations} according to consideration of \emph{one fixed order} of the kernels (nodes), or \emph{simultaneously all possible orderings} of them. We will treat \emph{both} types of questions, but the \emph{answers will be different}.
This is not a technical matter:
We will see that, e.g., it can well happen that in some prescribed ordering of the nodes (i.e., the kernels) the extremal configuration has equioscillation, while in some other ordering that fails.

\medskip 
We shall progress systematically with the aim of being as self-contained as possible and defining notation, properties and discussing details step by step. 
Our main result will only be proved in Section \ref{sec:summary}. 
In Section \ref{sec:basdef} we will first introduce the setup precisely, most importantly we will discuss the role of the permutation $\sigma$ appearing in 
Theorem \ref{thm:0mainspecialcase}, hoping that the reader will be satisfied with the motivation provided by this introduction. 
In subsequent sections we will discuss various aspects: continuity properties in Section \ref{sec:continuity}, 
other elementary properties motivated by Shi's setup \cite{shi}---like the Sandwich Property in Theorem \ref{thm:0mainspecialcase} (c)---in  Sections  \ref{sec:elemprops} and \ref{sec:equi}, limits and approximations in Section \ref{sec:approx}, concavity, distributions of local extrema in Sections \ref{sec:distrlocmins}, \ref{sec:concavity}, and \ref{sec:localprop}, 
existence and uniqueness of equioscillation points---as in Theorem \ref{thm:0mainspecialcase} (b)---in Section \ref{sec:equioscillation}. 
This systematic treatment is 
not only 
justified by the final proof  of Theorem \ref{thm:0mainspecialcase} and its far reaching consequences (an extension of the Hardin--Kendall--Saff result, see Corollary \ref{cor:HKSgen}, or Theorems \ref{thm:mainspecialcase3} and \ref{thm:bojanovgen}), 
but {\lila also} 
the developed techniques, such as Lemma \ref{lem:pertnodes00} or those in Section \ref{sec:approx}, are interesting in their own right and have the potential to prove themselves to be useful attacking also problems different from the present one. In Section \ref{sec:HKSmini} we sharpen the result, Theorem \ref{th:ABE-HKS}, of Hardin, Kendall, Saff by dropping the condition of the symmetry of the kernel. Finally, in Section \ref{sec:bojanovext} we shall describe, how extensions of  Bojanov's results can be derived via our equilibrium results.
}

\section{The setting of the problem}
\label{sec:basdef}
{\piros In this section we set up the framework and the notation for our investigations.}

For given $2\pi$-periodic kernel functions $K_0,\dots, K_n:\RR\to[-\infty,\infty)$ we are interested in solutions of minimax problems like
\begin{align*}
\inf_{y_0,\dots, y_n\in [0,2\pi)}\sup_{t\in [0,2\pi)} \sum_{j=0}^n K_j(t-y_j),
\end{align*}
and address questions concerning existence and uniqueness of solutions, as well as the distribution of the points $y_0,\dots, y_n$ (mod $2\pi$) in such extremal situations.
In the case when $K_0=\cdots=K_n$ similar problems were studied by Fenton \cite{Fenton} (on intervals), Hardin, Kendall and Saff \cite{Saff} (on the unit circle).
For twice continuously differentiable kernels an abstract framework for handling of such minimax problems was developed by Shi \cite{shi}, which in turn  is  based on the fundamental works of Kilgore \cite{kilgore1}, \cite{kilgore2}, and
de Boor, Pinkus \cite{deboor}  concerning  interpolation theoretic conjectures of Bernstein and Erd\H os. 
{\piros Apart from the fact that we do not generally pose $\Ce^2$-smoothness conditions on the kernels (as required by the setting of Shi), it will turn out that Shi's framework is not applicable in this general setting (cf.{} Example \ref{ex:diffbutnostrong} and Section \ref{sec:equi}).} 
The exact references will be given at the relevant places below, but let us  stress already here that we do not assume the functions $K_j$ to be smooth (in contrast to \cite{shi}), and that they may be different (in contrast to \cite{Fenton} and \cite{Saff}).

\medskip For convenience
we use the identification of
the unit circle (torus) $\TT$ with the interval $[0,2\pi)$ (with addition mod $2\pi$), and consider $2\pi$-periodic functions also as functions on $\TT$; we shall use the terminology of both frameworks, whichever comes more handy. 
So that we may speak about concave functions on $\TT$ (i.e., on $[0,2\pi)$), just as about arcs in $[0,2\pi)$ (i.e., in $\TT$); this shall cause no ambiguity. We also use the notation
\begin{equation}\label{eq:tmetric}
d_{\TT}(x,y)=\min\bigl\{|x-y|,2\pi-|x-y|\bigr\}\quad (x,y\in [0,2\pi]),
\end{equation}
and
\begin{equation}\label{eq:tnmetric}
d_{\TT^m}(\xx,\yy)=\max_{j=1,\dots,m}d_\TT(x_j,y_j) \quad (\xx,\yy\in \TT^m).
\end{equation}

\medskip {\piros
Let 
{\lila $K:(0,2\pi)\rightarrow(-\infty,\infty)$ } 
be a concave  function
which is not identically
$-\infty$, and suppose
\[
K(0):=\lim_{t\downto 0} K(t)=\lim_{t\upto 2\pi} K(t)=:K(2\pi),
\]
i.e., the two limits exist and they are the same. }
Such a function $K$ will
be called a \emph{concave kernel function} and
can be regarded as
a function on the torus $\TT$.

One of the conditions on the kernels that will be considered is the following:
\begin{equation}
\label{eq:kernsing}\tag{\hbox{$\infty$}}
K(0)=K(2\pi)=-\infty.
\end{equation}
{\piros Denote by $D_-f$ and $D_+f$ the left and right derivatives of a function $f$ defined on an interval, respectively. 
A \emph{concave} function $f$, defined on an open interval possesses at each points left and right derivatives $D_-f$, $D_+f$ with
{\lila $D_-f\leq D_+f$,  }
and these are non-increasing functions; moreover, $f$ is 
differentiable 
almost everywhere and (the a.e.~defined) $f'$ is non-increasing. 
Then, under condition \eqref{eq:kernsing} it is obvious that we must also have that
\begin{align}
\label{eq:kernsing'm}\tag{\hbox{$\infty'_-$}} &&\lim_{t\upto 0} D_+ K(t)=\lim_{t\upto 2\pi} D_+ K(t)&=\lim_{t\upto 2\pi} D_- K(t)=\lim_{t\upto 0} D_- K(t)=-\infty, \notag
\\\label{eq:kernsing'p}\tag{\hbox{$\infty'_+$}} \text{and}&&\lim_{t\downto 2\pi} D_- K(t)=\lim_{t\downto 0} D_- K(t)&=\lim_{t\downto 0} D_+ K(t)=\lim_{t\downto 2\pi} D_+ K(t)= \infty\notag.
\end{align}
We can  abbreviate this  by writing  $D_{\pm} K(2\pi)=D_{\pm} K(0)=\pm \infty$. 
These  assumptions then imply  $K'(\pm 0)=\pm \infty$. 
The two conditions \eqref{eq:kernsing'm} and \eqref{eq:kernsing'p} together constitute
\begin{equation}
D_{-} K(2\pi)=D_{-} K(0)=- \infty\quad\text{and}\quad D_{+} K(2\pi)=D_{+} K(0)=\infty. 
\tag{\hbox{$\infty'_{\pm}$}}
\label{eq:kernsing'pm}
\end{equation}}
More often, however, we shall make the following assumption on the kernel $K$:
\begin{equation}
\label{eq:kernsing'}\tag{\hbox{$\infty'$}}
D_{-} K(0)=- \infty\quad\text{or}\quad D_{+} K(0)=\infty.
\end{equation}

\medskip For $n\in \NN$ fixed let $K_0,\dots, K_n$ be concave kernel functions. We take $n+1$ points $y_0,y_1,y_2,\dots,y_n\in [0,2\pi)$, called \emph{nodes}. 
As a matter of fact, for definiteness, we shall always take $y_0=0\equiv 2\pi\mod 2\pi$. Then $\yy=(y_1,\dots,y_n)$ is called a \emph{node system}. For notational convenience we also
set $y_{n+1}=2\pi$.
 For a given node system $\yy$ we consider the function
\begin{equation}
\label{def:sumoftranslates}
F(\yy,t):=\sum_{j=0}^nK_j(t-y_j)=K_0(t)+\sum_{j=1}^nK_j(t-y_j).
\end{equation}
For a permutation $\sigma$ of $\{1,\dots,n\}$ we introduce the notation $\sigma(0)=0$ and $\sigma(n+1)=n+1$, and define the simplex
\[
S_\sigma:=\bigl\{\yy\in \TT^n: 0=y_{\sigma(0)}<y_{\sigma(1)}< \cdots < y_{\sigma(n)}< y_{\sigma(n+1)}=2\pi\bigr\}.
\]
In this paper the term \emph{simplex} is reserved exclusively for domains of this form.
{\piros Then $S_\sigma$ is
an
open subset of $\TT^n$ with
\[
\bigcup_{\sigma}\overline{S}_\sigma=\TT^n
\]
(here and in the future $\overline{A}$ denotes the closure of the set $A$) 
and the complement $\TT^n\setminus X$ of the set $X:=\bigcup_{\sigma}{S}_\sigma$
is the union of less than $n$-dimensional simplexes. Given a permutation $\sigma$ and
$\yy\in \overline{S_\sigma}$,
for  $k=0,\dots, n$ we define the arc $I_{\sigma,\sigma(k)}(\yy)$ (in the counterclockwise direction)
\[
I_{\sigma,\sigma(k)}(\yy):=[y_{\sigma(k)},y_{\sigma(k+1)}].
\]
For $j=0,\dots,  n$ we have $I_{\sigma,j}(\yy)=[y_j,y_{\sigma(\sigma^{-1}(j)+1)}]$. Of course, a priori, nothing prevents that some of these arcs $I_{\sigma,j}(\yy)$ reduce to a singleton, but their lengths sum up to $2\pi$
\[
\sum_{j=0}^n|I_{\sigma,j}(\yy)|=2\pi.
\]
Most of the time we will fix a simplex, hence a permutation $\sigma$. In this case we will leave out the notation of $\sigma$, and write $I_j(\yy)$ instead of $I_{\sigma,j}(\yy)$. 
If $\yy\in X$ the notation of $\sigma$ would be even superfluous, because, in this case, $\yy$  belongs to the interior of some uniquely determined simplex $S_\sigma$. Hence, $j$ and $\yy\in X$ uniquely determine $I_{\sigma,j}(\yy)$. However, for $\sigma\neq{\sigma'}$ and for $\yy\in \overline{S_\sigma}\cap \overline{S}_{\sigma'}$   on the (common) boundary, the system of arcs is still well defined,  
but the 
numbering of the arcs does 
depend on the permutations $\sigma'$ and $\sigma$.

\medskip \noindent
We set
\[
m_{\sigma,j}(\yy):=\sup_{t\in I_{\sigma,j}(\yy)} F(\yy,t),
\]
and as above, if $\sigma$ is unambiguous from the context, or if it is immaterial for the considerations, we leave out its notation, i.e., simply write $m_j(\yy)$. Saying that $S=S_\sigma$ is a simplex implies that the permutation $\sigma$ is fixed and the ordering of $m_j$ is understood accordingly.}

We also introduce the functions
\begin{align*}
\MM: &\TT^n\to [-\infty,\infty),\quad \MM(\yy):=\max_{j=0,\dots,n}m_j(\yy)=\sup_{t\in\TT} F(\yy,t),\\
\mm:&\TT^n\to [-\infty,\infty),\quad \mm(\yy):=\min_{j=0,\dots,n}m_j(\yy).
\end{align*}
{\piros (For example, here it is immaterial which $\sigma$ is chosen for a  particular $\yy$.)} Of interest are then the following two minimax type expressions:
\begin{align}
\label{eq:minmax}M&:=\inf_{\yy\in \TT^n} \MM(\yy)=\inf_{\yy\in \TT^n}\max_{j=0,\dots,n} m_j(\yy)=\inf_{\yy\in \TT^n} \sup_{t\in\TT} F(\yy,t),
\\ \label{eq:maxmin} m&:=\sup_{\yy\in \TT^n} \mm(\yy) = \sup_{\yy\in \TT^n} \min_{j=0,\dots,n}m_j(\yy).
\end{align}
Or, more specifically, for any given simplex $S=S_\sigma$ we may consider the problems:
\begin{align}
\label{eq:Sminmax}M(S)&:=\inf_{\yy\in S} \MM(\yy)=\inf_{\yy\in S}\max_{j=0,\dots,n}m_j(\yy)=\inf_{\yy\in S} \sup_{t\in\TT} F(\yy,t),\\
\label{eq:Smaxmin} m(S)&:=\sup_{\yy\in S} \mm(\yy)=\sup_{\yy\in S}\min_{j=0,\dots,n}m_j(\yy).
\end{align}
For notational convenience for any given set $A\subseteq \TT^{n}$ we also define
\begin{align*}
M(A):&=\inf_{\yy\in A}\MM(\yy)=\inf_{\yy\in A}\max_{j=0,\dots,n}m_j(\yy)=\inf_{\yy\in A} \sup_{t\in\TT} F(\yy,t) ,
\\ m(A):&=\sup_{\yy\in A}\mm(\yy)=\sup_{\yy\in A}\min_{j=0,\dots,n}m_j(\yy).
\end{align*}
It will be proved in Proposition \ref{prop:exists} below that $m(S)=m(\overline{S})$ and $M(S)=M(\overline{S})$.
Observe that then
we can also write
\begin{align}
\label{eq:minmaxplus}M &= \min_{\sigma}\inf_{\yy\in \overline{S_\sigma}} \MM(\yy)= \min_{\sigma} M(\overline{S}_\sigma),
\\ \label{eq:maxminplus} m&= \max_{\sigma}\sup_{\yy\in \overline{S_\sigma}} \mm(\yy) = \max_{\sigma} m(\overline{S}_\sigma).
\end{align}
We are interested in whether the infimum or supremum are always attained, and if so, what can be said about the extremal configurations.

\begin{example}
If the kernels are only concave  and not strictly concave, then the minimax  problem  \eqref{eq:Sminmax} may have many solutions, even
on the boundary $\partial S$ of $S=S_\sigma$.
Let $n$ be fixed, $K_0=K_1=\cdots=K_n=K$ and let $K$ be a symmetric kernel ($K(t)=K(2\pi-t)$)  which is constant $c_0$ on the interval $[\delta, 2\pi-\delta]$, where $\delta<\frac{\pi}{n+1}$. Then for any node system $\yy$ we have $\max_{t\in\TT^n} F(\yy,t)= (n+1) c_0$, because the $2\delta$ long intervals around the nodes cannot cover $[0,2\pi]$.
\end{example}

\begin{prp}\label{prp:estbelow}
For every $\delta>0$ there is $L=L(K_0,\dots,K_n,\delta)\geq 0$ such that for every $\yy\in \TT^{n}$ and for every $j\in \{0,\dots,n\}$ with $|I_j(\yy)|>\delta$ one has $m_j(\yy)\geq -L$.
\end{prp}
\begin{proof}
Let $\delta\in (0,2\pi)$. Each function $K_j$, $j=0,\dots, n$ is bounded from below by $-L_j:=-L_j(\delta)\leq 0$ on $\TT\setminus (-\delta/2, \delta/2)$. So that for $\yy\in \TT^n$ the function $F(\yy,t)$ is bounded from below by $-L:=-(L_0+\cdots+ L_n)$ on $B:=\TT\setminus \bigcup_{j=0}^n(y_j-\delta/2,y_j+\delta/2)$.
Let $\yy\in \TT^{n}$ and $j\in \{0,\dots,n\}$ be such that $|I_j(\yy)|>\delta$, then there is $t\in B\cap I_j(\yy)$, hence  $m_j(\yy)\geq -L$.
\end{proof}

\begin{corollary} \label{cor:mMfin}
\begin{abc}
\item
The mapping $\MM$
is finite valued on $\TT^n$.
\item
$\MM$ is bounded.
\item For each simplex $S:=S_\sigma$ we have that $m(S),M(S)$ are finite, in particular $m,M\in \RR$.
\end{abc}

\end{corollary}
\begin{proof}
Since $K_0,\dots, K_n$ are bounded from above, say by $C\geq 0$, $F(\yy,t)\leq (n+1)C$ for every $t\in\TT$ and $\yy\in \TT^{n}$. This yields $m(S),M(S)\leq (n+1)C$.

Take any $\yy\in S$ consisting of distinct nodes, so
$m_j(\yy)>-\infty$ for each $j=0,\dots,n$. Hence $m(S)\geq\min_{j=0,\dots,n}m_j(\yy)>-\infty$.

For $\delta:=\frac{2\pi}{n+2}$ take $L\geq0$ as in Proposition \ref{prp:estbelow}. Then for every $\yy\in S$
there is $j\in \{0,\dots,n\}$ with $|I_j(\yy)|>\delta$, so that for this $j$ we have $m_j(\yy)\geq-L$. This implies
$M(S)\geq M \geq -L>-\infty$.
\end{proof}

	\section{Continuity properties}
\label{sec:continuity}
	{\piros In this section we study the continuity properties of the various functions, $m_j$, 
$\mm$, $\MM$,  defined in Section \ref{sec:basdef}.
As
a consequence, we prove that for each of the problems \eqref{eq:Sminmax}, \eqref{eq:Smaxmin} extremal configurations exist, this is Proposition \ref{prop:exists}, a central statement of this section.}

	\medskip\noindent
	To facilitate the argumentation we shall consider $\RRR=[-\infty,\infty]$ endowed with the metric \[
    d_{\RRR}:
    [-\infty,\infty]\to \RR,\quad
    d_{\RRR}(x,y)
    :=|\arctan(x)-\arctan(y)|\]
	which makes it a compact metric space, with convergence meaning the usual convergence of real sequences to some finite or infinite limit. In this way, we may speak about uniformly continuous functions with values in $[-\infty,\infty]$. Moreover, $\arctan:[-\infty,\infty]\to [-\frac\pi2,\frac\pi2]$ is an order preserving homeomorphism, and hence $[-\infty,\infty]$ is order complete, and therefore a continuous function defined on a compact set attains maximum and minimum (possibly $\infty$ and $-\infty$).

By assumption any concave kernel function $K:\TT\to [-\infty,\infty)$ is (uniformly) continuous in  this extended sense.

	\begin{prp}\label{prop:cont0}
		For any concave kernel functions $K_0,\dots, K_n$ the sum of translates function
	\begin{align*}
	F&: \TT^n\times \TT\to [-\infty,\infty)	
	\end{align*}
    defined in \eqref{def:sumoftranslates}	is uniformly continuous (in the above defined extended sense).
	\end{prp}
	\begin{proof}
	Continuity of $F$ (in the extended sense) is trivial since
the $K_j$'s
are  continuous in the sense described in the preceding paragraph.
Also, they do not take the value $\infty$. Since $\TT^n\times \TT$ is 
compact, 
uniform continuity follows.
\end{proof}

\medskip\noindent
Next, a node system $\yy$ determines $n+1$ arcs on $\TT$, and we would like to look at the continuity (in some sense) of the arcs as
a function
of the nodes. The technical difficulties are that the nodes may coincide and they may jump over $0\equiv2\pi$. 
Note that passing from one simplex to another one may  cause jumps in the definitions of the arcs $I_j(\yy)$, entailing jumps also in the definition of the corresponding $m_j$. 
{\piros Indeed, at points $\yy \in \TT^n\setminus X$, on the (common) boundary of some simplexes, the change of the arcs $I_j$ may be discontinuous. 
E.g., when $y_j$ and $y_k$ changes place (ordering changes between them, e.g., from $y_\ell<y_j \leq y_k<y_r$ to $y_\ell< y_k < y_j <y_r$), 
then the three arcs between these points will change from the system $I_\ell=[y_\ell,y_j], I_j=[y_j,y_k], I_k=[y_k,y_r]$ 
to the system $I_\ell=[y_\ell,y_k], I_k=[y_k,y_j], I_j=[y_j,y_r]$. 
This also means that the functions $m_j$ may be defined \emph{differently} on a boundary point $\yy \in \TT^n\setminus X$ 
depending on the simplex we use: the interpretation of the equality $y_j=y_k$ as part of the simplex with $y_j\leq y_k$ 
in general furnishes a different value of $m_j$ 
than the interpretation as  
part of the simplex with $y_k\leq y_j$ (when it becomes $\max_{t\in [y_j,y_r]} F(\yy,t)$).}

\medskip\noindent
 These problems can be overcome by the next
considerations.

\begin{remark}\label{rem:arcsystemcont} \label{rem:cut}
{\piros Let us fix any node system $\yy_0$, together with a small $0<\delta <\pi/(2n+2)$, then there exists an arc $I(\yy_0)$ among the ones determined by $\yy_0$, together with its center point $c=c(\yy_0)$ such that $|I(\yy_0)|> 2\delta$, so in a (uniform-) $\delta$-neighborhood
$U:=U(\yy_0,\delta):=\{ \xx\in \TT^n~:~ d_{\TT^n}(\xx,\yy_0)<\delta\}$
of $\yy_0\in\TT^n$, 
none of the nodes of the configurations can reach $c$. 
We {cut} the torus at $c$ and represent the points of the torus $\TT=\RR/2\pi\ZZ$ by the interval $[c,c+2\pi)\simeq [0,2\pi)$ and use the ordering of this interval.
Henceforth, such a cut---as well as the cutting point $c$---will be termed as an \emph{admissible cut}. Of course, the cut  depends on the fixed point $\yy_0$, but it will cause no confusion if this dependence is left out of the notation, as we did here.}

Moreover, for {\piros $\yy\in U$} and
$i=1,\dots,n$ we define
\begin{align*}	
\ell_i(\yy)&:=\min
\left\{
t\in [c,c+2\pi) ~:~ \#\{k:y_k\leq t\}\geq i
\right\},
\\
r_i(\yy)&:=
\sup\left\{
t\in [c,c+2\pi) ~:~ \#\{k:y_k\leq t\}\leq i
\right\},
\\
\widehat I_i(\yy)&:=[\ell_i(\yy),r_i(\yy)
],
\intertext{and we set}
\widehat I_0(\yy)&:=[c,\ell_1(\yy)]\cup [r_n(\yy),c+2\pi]=:[\ell_0(\yy),r_0(\yy)]\subseteq \TT \quad \text{(as an arc)}.
\end{align*}
Then $\widehat I_i(\yy)$ is the $i^{\text{th}}$ arc in this \emph{cut} of torus along $c$ corresponding to the node system $\yy$.
{\piros We immediately see the continuity of the mappings
\[
U\to \TT,\quad \yy\mapsto \ell_i(\yy)\in \TT\quad\text{and}\quad \yy\mapsto r_i(\yy)\in \TT
\]
at $\yy_0$ for each $i=0,\dots, n$.
 Obviously, the
 \emph{system of arcs}
 $\{ I_{\sigma,j}(\yy)~:~ j=0,\dots,n\}$ is the same as $\{ \widehat I_i(\yy)~:~ i=0,\dots,n\}$
 independently of $\sigma$.}
\end{remark}

	\begin{prp}\label{prop:cont}
		Let $K_0,\dots, K_n$ be any concave kernel functions, let $\yy_0\in \TT^n$ be a node system and let $c$ be an admissible cut
        (as in Remark \ref{rem:arcsystemcont}). Then for  $i=0,\dots,n$ the functions
		\begin{align*}
	\yy\mapsto		\widehat m_i(\yy):=\sup_{t\in \widehat I_i(\yy)}F(\yy, t)\in [-\infty,\infty]
		\end{align*}
	are continuous at $\yy_0$ (in the extended sense).
	\end{prp}
\begin{proof}
	By Proposition \ref{prop:cont0} the function $\arctan\circ F:\TT^n \times \TT\to [-\frac\pi2,\frac\pi2]$ is  continuous at $\{\yy_0\}\times \TT$. Hence  $f_i(\yy):=\max_{t\in \widehat I_i(\yy)}\arctan\circ F(\yy,t)$ (and thus also $\widehat m_i=\tan\circ f_i$) is continuous, since $\ell_i$ and $r_i$ are continuous (see Remark \ref{rem:arcsystemcont}).
\end{proof}

The continuity of $\widehat m_i$ for fixed $i$ involves the cut of the torus at $c$. However, if we consider the system $\{m_0,\dots, m_n\}=\{\widehat m_0,\dots, \widehat m_n\}$ the dependence on the cut of the torus  can be cured. For $\xx\in \TT^{n+1}$ define
\begin{align*}
T_i(\xx)&:=\min\{t\in [c,c+2\pi)~:~ \exists k_0,\dots,k_i~{\text{s.t.}} ~ x_{k_0},\dots,x_{k_i}\leq t\}\quad \text{($i=0,\dots,n$)}
\intertext{and}
T(\xx)&:=(T_0(\xx),\dots, T_n(\xx)).
\end{align*}
The mapping $T$ arranges the coordinates of $\xx$
non-decreasingly
and it is easy to see that
$T:\RR^{n+1}\rightarrow \RR^{n+1}$ is continuous.

\begin{corollary}\label{cor:systemmjcont}
	For any concave kernel functions $K_0,\dots, K_n$ the mapping
	\[
	\TT^n\ni \yy\mapsto T(m_0(\yy),\dots,m_n(\yy))
	\]
	is (uniformly) continuous (in the extended sense).
\end{corollary}
\begin{proof}
	We have $T(m_0(\yy),\dots,m_n(\yy))=T(\widehat m_0(\yy),\dots,\widehat m_n(\yy))$ for any $\yy\in\TT$, while $\yy\mapsto (\widehat m_0(\yy),\dots,\widehat m_n(\yy))$ is continuous at any given point $\yy_0\in \TT^n$  and  for any {\piros fixed} admissible cut. But the left-hand term here  does not depend on the cut, so the assertion is proved.
\end{proof}

\begin{corollary}\label{cor:mmMMcont}
		Let  $K_0,\dots, K_n$ be any concave kernel functions.
 The functions $\MM: \TT^n\rightarrow (-\infty,\infty)$
and $\mm:\TT^n\rightarrow [-\infty,\infty)$ are continuous (in the extended sense).
\end{corollary}
\begin{proof}
The assertion immediately follows from Proposition \ref{prop:cont} and Corollary \ref{cor:mMfin} (a) and (b).
\end{proof}
\begin{corollary}\label{cor:mjsimplexcont}
		Let  $K_0,\dots, K_n$ be any concave kernel functions, and let  $S:=S_\sigma$ be a simplex.
For $j=0,\dots,n$ the functions
		\[
		m_j: \overline{S} \to [-\infty,\infty]
		\]
		are (uniformly) continuous (in the extended sense).
\end{corollary}
\begin{proof}
	Let $\yy_0\in \overline{S}$, then there is an admissible cut at some $c$ (cf.{} Remark \ref{rem:arcsystemcont}) and there is some $i$, such that we have $m_j(\yy)=\widehat m_i(\yy)$ for all $\yy$ in a small neighborhood $U$ of $\yy_0$ in $S$. So the continuity follows from Proposition \ref{prop:cont}.
\end{proof}

\begin{remark}\label{rem:zjdef}
Suppose that the kernel functions are concave and at least one of them is strictly concave. {\piros For a fixed simplex $S_\sigma$ and $\yy\in S_\sigma$} also
$F(\yy,\cdot)$ is strictly concave on the interior of each arc
$I_j(\yy)$ and continuous on $I_j(\yy)$ (in the extended sense), so there is a
\emph{unique} $z_j(\yy)\in I_j(\yy)$ with
\[
m_j(\yy)=F(\yy,z_{j}(\yy))
\]
(this being trivially true if $I_j(\yy)$ is degenerate).
\end{remark}

If condition \eqref{eq:kernsing} holds, then it is evident that $z_j(\yy)$ belongs to the interior of $I_j(\yy)$ (if this latter is non-empty). {\piros However, we can obtain the same even under the weaker assumption \eqref{eq:kernsing'}, for which purpose we state the next lemma.}

\begin{lemma}\label{lem:new}
	Suppose that $K_0,\dots,K_n$ are concave kernel functions, with at least one of them
	strictly concave.
    \begin{abc}
	\item
	If condition \eqref{eq:kernsing'p} holds for $K_j$, then for any $\yy\in \TT^n$ the sum of translates function $F(\yy,\cdot)$ is strictly increasing on $(y_j,y_j+\varepsilon)$ for some  $\varepsilon>0$.
		\item
	If condition
\eqref{eq:kernsing'm}
holds for $K_j$, then for any $\yy\in \TT^n$
the sum of translates function
$F(\yy,\cdot)$ is strictly decreasing on $(y_j-\varepsilon,y_j)$ for some  $\varepsilon>0$.

	\end{abc}
	\end{lemma}
\begin{proof}(a)
	 Obviously, in case $K_j(0)=-\infty$, we also have
	$F(\yy,y_j)=-\infty$ and the assertion follows trivially since $F(\yy,\cdot)$ is concave on an interval $(y_j,y_j+\varepsilon)$, $\varepsilon>0$.
	So we may assume $K_j(0)\in \RR$, in which case
	$F(\yy, \cdot)$ is finite, continuous and concave  on
	$[y_j,y_j+\varepsilon]$ for some $\varepsilon>0$.
	Then for the fixed $\yy$ and for the function $f=F(\yy,\cdot)$ we have for any fixed
	$t\in (y_j,y_j+\varepsilon)$ that
	\[
	D_+f(y_j) =\lim\limits_{s\downto y_j} \sum_{k=0}^n D_+K_k(s-y_k) \ge
	\sum_{k=0, k\ne j}^n D_+K_k(t-y_k) + \lim\limits_{s\downto y_j} D_+K_j(s-y_j) = \infty,
	\]
	since $D_+K_k(\cdot-y_k)$ is
    non-increasing
    by concavity.
	Therefore, choosing $\varepsilon$ even smaller, we find that
	$D_+F(\yy,\cdot) >0$ in the interval $(y_j,y_j+\varepsilon)$, which implies that
	$F(\yy,\cdot)$ is strictly increasing in this interval.

	\medskip\noindent (b) Under condition
	\eqref{eq:kernsing'm}
	the proof is similar for the interval $(y_j-\varepsilon,y_j)$.
\end{proof}

{\piros \begin{prp}\label{prop:zjinter}
Suppose that $K_0,\dots,K_n$ are concave kernel functions, with at least one of them
	strictly concave. Let $S_\sigma$ be a simplex and let  $\yy\in S_\sigma$ (so that $\sigma$ is fixed, and $I_0(\yy),\dots,I_j(\yy)$ are well-defined). 
\begin{abc}
\item For each $j=0,\dots,n $ there is 
unique maximum point $z_j(\yy)$ of $F(\yy,\cdot)$ in $I_j(\yy)$, i.e., $F(\yy,z_j(\yy))=m_j(\yy)$.
\item
If condition \eqref{eq:kernsing'p} holds for $K_j$, and  $I_j(\yy)=[y_j,y_r]$ is non-degenerate, then $z_j(\yy)\neq y_j$.
\item
If condition \eqref{eq:kernsing'm} holds for $K_j$, and  $I_\ell(\yy)=[y_\ell,y_j]$ is non-degenerate, then $z_\ell(\yy)\neq y_j$.

\item If condition \eqref{eq:kernsing'pm} holds for each $K_j$, $j=0,\dots,n$,
then $z_j(\yy)$ belongs to the interior of $I_j(\yy)$
whenever 
 $I_j(\yy)$ 
is non-degenerate.
\end{abc}
\end{prp}}
\begin{proof} (a) Uniqueness of a maximum point, i.e., the definition of $z_j(\yy)$ has been already discussed in Remark \ref{rem:zjdef}.

\medskip\noindent The assertions (b) and (c) follow from Lemma \ref{lem:new}
and they imply (d).
\end{proof}

For the next lemma we need that the function $z_j$ is
well-defined for each $j=0,\dots,n$, so we need $F(\yy,\cdot)$
to be strictly concave, in order to which it suffices if at
least one of the kernels is strictly concave.

{\piros \begin{lemma}\label{lem:zjcont} 
Suppose that $K_0,\dots,K_n$ are concave kernel functions with at least one of them
strictly concave.
{\lila \begin{abc}
\item Let $S=S_\sigma$ be a simplex.
  (Recall that, because of strict concavity,  the maximum point $z_j(\yy)$ of
$F(\yy,\cdot)$ in $I_j(\yy)$ is unique for every $j=0,\dots,n$.) For
each $j=0,\dots, n$  the mapping
	\[
	z_j:\overline{S}\to \TT,\quad \yy\mapsto z_j(\yy)
	\]
	is continuous. 
    \item For a given $\yy_0\in \TT^n$ and an admissible cut of the torus (cf.{} Remark \ref{rem:cut}) the mapping
	\[
	\yy\mapsto \widehat z_i(\yy)
	\]
	is continuous at $\yy_0$.
    \end{abc}}
\end{lemma}}
\begin{proof}
{\piros Let $\yy_n\in \overline{S}$ with $\yy_n\to \yy\in \overline{S}$.} Then, by Proposition \ref{prop:cont}, $m_j(\yy_n)\to m_j(\yy) \in [-\infty,\infty)$.
 Let $x\in \TT$ be any accumulation point of the sequence $z_j(\yy_n)$, and by passing to a subsequence assume $z_j(\yy_n)\to x$.

	By definition of $z_j$, we have $F(\yy_n,z_j(\yy_n))=m_j(\yy_n)\to m_j(\yy)$, and by continuity of $F$ also $F(\yy_n,z_j(\yy_n))\to F(\yy,{x})$, so $F(\yy,{x})=m_j(\yy)$. But we have already remarked that by strict concavity there is a \emph{unique} point, where $F(\yy,\cdot)$ can attain its maximum on $I_j$ (this provided us the definition of $z_j(\yy)$ as a uniquely defined point in $I_j$). Thus we conclude $z_j(\yy)={x}$. 
    
    The second assertion follows from this in an
obvious way.
\end{proof}

\begin{prp}\label{prop:exists}
For a simplex $S=S_\sigma$ we always have $M(S)=M(\overline{S})$ and $m(S)=m(\overline{S})$. Furthermore, both minimax problems \eqref{eq:Sminmax} and \eqref{eq:Smaxmin} have finite extremal values, and both have an extremal node system, i.e., there are $\wS, \ws\in \overline{S}$ such that
	\begin{align*}
	&\MM(\wS)=M(S):=\inf_{\yy\in S}\MM(\yy)=M(\overline{S})=\min_{\yy\in \overline{S}}\MM(\yy)\in \RR,\\
	&\mm(\ws)=m(S):=\sup_{\yy\in {S}}\mm(\yy)=m(\overline{S})=\max_{\yy\in \overline{S}}\mm(\yy)\in \RR.
	\end{align*}
	\end{prp}
	\begin{proof}
	By Proposition \ref{prop:cont} the functions $\mm$ and $\MM$  are continuous (in the extended sense), whence we conclude  $m(S)=m(\overline{S})$ and $M(S)=M(\overline{S})$. Since $\overline{S}$ is compact, the function $\mm$ has a maximum on $\overline{S}$, i.e., \eqref{eq:Sminmax} has an extremal node system $\ws$. Similarly, $\MM$ has a minimum, meaning that \eqref{eq:Smaxmin} has an extremal node system $\wS$.

	Both of  these extremal values, however, must be \emph{finite}, according to Corollary \ref{cor:mMfin}.
	\end{proof}

 As a consequence, we obtain the following.
	\begin{corollary} \label{cor:minmaxexist}
		Both minimax problems \eqref{eq:minmax} and \eqref{eq:maxmin} have an extremal node system.
	\end{corollary}
	To decide whether the extremal node systems belong to $S$ or to the boundary $\partial S$ is the subject of the next sections.

	\section{Approximation of kernels}\label{sec:approx}
	In this section 
    we consider sequences $K_j^{(k)}$ of kernel functions converging to $K_j$ as $k\to \infty$ for each $j=0,\dots,n$ (in some sense or another). The corresponding values of local maxima and related quantities will be denoted by $m_j^{(k)}(\xx)$, $\mm^{(k)}(\xx)$, $\MM^{(k)}(\xx)$, $m^{(k)}(S)$, $M^{(k)}(S)$, and we study the limit behavior of these as $k\to \infty$.  Of course, one has here a number of notions of convergence for the kernels, and we start with the easiest ones.

	\medskip
	Let $\Omega$ be a compact space and let  $f_n,f\in \Ce(\Omega;\RRR)$ (the set of continuous functions with values in $\RRR$). We say that $f_n\to f$ \emph{uniformly} (in the extended sense, e.s.{} for short) if $\arctan f_n\to \arctan f$ uniformly in the ordinary sense (as real valued functions).
	We say that $f_n\to f$ \emph{strongly uniformly} if for all $\varepsilon>0$ there is $n_0\in \NN$ such that
	\[
	f(x)-\varepsilon\leq f_n(x)\leq f(x)+\varepsilon\quad\text{for every $x\in K$ and $n\geq n_0$}.
	\]
	\begin{lemma}\label{lem:unifconv1}
		Let $f,f_n\in \Ce(\Omega;\bar\RR)$ be uniformly bounded from above. We then have $f_n\to f$ uniformly (e.s.) if and only if for each $R>0,\eta>0$ there is $n_0\in\NN$ such that for all $x\in \Omega$ and all $n\geq n_0$
		\begin{align}\label{eq:unifconv1}
		f_n(x)<-R+\eta\quad &\text{whenever $f(x)<-R$ and }\\
		\notag f(x)-\eta\leq f_n(x)\leq f(x)+\eta\quad & \text{whenever $f(x)\geq -R$.}
		\end{align}
		\end{lemma}
		\begin{proof}
        {\lila Let $C\geq 1$ be such that $f,f_n\leq C$ for each $n\in \NN$.}
			Suppose first that $f_n\to f$ uniformly (e.s.), and let $\eta>0$, $R>0$ be given. The set $L:=\arctan[-R-1,C+1]$ is compact in $(-\frac\pi2,\frac\pi2)$, and $\tan$ is uniformly continuous thereon. Therefore there is $\varepsilon\in(0,1]$ sufficiently small such 
that
\[\tan(s)-\eta\leq \tan (t)\leq \tan(s)+\eta\] 
whenever $|s-t|\leq \varepsilon$, $s\in\arctan[-R,C]$, {\lila in particular $\tan(\arctan(-R)+\varepsilon)\leq -R+\eta$. }
Let $n_0\in\NN$ be so large that 
$\arctan f(x)-\varepsilon\leq \arctan f_n(x)\leq \arctan f(x)+\varepsilon$ holds for every $n\geq n_0.$
			Apply the $\tan$ function to this inequality to obtain  that $f(x)-\eta\leq f_n(x)\leq f(x)+\eta$ for $x\in \Omega$ with $f(x)\in [-R,C]$, and
		 \[f_n(x)\leq \tan(\arctan f(x)+\varepsilon)<\tan(\arctan(-R)+\varepsilon)<-R+\eta\]
		  for $x\in \Omega$ with $f(x)<-R$.
	
		 \medskip\noindent Suppose now that condition \eqref{eq:unifconv1} involving $\eta$ and $R$ is satisfied, and let $\varepsilon>0$ be arbitrary. Take $R>0$ so large that $\arctan(t)<-\frac\pi 2+\varepsilon$ whenever $t<-R+1$. For $\varepsilon>0$ take $1>\eta>0$ according to the uniform continuity of $\arctan$. By assumption there is $n_0\in \NN$ such that for all $n\geq n_0$ we have \eqref{eq:unifconv1}. Let $x\in \Omega$ be arbitrary. If $f(x)<-R$, then
		 \begin{align*}
		 \arctan f(x)-\varepsilon<-\frac\pi 2&\leq \arctan f_n(x)\\
		 &\leq \arctan(-R+\eta)<-\frac\pi 2+\varepsilon<\arctan f(x)+\varepsilon.
		 \end{align*}
		 On the other hand, if $f(x)\geq -R$, then by the choice of $\eta$ and by the second part of \eqref{eq:unifconv1} we immediately obtain
		 \begin{equation*}
		 \arctan f(x)-\varepsilon< \arctan f_n(x)\leq \arctan f(x)+\varepsilon.
		 \end{equation*}
			\end{proof}
			The previous lemma has an obvious version for sequences that are not uniformly bounded from above. This is, however a bit more technical and will not be needed.	It is now also clear that strong uniform convergence implies uniform convergence. Furthermore, the next assertions follow immediately from the corresponding classical results about real-valued functions.
	\begin{lemma}\label{lem:unifconv}
For $n\in \NN$ let $f_n,g_n,f,g\in \Ce(\Omega;\RRR)$.
	\begin{abc}
		\item If $f_n,g_n\leq C<\infty$ and $f_n\to f$ and $g_n\to g$ uniformly (e.s.), then $f_n+g_n\to f+g$ uniformly (e.s.).
	\item If $f_n\downto f$ pointwise, i.e., if $f_n(x)\to f(x)$
    non-increasingly
    for each $x\in \Omega$, then $f_n\to f$
uniformly (e.s.).
	\item If $f_n\to f$ uniformly (e.s.), then $\sup f_n\to \sup f$ in $[-\infty,\infty]$.
	\end{abc}
	\end{lemma}
	\begin{proof}
		(a) The proof can be based on Lemma \ref{lem:unifconv1}.
	
		\medskip\noindent (b) This is a consequence of Dini's theorem.
	
		\medskip\noindent (c) Follows from standard properties of $\arctan$ and $\tan$, and from the corresponding result for  real-valued functions.
		\end{proof}

	\begin{prp}\label{prp:unifconvm}
		Suppose the sequence of kernel functions $K_j^{(k)}\to K_j$ uniformly (e.s.) for $k\to \infty$ and $j=0,1, \dots, n$. Then for each simplex $S:=S_\sigma$ we have
		that $m^{(k)}_j\to m_j$ uniformly (e.s.) on $\Bar{S}$ ($j=0,1, \dots, n$). As a consequence, $m^{(k)}(S)\to m(S)$ and $M^{(k)}(S)\to M(S)$ as $k\to \infty$.
		\end{prp}
		\begin{proof}
			The functions $F^{(k)}(\xx,t)=\sum_{j=0}^n K^{(k)}_j(t-x_j)$ are continuous on $\TT^{n+1}$ and converge uniformly (e.s.) to $F(\xx,t)=\sum_{j=0}^n K_j(t-x_j)$ by (a) of Lemma \ref{lem:unifconv}. So that we can apply part (c) of the same lemma, to obtain the assertion.
			\end{proof}

	We now relax the notion of convergence of the kernel functions, but, contrary to the above, we shall make essential use of the concavity of kernel functions. We say that a sequence of functions over a set $\Omega$ converges \emph{locally uniformly}, if this sequence of functions converges uniformly on each compact subset of $\Omega$.

\begin{remark}
	By using the facts that pointwise convergence of continuous monotonic functions, and pointwise convergence of concave functions, with a continuous limit function, is actually uniform (on compact intervals, see, e.g., \cite[Problems 9.4.6, 9.9.1]{Thomson} and \cite{Guberman}), it is not hard to see that if the kernel functions  $K_n$ converge to $K$ pointwise  on $[0,2\pi]$, then they even converge uniformly in the extended sense.
	\end{remark}

Recall the definitions of  $d_{\TT}(x,y)$ and $d_{\TT^m}(\xx,\yy)$ from \eqref{eq:tmetric} and \eqref{eq:tnmetric}.
	Define the compact set
	\[
	D:=\bigl\{ (\xx,t)~:~ \exists i\in\{0,1,\dots,n\}, \text{ such that }
	t=x_i\bigr\}=\bigcup_{i=0}^n \bigl\{ (\xx,t)~:~ t=x_i\bigr\}\subseteq \TT^{n+1}.
	\]

	\begin{lemma}\label{l:approx}
		Suppose the sequence of kernel functions $\Kkj$ converges to the kernel function $K_j$ locally
		uniformly on $(0,2\pi)$. Then $F^{(k)}(\xx,t)\to F(\xx,t)$ locally uniformly on $\TT^{n+1}\setminus D$, i.e., for every compact subset $H\subseteq \TT^{n+1}\setminus D $
		 one has $F^{(k)}(\xx,t)\to F(\xx,t)$ uniformly on $H$ as $k\to\infty$.
	\end{lemma}

	Note that in general $F$ can attain $-\infty$, and that
	convergence in $0$ of the kernels is not postulated.

	\begin{proof} Because of
	compactness of $H$ and $D$  we have
$0<\rho:=d_{\TT^{n+1}}(H,D)$.	

Take $0<\delta<\rho$ arbitrarily and consider for any
	$(\xx,t)\in H$ the defining expression
	$F^{(k)}(\xx,t):=\sum_{i=0}^n \Kki(t-x_i)$.
For points of $H$ we have $|t-x_i|\geq\min\left(|t-x_i|,2\pi-|t-x_i|\right)=d_\TT(t,x_i)=d_{\TT^{n+1}}\left((\xx,t),(\xx,x_i)\right)\geq \rho>\delta$.
    In other words, $\Phi_i(H)\subset [\delta,2\pi-\delta]$
	for $i=0,1,\dots,n$, where $\Phi_i(\xx,t):=t-x_i$ is
	continuous---hence also uniformly continuous---on the whole
	$\TT^{n+1}$.

	As the locally uniform convergence of $\Kki$ (to $K_i$) on
	$(0,2\pi)$ entails uniform convergence on $[\delta,2\pi-\delta]$, we
	have uniform convergence of $f_i^{(k)}:=\Kki \circ \Phi_i$ on the compact set
	$H$ (to the function $K_i\circ \Phi_i$). It follows that
	$F^{(k)}=\sum_{i=0}^n f_i^{(k)}$ converges uniformly (to
	$F=\sum_{i=0}^n f_i$) on $H$, whence the assertion follows.
	\end{proof}

	\begin{lemma}\label{lem:Kconvuv}
		Let $K:(0,2\pi)\to \RR$ be {\lila any} concave function (so $K$ has limits, possibly $-\infty$, at $0$ and $2\pi$, {\lila defining $K(0)$ and $K(2\pi)$}).
		For each $u,v\in[0,1]$
	 we have
		\begin{align*}
		K(u) & \le K(u+v) - v\bigl(K(\pi+1/2)-K(\pi-1/2)\bigr), \\
		K(2\pi-u) &\le K(2\pi-u-v) + v\bigl(K(\pi+1/2)-K(\pi-1/2)\bigr).
		\end{align*}
	\end{lemma}
	\begin{proof}It is sufficient to prove the statement for $u>0$ only, as the case $u=0$ follows from that by
 passing to the limit.

		Also we may suppose $v>0$ otherwise the inequalities are trivial.
	By concavity of $K$ for any system of four points
		$0<a<b<c<d<2\pi$ we clearly have the inequality
		\[
		\frac{K(b)-K(a)}{b-a}\geq \frac{K(d)-K(c)}{d-c}
		\]
see e.g. \cite{WayneVarberg}, p. 2, formula (2).
		 Specifying
$a:=u$, $b:=u+v \le 2 < c:=\pi-1/2$ and $d:=\pi+1/2$		
        yields the first inequality, while for
$a:=\pi-1/2$, $b:=\pi+1/2 < 4 < c:=2\pi-u-v$ and $d:=2\pi-u$,	
        we obtain the second one.
			\end{proof}
	
	\begin{thm}\label{thm:approx} Suppose that the kernels are such that
for all $\xx\in \TT^n$ and $z\in\TT$ with $F(\xx,z)=\MM(\xx)$
one has $z\ne x_j$, $j=0,\dots,n$.
	If the sequence of kernel functions $\Kkj \to K_j$ locally
	uniformly on $(0,2\pi)$, then $\MM^{(k)}(\xx)\to \MM(\xx)$
	uniformly on $\TT^n$.
	\end{thm}
	\begin{proof} Let us define the set $H_0:=\{(\xx,z)~:~
	F(\xx,z)=\MM(\xx)\}\subset \TT^{n+1}$, which is obviously closed by virtue of the continuity of the occurring functions.
	By assumption $H_0\subseteq \TT^{n+1}\setminus D$, so the condition of Lemma \ref{l:approx} is satisfied,
	hence $F^{(k)}\to F$ uniformly on $H_0$.

	Let now $\xx\in \TT^n$ be arbitrary, and take any $z\in \TT$ such
	that $F(\xx,z)=\MM(\xx)$ (such a $z$ exists by compactness and continuity). Now, $\MM^{(k)}(\xx)\ge F^{(k)}(\xx,z)
	> F(\xx,z)-\varepsilon = \MM(\xx)-\varepsilon$ whenever $k>k_0(\varepsilon)$, hence
	$\liminf_{k\to \infty} \MM^{(k)}(\xx)\ge \MM(\xx)$ is clear,
	moreover, according to the above, this holds uniformly on
	$\TT^n$, as $\MM^{(k)}(\xx)> \MM(\xx)-\varepsilon$ for each $\xx\in \TT^n$
    whenever $k>k_0(\varepsilon)$.

	\medskip\noindent
	It remains to see that, given $\xx\in\TT^n$ and
    $\varepsilon>0$, there exists
$k_1(\varepsilon)$
such that $\mkx <\MM(\xx) +\varepsilon$ for all
$k>k_1(\varepsilon)$.
	 Let us define the constant
	\[
		C:= \max_{j=0,1,\dots,n} \max_{k\in\NN}
	|K^{(k)}_j (\pi+1/2) - K^{(k)}_j (\pi-1/2) |.
	\]
	The inner expression is indeed a finite maximum,
	as $K^{(k)}_j (\pi\pm
	1/2)\to K_j(\pi \pm 1/2)$ for $k\to\infty$. By Lemma \ref{lem:Kconvuv} for
	all $u, v \in [0,1]$
	\begin{equation}\label{eq:KkjuvC}
	\Kkj(u) \le \Kkj(u+v) + C v, \quad
	\Kkj (2\pi-u) \le \Kkj(2\pi-u-v)+Cv.
	\end{equation}
	For the given $\varepsilon>0$ choose $\delta\in (0,1/2)$ such that $\MM(\yy)\leq\MM(\xx)+\tfrac{\varepsilon}{3}$ holds for all $\yy$
	 with $d_{\TT^n}(\xx,\yy)<\delta$ (use Corollary \ref{cor:mmMMcont}, the uniform continuity of $\MM:\TT^n\to\RR$). Fix moreover
 $0<h<\min\{\delta/2,\varepsilon/(3C(n+1))\}$
and define
	 \[H:=\bigl\{(\yy,w)\in\TT^{n+1}~:~ d_\TT(y_i,w)\ge h~(i=0,1,\dots,n)\bigr\}.\]

	\medskip\noindent For
	an arbitrarily given point $(\xx,z)\in \TT^{n+1}$ we construct another one $(\yy,w)\in \TT^{n+1}$, which we will call
	``approximating point'', in two steps as follows. First, we shift them (even $x_0$ which was assumed to be $0$ all the time), and then correct them. So we set for $i=0,1,\dots,n$
	\[
	x'_i:=\begin{cases}
	x_i \qquad &\text{if }  \quad d_{\TT}(x_i,z)\ge h, \\
	x_i\pm h \qquad &\text{if } \quad d_{\TT}(x_i,z)\le h,
	\end{cases}
	\]
	where we add $h$ or $-h$ such that $d_{\TT}(x_i\pm h,z)\geq h$. Then we set
	 $y_i:=x'_i-x'_0$ ($i=0,1,\dots,n$) and $w:=z-x'_0$. This new approximating point $(\yy,w)$ has the following properties:
	\begin{equation}
	d_{\TT^n}(\xx,\yy)=\max_{i=1,\dots,n} d_{\TT}(x_i,y_i) \le 2h<\delta, \qquad d_\TT(z,w) \le h<\delta.
	\end{equation}
	Moreover, we have $(\yy,w)\in H$, since $d_\TT(y_i,w)=d_\TT(x_i',z_i)\geq h$ for $i=0,1,\dots,n$.

	\medskip\noindent By construction of $(\yy,w)$ we have
	\begin{align}\label{ea:approxpoint}
	y_i-w &=
    x_i-z \qquad \qquad
    &\text{if } \quad d_{\TT}(x_i,z)\ge h, \notag \\
	y_i-w &=
    x_i-z\pm h \qquad
    &\text{if } \quad d_{\TT}(x_i,z)\le h.
	\end{align}
	So by using both
	inequalities
     in \eqref{eq:KkjuvC}
    we conclude
	\begin{align*}
	\Kkj(x_j-z) \le \Kkj(y_j-w) +
    C h  \qquad (j=0,1,\dots,n),
    \end{align*}
	providing us
	\begin{equation*}
	\Fk(\xx,z) = \sum_{j=0}^n \Kkj(x_j-z) \le \sum_{j=0}^n (\Kkj(y_j-w) +
    C h  )
	= \Fk(\yy,w) +
    (n+1)Ch.
	\end{equation*}

	\medskip  Now, for given $\xx\in \TT^n$ let  $z_k\in \TT$ be any point with
	$\Fk(\xx,z_k)=\MMk(\xx)$, and let $(\yy^{(k)},w_k)\in H$ be the corresponding approximating point.
	So that we have
	\begin{equation}\label{eq:mkxupper}
	\mkx=\Fk(\xx,z_k)\le \Fk(\yy^{(k)},w_k) +
    (n+1)Ch.
	\end{equation}
	Since $(\yy^{(k)},w_k)\in H\subseteq \TT^n\setminus D$ we can
	invoke Lemma \ref{l:approx} to get $\Fk\to F$ uniformly on
	$H$. Therefore, for the given $\varepsilon>0$ there exists $k_1(\varepsilon)$
	with
	\begin{equation*}
	\Fk(\yy^{(k)},w_k) \le \max \bigl\{ F(\yy,w) ~:~
	(\yy,w)\in H, d_{\TT^n}(\xx,\yy) \le \delta, d_{\TT}(z,w)\le \delta \bigr\} + \tfrac{\varepsilon}{3}
	\end{equation*}
	for all $k\geq k_1(\varepsilon)$. Extending further the maximum on the right-hand side to
	arbitrary $w\in \TT$ we are led to
	\begin{equation}\label{eq:Fkwklimit}
	\Fk(\yy^{(k)},w_k) \le \max \bigl\{ \MM(\yy) ~:~
 d_{\TT^n}(\xx,\yy) \le \delta\bigr\}+ \tfrac{\varepsilon}{3} \qquad (k>k_1(\varepsilon)).
	\end{equation}
	From \eqref{eq:mkxupper},
	\eqref{eq:Fkwklimit} and by the choices of $h,\delta>0$ we conclude
	\begin{equation*}\label{eq:mkfinalesti}
	\mkx \le \Fk(\yy^{(k)},w_k)  +(n+1)Ch
    \le (\MM(\xx)+\tfrac{\varepsilon}{3}) +\tfrac{\varepsilon}{3}
+(n+1)Ch
    < \MM(\xx)+\varepsilon
	\end{equation*}
	for all $k>k_1(\varepsilon)$. So that we get that uniformly on $\TT^n$
	$\limsup_{k\to\infty} \mkx \le \MM(\xx)$ holds.

	Since $k_1(\varepsilon)$ does not depend on $\xx$, by using also the first part we obtain
	$\lim_{k\to\infty} \mkx = \MM(\xx)$ uniformly on $\TT^n$.
	\end{proof}
	
\section{Elementary properties}
\label{sec:elemprops}

In this section we record some elementary properties of the function $m_j$ that are useful in the study of minimax and maximin problems and constitute also a substantial part of the abstract framework of \cite{shi}.
Moreover, our aim is
to reveal the structural connections between these {\lila properties}.
\begin{prp}\label{prop:bdry}
	Suppose that the kernels $K_0,\dots,K_n$ satisfy \eqref{eq:kernsing}.
Let $S=S_\sigma$ be a simplex. Then
\begin{equation}
	\label{eq:bdry}
\lim_{\yy\to\partial S\atop{\piros \yy\in S}}  \max_{k=0,\dots,n-1}\bigl|m_{\sigma(k)}(\yy)-m_{\sigma(k+ 1)}(\yy)\bigr|=\infty.
\end{equation}
\end{prp}
\begin{proof}
	Without loss of generality we may suppose that $\sigma=\id$, i.e., $\sigma(k)=k$.
Let $\yy^{(i)}\in S$ be convergent to some $\yy^{(0)}\in\partial S$ as $i\to\infty$. This means that some arcs determined by the nodes $\yy^{(i)}$ and
$y_0=0\equiv 2\pi$ shrink to a singleton. On any such  arc $I_j(\yy^{(i)})$ we obviously
have, with the help of \eqref{eq:kernsing},
\begin{equation*}
m_j\bigl(\yy^{(i)}\bigr) \to -\infty \quad\mbox{as $i\to\infty$}.
\end{equation*}
Of course, there is at least one such arc, say with index $j_0$, that has a neighboring arc with index $j_0\pm 1$ which is not shrinking to a singleton as $i\to\infty$.
This means
\begin{equation*}
\left|m_{j_0}\bigl(\yy^{(i)}\bigr)-m_{j_0\pm 1}\bigl(\yy^{(i)}\bigr)\right|\to \infty\quad  \mbox{as } i\rightarrow\infty,
\end{equation*}
and the proof is complete.
\end{proof}

The properties introduced below have nothing to do with the conditions we pose on the kernel functions $K_0,\dots, K_n$ (concavity and some type of singularity at $0$ and $2\pi$), so we can formulate them in whole generality. (Note that $m_j$, in contrast to $z_j$, is well-defined even if the kernels are not strictly concave).

\begin{definition} Let $S=S_\sigma$ be a simplex.
\begin{abc}
\item \textbf{Jacobi Property:}\\
The functions $m_0,\dots,m_n$ are in $\Ce^1(S)$ and
\begin{align*}
\det \Bigl(\partial_i m_{\sigma(j)}\Bigr)_{i=1,j=0,j\neq k}^{n,n}\neq 0\quad\mbox{for each $k\in \{0,\dots,n\}$.}
\end{align*}
\item \textbf{Difference Jacobi Property:}\\
The functions $m_0,\dots,m_n$ belong to $\Ce^1(S)$  and
\begin{align*}
\det \Bigl(\partial_i (m_{\sigma(j)}-m_{\sigma(j+1)})\Bigr)_{i=1,j=0}^{n,n-1}\neq 0.
\end{align*}
\end{abc}
\end{definition}

\begin{remark}\label{rem:shi1}
	Shi \cite{shi} proved that under the condition \eqref{eq:bdry} (which is now a consequence of the assumption \eqref{eq:kernsing}) the  Jacobi Property implies the Difference Jacobi Property.
\end{remark}

\begin{definition}\label{def:equi} Let $S=S_\sigma$ be a simplex.
\begin{abc}
\item \textbf{Equioscillation Property:}\\
There exists an \textbf{equioscillation point} $\yy\in S$, i.e.,
\[
\MM(\yy)=\mm(\yy)=m_0(\yy)=m_1(\yy)=\cdots=m_n(\yy).\]
\item \textbf{(Lower) Weak Equioscillation Property:}\\
There exists a \textbf{weak equioscillation point} $\yy\in \overline{S}$, i.e.,
\[
m_j(\yy)\begin{cases} =\MM(\yy), \quad &{\text{if}} ~ {\piros I_j(\yy)}~\text {is non-degenerate}, \\
<\MM(\yy), \quad &{\text{if}} ~ {\piros I_j(\yy)}~\text{is degenerate}. \\\end{cases}
\]
\end{abc}
\end{definition}

 \begin{remark}\label{rem:equioscillation}
 	For a given $S=S_\sigma$ the Equioscillation Property implies the inequality $M(S)\leq m(S)$.
  \end{remark}
\begin{proof}
	Let $\yy\in S$ be an equioscillation point. Then for this particular point
	\[\MM(\yy)=\max_{j=0,\dots,n}m_j(\yy)=\min_{j=0,\dots,n}m_j(\yy)=\mm(\yy),\] hence
	\[
	M(S)\leq \MM(\yy)=\mm(\yy)\leq m(S).
	\]
\end{proof}

\begin{prp}\label{rem:sandwich}
	 Given a simplex $S=S_\sigma$ the following are equivalent:
\begin{iiv}
\item $M(S)\geq m(S)$.
\item For every $\xx\in S$ one has $\mm(\xx)=\min_{j=0,\dots,n} m_j(\xx)\leq M(S)$.
\item For every $\yy\in S$ one has $\MM(\yy)=\max_{j=0,\dots,n} m_j(\yy)\geq m(S)$.
\item  There exists a value $\mu\in \RR$ such that for each $\yy\in S$
\begin{equation*}
\MM(\yy)=\max_{j=0,\dots,n} m_j(\yy)\geq \mu \geq \mm(\yy)=\min_{j=0,\dots,n} m_j(\yy).
\end{equation*}
\end{iiv}
\end{prp}

\begin{proof} Recalling the inequalities
\begin{equation*}
\MM(\yy)=\max_{j=0,\dots,n} m_j(\yy)\geq M(S)=\inf_S \MM,\quad \sup_S \mm=m(S)\geq \mm(\xx)=\min_{j=0,\dots,n} m_j(\xx)
\end{equation*}
being true for each $\xx,\yy\in S$, the equivalence of (i), (ii) and (iii) is obvious. Suppose (i) and take $\mu\in[m(S),M(S)]$. Then  (iv) is evident.
From  (iv) assertion  (i) follows trivially.
\end{proof}

\begin{definition}\label{def:sandwich} Let $S=S_\sigma$ be a simplex. We say that the
 \textbf{Sandwich Property} is satisfied if any of the equivalent  assertions in Proposition \ref{rem:sandwich} holds true, i.e., if
for each $\xx,\yy\in S$
\begin{equation*}
\max_{j=0,\dots,n} m_j(\yy)=\MM(\yy)\geq \mm(\xx)=\min_{j=0,\dots,n} m_j(\xx).
\end{equation*}
\end{definition}

\begin{remark}\label{rem:fullsandwich}
	For given $S=S_\sigma$ the Equioscillation Property and the Sandwich Property together imply that $M(S)=m(S)$.
 \end{remark}

\begin{remark}
\label{rem:sandwichname}
The above are fundamental properties in interpolation theory, and thus have been extensively investigated.
First, for the Lagrange interpolation on $n+1$ nodes in $[-1,1]$ the maximum norm of the Lebesgue function is minimal if and only if all its local maxima are equal.
This equioscillation property was conjectured by Bernstein \cite{Bernstein} and proved by Kilgore \cite{kilgore2}, using also a lemma (Lemma 10 in the paper
\cite{kilgore2}) whose proof, in some extent, was based on direct input from de Boor and Pinkus \cite{deboor}.
Second, the property that the minimum of the local maxima is always below this equioscillation value was conjectured by Erd\H os in \cite{Erdos}, and proved in the paper \cite{deboor} of de Boor and Pinkus,  which appeared in the same issue as the article of Kilgore \cite{kilgore2}, and which is based very much on the analysis of Kilgore.
This latter property is just an equivalent formulation of the Sandwich Property, see Proposition \ref{rem:sandwich}.
For more details on the history of these prominent  questions of interpolation theory see in particular \cite{kilgore2}.
The name ``Sandwich Property'' seems to have appeared first in \cite{SzV}, see p.{} 96.
\end{remark}

\begin{definition}\label{def:majorization} {\piros Let $S=S_\sigma$ be a simplex and let $\xx,\yy\in \overline{S}$. We say that $\xx$ \emph{majorizes} (or \emph{strictly majorizes}) $\yy$---and $\yy$ \emph{minorizes} (or \emph{strictly minorizes}) $\xx$---if $m_j(\xx)\geq m_j(\yy)$ (or if $m_j(\xx)>m_j(\yy)$) for all $j=0,\dots,n$.
 We define the following properties on $S$.}
\begin{abc}
\item \textbf{Local (Strict) Comparison Property at $\zz$:} \\
There exists $\delta>0$ such that if $\xx,\yy\in \Ball(\zz,\delta)$
and $\xx$ (strictly) majorizes $\yy$, then $\xx=\yy$. In other words, there are no two different $\xx\ne \yy \in \Ball(\zz,\delta)$ with $\xx$ (strictly) majorizing $\yy$.
\item \textbf{Local (Strict) non-Majorization Property at $\yy$:}\\
There exists $\delta>0$ such that there is no $\xx\in (S\cap\Ball(\yy,\delta))\setminus \{\yy\}$ which (strictly) majorizes $\yy$.
\item \textbf{Local (Strict) non-Minorization Property at $\yy$:}\\
There exists $\delta>0$ such that there is no $\xx\in (S\cap\Ball(\yy,\delta))\setminus \{\yy\}$ which (strictly) minorizes $\yy$.
\end{abc}
Further, we will pick the following special cases as important.
\begin{ABC}
	\item \textbf{(Strict) Comparison Property on $S$:}\\
	If $\xx,\:\yy\in S$ and $\xx$ (strictly) majorizes $\yy$, then $\xx=\yy$. In other words, there exists no two different $\xx\ne \yy \in S$ with $\xx$ (strictly) majorizing $\yy$.
\item \textbf{Local (Strict) Comparison Property on $S$:}\\
At each point $\zz\in\overline{S}$, the Local (Strict) Comparison Property holds.
\item \textbf{Local (Strict) non-Majorization Property on $S$:}\\
At each point $\yy\in\overline{S}$, the Local (Strict) non-Majorization Property holds.
\item \textbf{Local (Strict) non-Minorization Property on $S$:}\\
At each point
$\yy\in\overline{S}$,
the Local (Strict) non-Minorization Property holds.
\item \textbf{Singular (Strict) Comparison Property on $S$:} \\
At each \emph{equioscillation point}
$\zz\in S$
the Local (Strict) Comparison Property holds.
\item \textbf{Singular (Strict) non-Majorization Property:} \\
At each \emph{equioscillation point} $\yy\in S$ the Local (Strict) non-Majorization Property holds.
\item \textbf{Singular (Strict) non-Minorization Property:} \\
At each \emph{equioscillation point} $\yy\in S$ the Local (Strict) non-Minorization Property holds.
\end{ABC}
\end{definition}

\begin{remark}\label{rem:hierarchy}
The comparison properties are symmetric in $\xx$ and $\yy$, while the non-majorization and non-minorization properties are not.
One has the following relations between the previously defined properties: (a)$\Rightarrow$(b) and (c), (A)$\Rightarrow$(B)$\Rightarrow$(E), (B)$\Rightarrow$(C) and (D), (E)$\Rightarrow$(F) and (G), (C)$\Rightarrow$(F), (D)$\Rightarrow$(G).
It will be proved in Corollary \ref{cor:majcomp} that for
\textit{strictly concave kernels}
all comparison, non-majorization and non-minorization properties (A), (B), (C), (D) (with their strict version as well) are equivalent to each other.
\end{remark}

\begin{remark}\label{rem:shi2}
	Shi \cite{shi} proved that (under condition \eqref{eq:bdry}) the Jacobi Property implies the  Comparison Property, the Sandwich Property,  and that the  Difference Jacobi Property implies the Equioscillation Property.
	Example \ref{ex:diffbutnostrong} below shows
that the  Comparison Property (even the Local Strict non-Majorization Property) fails in general, even though one has the Difference Jacobi Property.  In Proposition \ref{prop:mmatrix3} we will show that in our setting we always have the Difference Jacobi Property provided the kernels are at least twice continuously differentiable and, moreover we have the Equioscillation Property.
\end{remark}

\begin{example}\label{ex:diffbutnostrong}
	Let $n=1$ and $K_0:(0,2\pi)\to \RR$ be a strictly concave  kernel  function in $\Ce^\infty(0,2\pi)$ satisfying \eqref{eq:kernsing} and such that
	the  maximum of $K_0$ is $0$,
  while with some fixed $0<\alpha<\pi$ the function  $K_0$ is
    increasing
    in $(0,\alpha)$
     and is
    decreasing
    in $(\alpha,2\pi)$, and let $K_1(t):=K_0(2\pi-t)$.
	For $\yy:=y\in (0,2\pi)$ we have $F(\yy,t)=K_0(t)+K_1(t-y)=K_0(t)+K_0(2\pi+y-t)$, so
    by symmetry and concavity
    we obtain $z_0(\yy)=\frac{y}2$ and $z_1(\yy)=\frac{2\pi+y}2$.
So that
	\begin{align*}
m_0(\yy)&=F(\yy,z_0(\yy))=K_0(\tfrac{y}2)+K_0(2\pi+y-\tfrac{y}2)=2K_0(\tfrac{y}2),\\ m_1(\yy)&=F(\yy,z_1(\yy))=K_0(\tfrac{2\pi+y}2)+K_0(2\pi+y-\tfrac{2\pi+y}2)=2K_0(\tfrac{2\pi+y}2).
	\end{align*}
Whence we conclude that
	\begin{align*}
m_0(\yy+h)<m_0(\yy)\quad\text{and}\quad m_1(\yy+h)<m_1(\yy),
	\end{align*}
whenever $y\in (2\alpha,2\pi)$ and $h>0$ with $y+h\in (2\alpha,2\pi)$.
This shows that the non-Majorization Property does not hold in general. Since $m_0'(2\alpha)=0$, the Jacobi Property fails for this example (which anyway follows from Remark \ref{rem:shi1}). Notice also that
\[
	m_0'(\yy)-m_1'(\yy)=K_0'(\tfrac{y}2)-K_0'(\tfrac{2\pi+y}2)>0,
\]
since $K_0'$ is strictly decreasing, meaning that we have the Difference Jacobi Property (this holds in general, see Proposition \ref{prop:mmatrix3}).
Finally, we remark that we have the Singular non-Majorization Property. Indeed, $\yy$ is an equioscillation point if and only if
	\begin{align*}
	2K_0(\tfrac{y}2)=m_0(\yy)=m_1(\yy)=2K_0(\tfrac{2\pi+y}2),
	\end{align*}
i.e., at the corresponding points in the graph of $K_0$ there is a horizontal chord of length $\pi$. This implies that $y/2$ falls in the interval where $K_0$ is strictly increasing, whereas $\pi+y/2$ belongs to the interval where $K_0$ is strictly decreasing. Hence if we move $\yy=y$ slightly, $m_0$ and $m_1$ will change in different directions.
\end{example}
This example shows that Shi's results are not applicable in this general setting, even if we supposed the kernels to be in $\Ce^\infty(0,2\pi)$.

\section{Distribution of local minima of $\MM$}
\label{sec:distrlocmins}
{\piros In this section we start with a central perturbation result, which describes how for fixed permutation $\sigma$ the functions $m_{\sigma,j}(\yy)$ change for a small perturbation of $\yy$. This will allow us to relate local minimum points of $\MM$ and equioscillation points, see Proposition \ref{prop:minmaxnplusone}. Moreover,  the equioscillation property of the solutions of the minimax problem \eqref{eq:minmax} is established in Corollary \ref{cor:ezkell} under appropriate conditions on the kernels.}

\begin{remark}\label{rem:muij}
Suppose $f_j$ are (strictly) concave functions for $j=0,\dots, n$ and  let $f=\sum_{j=0}^n f_j$.
Let $\mu_j$ be the slope of a supporting line of $f_j$ at some point $t$. Then $\mu:=\sum_{j=0}^n\mu_j$ is the slope of a supporting  line of $f$ at the same point $t$. Conversely, if $\mu$ is given as the slope of a supporting  line at some point $t$, then it is not hard to find some $\mu_j$, $j=0,\dots, n$ being the slope of some supporting  line of $f_j$ at $t$ with $\mu=\sum_{j=0}^n \mu_j$.
\end{remark}
	
	\begin{lemma}[(Perturbation
lemma)]\label{lem:pertnodes00} Suppose that $K_0,\dots,
K_n$ are strictly concave. Let $\yy\in \TT^n$ be a node
system, and for $k\in\NN$, $1\leq k\leq n$ let $t_1,\dots, t_k\in
(0,2\pi)$ be all different from the nodes in $\yy$. Let
\[
\delta:=\tfrac 12\min\bigl\{|t_i-y_j|:i=1,\dots, k,\: j=0,\dots, n\bigr\}.
\]
For $i=1,\dots, k$
let $\mu^{(i)}$ be the slope of a supporting  line to the graph of $F(\yy,\cdot)$ at the point $t_i$.
Finally, let $\xx_{1},\dots, \xx_{n-k}\in \RR^n$ be fixed
arbitrarily.
\begin{abc}
	\item
Then there is $\vc{a}\in [-1,1]^n\setminus\{\vc{0}\}$  such
that $\xx_\ell^\trp \vc{a}= 0$ for $\ell=1,\dots,n-k$ and for all
$0<h<\delta$ we have
\begin{equation*}
F(\yy+h\vc{a},s_i)<F(\yy,t_i)+ \mu^{(i)}(s_i-t_i)
\end{equation*}
for all $s_i$ with $|s_i-t_i|<\delta$, $i=1,\dots,k$.

\item {\piros Let $S=S_\sigma$ be a simplex, and let $\yy\in \overline{S}$.}
If $F(\yy,\cdot)$ has local maximum in $t_i$
for some $i\in\{1,\dots,k\}$, i.e., if $t_i=z_{j}(\yy)\in \inter I_j(\yy)$ for some $j\in\{0,\dots,n\}$, then
\begin{equation*}
F(\yy+h\vc{a},s_i)<F(\yy,z_j(\yy))=m_j(\yy) \quad \mbox{for all $s_i$ with $|s_i-z_{j}(\yy)|<\delta$}.
\end{equation*}
\item For the fixed node system $\yy$ consider an admissible cut of the torus (cf.{} Remark \ref{rem:cut}). Let $i_1,\dots, i_k\in \{0,\dots,n\}$ be pairwise different, and suppose that $\widehat I_{i_1}(\yy),\dots, \widehat I_{i_k}(\yy)$ are non-degenerate and {\piros $\widehat z_{i_j}(\yy)\in\inter \widehat I_{i_j}(\yy)$} for each $j=1,\dots,k$. Then
there is $\eta>0$ such that for all $0<h<\eta$
\[
\widehat m_{i_j}(\yy+h\vc{a})<\widehat m_{i_j}(\yy) \quad \text{$j=1,\dots,k$}.
\]
\end{abc}
\end{lemma}

\begin{proof}
By Remark \ref{rem:muij} for  $i=1,\dots,k$ and $j=0,\dots,n$ there are $\mu_{ij}$ each of them being the slope of a supporting  line to the graph of $K_j$ at $t_i-y_j$, i.e., with
\[
\mu^{(i)}=\sum_{j=0}^n \mu_{ij}.
\]	
{\lila Take a vector  $\vc{a}\in[-1,1]^n\setminus\{\mathbf{0}\}$ with 
\begin{equation*}
\sum_{j=1}^n a_j\mu_{ij}\geq 0 \quad\mbox{for}~ i=1,\dots,k
\end{equation*}
	and with $\xx_\ell^\trp \vc{a}=0$ for $\ell=1,\dots,n-k$.}
Such a vector does  exist by standard linear algebra.  We set $a_0:=0$.

(a)  Since $K_j$ is concave, it follows
\[
K_j(s_i-(y_j+ha_j))\leq K_j(t_i-y_j) + \mu_{ij}(s_i-t_i-ha_j)
\]
for $s_i$ with $|s_i-t_i|<\delta$ and $0 \leq h <\delta$, because then $|s_i-t_i-ha_j|<\delta+|a_j|h <2\delta$ and $|t_i-y_j|\geq 2\delta$ guarantees that the full interval between the points $t_i-y_j$ and $s_i-(y_j+ha_j)$ stays in $(0,2\pi)$, i.e., the continuous change of $t_i-y_j$ to $s_i-(y_j+ha_j)$ happens within the concavity interval of $K_j$.

Observe that here in view of strict concavity equality holds for some $i,j$ if and only if $s_i-t_i-ha_j=0$.  However, for any given value of $i$, this cannot occur for all
$j=0,\dots,n$. Indeed, if this were so, then $a_0=0$ would imply $s_i=t_i$ and,
by $h> 0$, it would follow that $\vc{a}=0$, which was excluded.

Summing for all $j$, with at least one of the inequalities being strict, we obtain
\[
\sum_{j=0}^n K_j(s_i-(y_j+ha_j)) < \sum_{j=0}^n K_j(t_i-y_j)
+ \sum_{j=0}^n \mu_{ij}(s_i-t_i-ha_j)
\]
for $|s_i-t_i|<\delta$, $i=1,\dots, k$, i.e., dropping also $a_0=0$
\[
F(\yy+h\vc{a},s_i) < F(\yy,t_i) + \mu^{(i)} (s_i-t_i) - h \sum_{j=1}^n \mu_{ij} a_j.
\]
Now, by the choice of $\vc{a}$, the last sum is
non-negative, and since $h>0$ the last term can be
estimated from above by $0$, and we obtain the first statement.

\medskip\noindent (b) In the case when  $t_i=z_j(\yy)$ for some $j$ (and only then) the supporting  line can be chosen horizontal, i.e., $\mu^{(i)}=0$. Therefore, with this choice the already proven result directly implies the second statement.

\medskip\noindent (c) Take a fixed $\yy$ and an admissible cut of the torus at some $c$ (cf.{} Remark \ref{rem:cut}). 
For sufficiently small $\eta$ we have {\piros $\widehat z_{i_j}(\yy)\in \widehat I_{i_j}(\yy+h\vc a)$} for all $0<h<\eta$ and $j=1,\dots, k$. Since $\xx\mapsto \widehat z_{i_j}(\xx)$ is continuous at $\yy$ (see Lemma \ref{lem:zjcont}), for some possibly even smaller $\eta>0$ we have
$|\widehat z_{i_j}(\yy)-\widehat z_{i_j}(\yy+h\vc{a})|<\delta$, whenever $0<h<\eta$.
 From this we conclude, by the already proven part (b), that for all $j=1,\dots,k$
\[
\widehat m_{i_j}(\yy+h\vc{a})=F(\yy+h\vc{a},\widehat z_{i_j}(\yy+h\vc{a}))<\widehat m_{i_j}(\yy).\qedhere
\]
\end{proof}

The next lemma is an analogue of Lemma \ref{lem:new} for kernels in $\Ce^1(0,2\pi)$.
\begin{lemma}	Suppose the kernels $K_0,\dots, K_n$ are in $\Ce^1(0,2\pi)$ and are non-constant. Let $S=S_\sigma$ be a simplex, let $\yy\in \overline{S}$ and let $j\in \{0,\dots,n\}$. Then
there exists $\varepsilon>0$ such that either for all $t\in (y_j-\varepsilon,y_j)$ or for all $t\in (y_j,y_j+\varepsilon)$ we have
$F(\yy,t)>F(\yy,y_j)$.
\label{lem:KjConenew}
	\end{lemma}
    \begin{proof}
     Let the left and right neighboring non-degenerate arcs to
$y_j$ be $[y_\ell,y_j]$ and $[y_j,y_r]$, respectively.\footnote{If all nodes are positioned at $y_0=0$, these arcs can be the same.}  Let us write
$y_\ell < y_{j_1}=\dots=y_{j_\nu}<y_r$ with  $j_1=j$ 
{\lila (so that there exists a degenerate arc equal to $\{y_j\}$ 
precisely when $\nu>1$)}. 
We can assume $K_{j_\lambda}>-\infty$
for all $\lambda=1,\dots,\nu$, otherwise $F(\yy,y_j)=-\infty$, 
{\lila while} 
$F(\yy,\cdot)$ is finite valued on $(y_\ell,y_j)\cup (y_j,y_r)$, 
and the statement is trivial.
So summing up, $F(\yy,\cdot)$ is concave and 
continuously differentiable both on
$(y_\ell,y_j)$ and
$(y_j,y_r)$, and continuous on $[y_\ell,y_r]$.  
{
\lila

\medskip\noindent
Since  $F(\yy,\cdot)$ is concave, there
is a maximum point $z_\ell \in [y_\ell,y_j]$
(which, however, need not be unique if $F$ is not strictly concave), 
and by concavity $F(\yy,\cdot)$ is
non-decreasing
on $[y_\ell,z_\ell]$ and
non-increasing
on $[z_\ell,y_j]$. 
It follows that $F(\yy,z_\ell)\ge F(\yy,y_j)$. 
Moreover, in case we find strict inequality, 
we are done, for then 
\[F(\yy,t)\ge L(t):=\frac{y_j-t}{y_j-z_\ell}F(\yy,z_\ell)+\frac{t-z_\ell}{y_j-z_\ell}F(\yy,y_j) >F(\yy,y_j)\] for all $z_\ell<t<y_j$.

There remains the case when $F(\yy,z_\ell)=F(\yy,y_j)$, 
which means that $F(\yy,y_j)$ is maximum itself on $[y_\ell,y_j]$, too.

By an analogous reasoning either 
we find an interval $[y_j,y_j+\varepsilon]$, 
where the function is above $F(\yy,y_j)$, or $y_j$ 
is a maximum point even for the whole of $[y_j,y_r]$.

In all, either there are intervals as needed, or 
we find $F(\yy,y_j)=\max_{[y_\ell,y_r]} F(\yy,\cdot)$. 
Next, we show that this latter situation is impossible, 
which will conclude the proof.
}

\medskip\noindent So assume for a contradiction that $F(\yy,\cdot)$ stays below $F(\yy,y_j)$ on $[y_\ell,y_r]$, and hence we find
\[
D_-F(\yy,y_j)\geq 0 \geq D_+F(\yy,y_j).\]
Using the non-constancy of the kernel functions $K_i$
in the form that $D_-K_i(0)< D_+K_i(0)$, we find
\begin{align*}
D_-F(\yy,y_j) &= \lim_{t\upto y_j} \sum_{i=0}^n K_i'(t-y_i)
=
\sum_{\substack{\lambda=0
\\ \lambda \ne j_1,\dots,j_\nu}}^n
K_\lambda'(y_j-y_\lambda)
+ \sum_{\lambda=1}^\nu D_-K_{j_\lambda}(0)
\\& < \sum_{\substack{\lambda=0 \\ \lambda \ne j_1,\dots,j_\nu}}^n
K_\lambda'(y_j-y_\lambda)
+ \sum_{\lambda=1}^\nu D_+K_{j_\lambda}(0) = \lim_{t\downto y_j} \sum_{i=0}^n K_i'(t-y_i)= D_+F(\yy,y_j),
\end{align*}
which furnishes the required contradiction. Whence the
statement follows.
    \end{proof}

\begin{lemma}\label{l:neighborincrease} {\piros Let the kernel functions $K_0,\dots,K_n$ be concave, let $S_\sigma$ be a simplex, and let $\yy\in\overline{S_\sigma}$ be such that the interval
$I_j(\yy)=[y_j,y_{j'}]$ is degenerate, i.e., a singleton.}
	\begin{abc}
\item Suppose that the kernel $K_j$  satisfies condition \eqref{eq:kernsing'm}.  Then
there exists $\varepsilon>0$ such that for all $t\in (y_j-\varepsilon,y_j)$ we have
$F(\yy,t)>m_j(\yy)$.
\item Suppose that the kernel $K_j$  satisfies condition \eqref{eq:kernsing'p}.  Then
there exists $\varepsilon>0$ such that for all $t\in (y_j,y_j+\varepsilon)$ we have
$F(\yy,t)>m_j(\yy)$.
\item Suppose the kernels $K_0,\dots, K_n$ are in $\Ce^1(0,2\pi)$ and are non-constant. Then
there exists $\varepsilon>0$ such that either for all $t\in (y_j-\varepsilon,y_j)$ or for all $t\in (y_j,y_j+\varepsilon)$ we have
$F(\yy,t)>m_j(\yy)$.
\end{abc}
\end{lemma}
\begin{proof}
		Let $I_j(\yy)=\{y_j\}=\{y_{j'}\}=\{z_j(\yy)\}$ and let $\varepsilon>0$ be so small that the functions $K_k(\cdot-y_k)$ are all finite and concave on $(y_j-\varepsilon,y_j)$ and $(y_j,y_j+\varepsilon)$. In particular, for a point $t$ in one of these intervals $F(\yy,t)\in \RR$, so in case of $K_j(0)=-\infty$, we also have $F(\yy,z_j(\yy))=-\infty<F(\yy,t)$ and there is nothing  to prove.

	\medskip\noindent (a) and (b) follow from Lemma \ref{lem:new} and from the fact that $F(\yy,y_j)=m_j(\yy)$.

\medskip\noindent (c) follows from Lemma \ref{lem:KjConenew}
by also taking into account that $F(\yy,y_j)=m_j(\yy)$.
\end{proof}

\begin{corollary}\label{cor:neighborjump}
	Let the kernel functions $K_0,\dots,K_n$ be concave. {\piros Let $S_\sigma$ be a simplex and suppose that
	$I_j(\yy)$ is degenerate for some $\yy\in \overline{S_\sigma}$.}
		\begin{abc}
				\item Suppose that at least $n$ of the kernels $K_0,\dots, K_n$  satisfy condition \eqref{eq:kernsing'}.  Then for at least one neighboring, non-degenerate
			arc $I_\ell(\yy)$ we have $m_\ell(\yy)>m_j(\yy)$.

	\item Suppose the kernels  are in $\Ce^1(0,2\pi)$ and are non-constant. Then for at least one neighboring, non-degenerate
arc $I_\ell(\yy)$ we have $m_\ell(\yy)>m_j(\yy)$.
	\end{abc}
\end{corollary}

\begin{corollary}\label{cor:equiosciinside} If $K_0,\dots,K_n$
are non-constant, concave kernel functions and either  $n$ of them satisfy \eqref{eq:kernsing'}, or all belong to $\Ce^1(0,2\pi)$,
then an equioscillation point $\ee\in \TT^n$ must belong to
the interior of some simplex {\piros $S$, i.e., we  have
$\ee\in X=\bigcup_{\sigma} S_{\sigma}$.}
\end{corollary}
\begin{proof} {\piros Let $\yy \in \TT\setminus X$ be arbitrary, 
and choose a permutation $\sigma$ with $\yy\in \partial S_\sigma$.} 
Then there exists some $j$ with $I_j(\yy)$  degenerate. 
According to the above, there exists 
some $\ell\ne j$ with $m_j(\yy)<m_\ell(\yy)$, so there is no equioscillation at $\yy$.
\end{proof}

\begin{example} It can happen that an equioscillation point falls on the boundary of a simplex $S$, and that maximum points of non-degenerate arcs lie on the endpoints. Indeed, let $K_0:=-4\pi^3/|x|$ on $[-\pi,\pi)$, extended periodically, and let $K_1(x):=K_2(x):=-(x-\pi)^2$ on $(0,2\pi)$, again extended periodically. Observe that $K_0$ satisfies \eqref{eq:kernsing'pm} (and belongs to $\Ce^1((0,\pi)\cup(\pi,2\pi))$, and $K_1, K_2\in \Ce^1(0,2\pi)$. Still, for the node system $y_1=y_2=\pi$, we have $\yy\in \partial S=\partial S_{\rm Id}$, $F(\yy,x)=F(\yy,2\pi-x)=-4\pi^3/x-2x^2$ ($0\leq x\leq \pi$), hence $z_0=z_1=z_2=\pi$ and $m_0=m_1=m_2=F(\yy,\pi)=-6\pi^2$, showing that $\yy$ is in fact an equioscillation point.
\end{example}

\begin{lemma}	Suppose the kernels $K_0,\dots, K_n$ are strictly concave and either all
satisfy \eqref{eq:kernsing'}, or all belong to $\Ce^1(0,2\pi)$.
	{\piros Let $\ww\in \TT^n$ and fix a permutation $\sigma$ with $\ww\in \overline{S_\sigma}$ to determine the ordering of the nodes. If $j\in \{0,\dots,n\}$  is such that $m_j(\ww)=\MM(\ww)$, 
	 then $I_j(\ww)$ is non-degenerate and $z_j(\ww)$ belongs to the interior of $I_j(\ww)$.}
\label{lem:interioronndeg}

	\end{lemma}
	\begin{proof}
		By Corollary \ref{cor:neighborjump} it follows that the arc $I_j(\ww)=[w_j,w_r]$  is non-degenerate.
		
		\medskip\noindent
        Suppose first that all kernels satisfy \eqref{eq:kernsing'}.
{\piros In this case, $F$ can attain global maximum neither at $w_j$ nor at $w_r$, because $F$
is strictly increasing on a left or a right neighborhood of these nodes due to condition \eqref{eq:kernsing'} (use  Lemma \ref{lem:new}).} Therefore, in this case $z_j(\ww)$ belongs to the interior of $I_j(\ww)$.
		
		\medskip\noindent Next, let us suppose that the kernels are in $\Ce^1(0,2\pi)$. By an application of Lemma \ref{lem:KjConenew} we obtain $\MM(\ww)>F(\ww,w_i)$ for all $i=0,1\dots, n$. Hence, in the case $\MM(\ww)=m_j(\ww)=F(\ww,z_j)$, we cannot have $z_j=w_j$ or $z_j=w_r$.
		\end{proof}

As usual, we say that a point $\ww\in\TT^n$ is a local
minimum point of $\MM$ if there exists $\eta>0$ such that
\begin{equation}
\label{cond:locmin}
\MM(\wS)=\min \{ \MM(\yy) : d_{\TT^n}(\yy,\wS)<\eta \}.
\end{equation}
Note that the $\eta$-neighborhood here may intersect several
different simplexes.

\begin{prp}\label{prop:minmaxnplusone}
Suppose the kernels $K_0,\dots, K_n$ are strictly concave and {either all
satisfy \eqref{eq:kernsing'}, or all belong to $\Ce^1(0,2\pi)$. }
Let $\wS\in \TT^n$ be a local minimum point of $\MM$,
see \eqref{cond:locmin}.
{\piros Then $\wS$ is an equioscillation point, i.e.,
\begin{equation*}
\mm(\wS)=\MM(\wS).
\end{equation*}}
As a consequence, such a local minimum point belongs to $X=\bigcup_{\sigma} S_\sigma$.
\end{prp}
\begin{proof} Consider an  admissible cut of the torus (cf.{} Remark \ref{rem:cut}).
Suppose for a contradiction that $i_1,\dots,
i_k\in\{0,\dots,n\}$ with $k\leq n$ are precisely the
indices $i$ with \[\widehat{m}_i(\wS)=\MM(\wS)=:M_0.\]  
By Lemma \ref{lem:interioronndeg} {\piros $t_{j}:=\widehat z_{i_j}(\wS)$} (for $j=1,\dots, k$) belong to the interior of non-degenerate arcs.
With this choice we can use the Perturbation Lemma \ref{lem:pertnodes00} to
slightly move $\wS=(w_1,\dots, w_n)$ to $\vc{w}'=(w'_1,\dots,
w'_n)$, $|\vc{w}'-\wS|<\eta$ and achieve
\begin{equation*}
\max_{j=1,\dots,k}
\widehat{m}_{i_j}(\vc{w}')<M_0,
\end{equation*}
while at the same time $\widehat m_{q}(\vc{w}')$ for $q\neq i_j$,
$j=1,\dots,k$ do not increase too much (because by Proposition
\ref{prop:cont} the functions $\widehat m_q$ are continuous), i.e.,
\begin{equation*}
\max_{p=0,\dots,n}m_{p}(\vc{w}')=\max_{j=1,\dots,k}\widehat{m}_{i_j}(\vc{w}')<M_0,
\end{equation*}
which is a contradiction.

\medskip\noindent The last assertion follows now immediately from Corollary \ref{cor:equiosciinside}.
\end{proof}

\begin{corollary} \label{cor:SS'better} Suppose that the kernels $K_0,\dots, K_n$ are
strictly concave, and that either all
satisfy \eqref{eq:kernsing'}, or all belong to $\Ce^1(0,2\pi)$.
Let $S=S_\sigma$ be a simplex, and let $\wS\in \overline{S}$ be
an extremal node system for \eqref{eq:Sminmax}. Then the
following assertions hold.
\begin{abc}
\item If $\wS \in S$, then $\wS$ is an equioscillation point.
\item Even in case $\wS\in \partial S$ we have that $\wS$
    is a weak equioscillation point.
\item Furthermore, if also \eqref{eq:kernsing} holds, then
    we have $\{m_0(\wS),\dots, m_n(\wS)\}\subseteq
    \{-\infty,M(S)\}$, with $m_j(\wS)=-\infty$ iff
    $I_j(\wS)$ is degenerate.
\item If $\wS\in \partial
    S$, then there exists another simplex
    $S'=S_{\sigma'}$  with $\wS\in \overline{S} \cap
    \overline{S'}$ and $M(S')<M(S)$, moreover $\wS$ is not
    even a local (conditional) minimum within
    $\overline{S'}$.
\end{abc}
\end{corollary}
\begin{proof} (a) When the extremal node system $\wS$ lies
in the interior of the simplex $S$, it is necessarily a local
minimum point, hence the previous Proposition
\ref{prop:minmaxnplusone} applies.

\medskip\noindent (b) For notational convenience we assume without loss of generality that $\sigma=\id$, the identical pertmutation. Let $\wS=(w_1,\dots,w_n)\in \partial S$ and assume that
\begin{align*}
	0&=w_0=\dots=w_{i_0}<w_{i_0+1}=\dots=w_{i_0+i_1}<w_{i_0+i_1+1}=\dots =w_{i_0+i_1+i_2}\\
    &<\dots< w_{i_0+\dots +i_{s-1}+1} = \dots= w_{i_0+\dots +i_s}<w_{i_0+\dots +i_s+1} :=2\pi
\end{align*}
is the listing of nodes with
the number of equal ones being exactly $i_0, i_1,\dots,i_{s}$. Thus we have $i_0+\dots +i_{s}=n$ with $i_0$
possibly $0$ but the other $i_j$'s are at least $1$, and the number of distinct nodes strictly in $(0,2\pi)$
is $s$.

In between the equal nodes there are degenerate arcs $I_k$,
where---in view of Corollary \ref{cor:neighborjump}---the
respective maximum $m_k(\wS)$ of the function $F(\wS,\cdot)$ is
strictly smaller, than one of the maximums on the neighboring
non-degenerate arcs, hence $m_k(\wS)$ is also smaller than
$\MM(\wS)$.

So in particular if $s=0$ and there is only one non-degenerate
arc $I_{i_0}=[0,2\pi]$, with all the nodes merging to $0$, then weak equioscillation (of this one value $m_{i_0}$)
trivially holds.

Next, assume that there exists at least one node $0<w_k<2\pi$, and
let us now define a new system of $s$ ($1 \leq s <n$) nodes
$\yy=(y_1,\dots,y_s)$ with $y_j=w_{i_0+\dots+i_j}$
($j=1,\dots,s$) extended the usual way by $y_0=0$. Note that we
will thus have $0=y_0<y_1<\dots<y_s<2\pi$, and the arising $s$ non-degenerate
arcs between these nodes are exactly the same as the
non-degenerate arcs determined by the node system $\wS$.

Further, let us define new kernel functions $L_j:=K_{i_0+\dots+
i_{j-1}+1}+\dots +K_{i_0+\dots+i_j}$ for $j=1,\dots,s$, and
$L_0=K_0+K_1+\dots+K_{i_0}$.
Obviously, the new $s+1$-element system $L_0,L_1,\dots,L_s$
consists of strictly concave kernels, either {all}  satisfying
\eqref{eq:kernsing'}, or all belonging to $\Ce^1(0,2\pi)$, and
now the node system $\yy$ belongs to the interior of the
respective $s$-dimensional simplex $\tilde{S}$.

Observe that by construction we now have
\[\tilde{F}(\yy,t)=\sum_{j=0}^s L_j(t-y_j) = \sum_{i=0}^n
K_i(t-w_i) = F(\wS,t),\] and so from the assumption that
$\MM(\wS)$ is minimal within the simplex $S$, it also follows that
$\sup_{t\in \TT}\tilde F(\yy,t)$
is minimal within $\tilde{S}$.
Therefore, by part (a)  the maximum values
$\widetilde{m}_j$ of the function $\tilde{F}$ on
these non-degenerate arcs are all equal, and this was to be proven.

\medskip\noindent (c) is obvious once we have the weak equioscillation in view of
(b).

\medskip\noindent (d) If we had $\wS$ being a local conditional minimum point in each
of the simplexes to the boundary of which it belongs, then
altogether it would even be a local minimum point on $\TT^n$. Then Proposition
\ref{prop:minmaxnplusone} would yield $\wS\in X$, contradicting the
assumption. So there has to be some simplex $S'$, containing
$\wS$ in $\partial S'$, where $\wS$ is not a local conditional
minimum point. Consequently, $M(S')<\MM(\wS)=M(S)$, whence the
assertion follows.
\end{proof}

\begin{corollary}\label{cor:ezkell}
Suppose the kernels $K_0,\dots, K_n$ are strictly concave and
{either all
satisfy \eqref{eq:kernsing'}, or all belong to $\Ce^1(0,2\pi)$.}
 If $\ww$ is an extremal node system for
\eqref{eq:minmax}, i.e.,
\begin{equation*}
\MM(\ww)=\min_{\yy\in\TT^n}\MM(\yy)=M,
\end{equation*}
then the nodes $w_j$ ($j=0,\dots,n$) are pairwise different
(i.e., $\ww\in X$) and, moreover, $\ww$ is an equioscillation
point, i.e., we have
\begin{equation*}
m_{j}(\ww)=M\quad\mbox{for $j=0,\dots, n$}.
\end{equation*}
\end{corollary}

\section{Distribution of local maxima of $\mm$}
\label{sec:concavity}
{\piros In this section we prove that the function 
$\mm$ is (strictly) concave on any closed simplex $\overline{S}$, 
if the kernels are such. 
As a corollary we obtain a unique 
solution of the maximin problem \eqref{eq:Smaxmin}.}

\begin{lemma}
	Suppose the kernels $K_0,\dots, K_n$ are strictly concave.
	Let $S=S_\sigma$ be a simplex. Then $F(\yy,s) :\TT^n\times\TT \to [-\infty,\infty)$ restricted to the convex open set
	\[
	\DD:=\DD_{\sigma,i}:=\bigl\{(\yy,s) :\yy\in S ~\text{and}~ s\in \inter I_i(\yy)\bigr\}
	\]
	is strictly concave.
\label{l:Fstrictconcave}
\end{lemma}
\begin{proof} First, note that the set $\DD:=\DD_{\sigma,i}$ is a convex subset of $\TT^{n+1}$. Indeed, let $(\xx,r), (\yy,s) \in \DD$
and	$t\in [0,1]$. Then $x_i<x_\ell$ and $y_i<y_\ell$ imply $tx_i+(1-t)y_i < tx_\ell+(1-t)y_\ell$, and $x_i<r<x_\ell$, $y_i<s<y_\ell$ entails also $tx_i+(1-t)y_i <tr+(1-t)s <tx_\ell+(1-t)y_\ell$.

\medskip\noindent
Now, consider the sum representation of $F$ and concavity of each $K_\ell$ to conclude
\begin{align}\label{Fconcave}
F(t(\xx,r)+(1-t)(\yy,s))&=\sum_{\ell=0}^n K_\ell(tr+(1-t)s-
(tx_\ell+(1-t)y_\ell))
\notag \\ & \geq \sum_{\ell=0}^n t
K_\ell(r-x_\ell)
+(1-t)
K_\ell(s-y_\ell)
\notag \\ & = tF(\xx,r) +(1-t)F(\yy,s).
\end{align}
This shows concavity of $F$. To see strict concavity suppose $t\ne 0,1$ and that  $(\xx,r), (\yy,s) \in \DD$ are different points.  If $r\ne s$, then using the strict concavity of $K_0$ we must have \[K_0(tr+(1-t)s)> t K_0(r) +(1-t)K_0(s),\] and if $r=s$, but $x_\ell \ne y_\ell$ for some $1\leq \ell \leq n$, then using strict concavity of $K_\ell$ (and also that $r=s$) it follows that \[K_\ell(tr+(1-t)s-(tx_\ell+(1-t)y_\ell))= K_\ell(s-(tx_\ell+(1-t)y_\ell)) > t K_\ell(s-x_\ell)+(1-t)K_\ell(s-y_\ell).\] Altogether we obtain strict inequality in \eqref{Fconcave}.
\end{proof}

\begin{prp}\label{l:mistrictconcave}
		Suppose the kernels $K_0,\dots, K_n$ are strictly concave.
	Then for all $i=0,1,\dots,n$, the functions $m_i(\yy) :S \to \RR$ are also strictly concave.
	As a consequence,
	\[
	\mm:S\to [-\infty,\infty),\quad \mm(\yy):=\min_{j=0,\dots,n}m_j(\yy)
	\]
	is a strictly concave function.
\end{prp}
\begin{proof}
Take $i\in \{0,1,\dots,n\}$, $\xx, \yy \in S$ and abbreviate $w:=z_i(\xx)$, $v:=z_i(\yy)$ (the unique maximum points of $F(\xx,\cdot)$ and $F(\yy,\cdot)$ in $I_i(\xx)$ and $I_i(\yy)$, respectively, i.e., $m_i(\xx)=F(\xx,w)$, $m_i(\yy)=F(\yy,v)$).  Let $\zeta(t):=z_i(t\xx+(1-t)\yy)$, $\zeta(0)=v$,  $\zeta(1)=w$.
According to the previous Lemma \ref{l:Fstrictconcave} the function $F$ is strictly concave on $\DD_{\sigma,i}$, hence for different $\xx\ne \yy$ we necessarily have
\[
F(t(\xx,w)+(1-t)(\yy,v))> tF(\xx,w) +(1-t)F(\yy,v)=tm_i(\xx)+(1-t)m_i(\yy).
\]
Here the left-hand side can be written as $F(t\xx+(1-t)\yy,\omega(t))$ with \[
\omega(t)=tw+(1-t)v\in I_i(t\xx+(1-t)\yy).\] Thus by the definition of $m_i$
we have
\[
m_i(t\xx+(1-t)\yy)=\max_{s\in I_i(t\xx+(1-t)\yy)} F(t\xx+(1-t)\yy,s) \geq
F(t(\xx,w)+(1-t)(\yy,v)).
\]
Hence, the previous considerations yield even $m_i(t\xx+(1-t)\yy)>tm_i(\xx)+(1-t)m_i(\yy)$, whence the first assertion follows. Since minimum of strictly concave functions is strictly concave, the last assertion follows, too.
\end{proof}

\begin{corollary} \label{cor:muniquenomaj}
	Suppose the kernels $K_0,\dots, K_n$ are strictly concave, and let $S:=S_\sigma$ be a simplex.
\begin{abc}
\item In $\overline{S}$ the function $\mm$ has a \emph{unique} global maximum point $\ys$, and no local minimum point in $S$.
\item If the kernels satisfy \eqref{eq:kernsing}, then $\ys\in S$.
\item There is no other point in $\overline{S}$  majorizing $\ys$ than $\ys$ itself.
\end{abc}
\end{corollary}
\begin{proof}
	(a) Since $\mm$ is strictly concave on ${S}$  and continuous on $\overline{S}$ the assertion is evident.

\medskip\noindent (b) Under condition \eqref{eq:kernsing} we have $\mm|_{\partial S}=-\infty$, whence the assertion is immediate.

\medskip\noindent (c) If $\xx\in\overline{S}$ with $m_j(\xx) \geq m_j(\ys)$ for all $j=0,1,\dots,n$, then for
$\mm=\min_{j=0,\dots,n} m_j$
we also have $\mm(\xx)\geq \mm(\ys)$, hence $\xx$ is also a maximum point, and by uniqueness (part (a)) this entails $\xx=\ys$.
\end{proof}

\section{Local properties of sums of translates}\label{sec:localprop}
{\piros
Exploiting concavity of $\mm$ (as has been proven in the previous section), we can study now the Comparison Property and the Sandwich Property and relate these to the non-uniqueness of equioscillation points in a closed simplex $\overline{S}$, see Proposition \ref{prop:whatifequinon-unique}. By putting the previous results together we can prove a version of Theorem \ref{thm:0mainspecialcase} for a \emph{given and fixed} simplex. This is the content of Proposition \ref{prop:Mmextremal}.}

\begin{corollary}
	Suppose the kernels $K_0,\dots, K_n$ are strictly concave. Let $S:=S_\sigma$ be a simplex.
\label{cor:majcomp}
	\begin{abc}
\item	Let $\yy\in S$, $\xx\in\overline{S}$, $\xx\neq \yy$ be such that $\xx$ majorizes $\yy$, i.e., $m_j(\xx)\ge m_j(\yy)$ for each $j=0,\dots,n$. Then there are $\vca \in \RR^n$ and $\delta>0$ such that for every $j=0,\dots, n$
\begin{align*}
 &m_j(\yy+t\vca)> m_j(\yy) \qquad \big( t\in (0,\delta) \big),\\
		& m_j(\yy-t\vca)< m_j(\yy)   \qquad  \big( t\in (0,\delta) \big).
        \end{align*}
In particular, the Local Strict non-Majorization Property (b) and non-Minorization Property (c) fail at $\yy$.
\item On $S$  the Local non-Majorization Property (C), the Local non-Minorization Property (D), the Local Comparison Property (B) and the Comparison Property (A) are all equivalent, also together with their strict versions.
\end{abc}
\end{corollary}
\begin{proof} (a)
Take $\vca:=\xx-\yy$ and let
\[
\yy_t:=\yy + t \vca = (1-t) \yy + t \xx.
\]
For sufficiently small $\delta>0$ we have $\yy_t\in S$ for every $(-\delta,1]$ (since $S$ is convex and open).
By the strict concavity of $m_j$ we obtain for $t\in (0,1)$ that
	\begin{align*}
		m_j(\yy_t) &>(1-t) m_j(\yy) + t m_j(\xx) \geq (1-t) m_j(\yy)+t m_j(\yy)=m_j(\yy)
		\intertext{and for $t\in (-\delta,0)$}
		m_j(\yy_t) &<(1-t) m_j(\yy) + t m_j(\xx) \leq (1-t) m_j(\yy)+t m_j(\yy)=m_j(\yy).
	\end{align*}
This proves the first assertion.

\medskip\noindent (b) The Comparison Property evidently implies the Local Comparison Property and that implies further the Local non-Minorization and non-Majorization Properties. The already established first assertion (a) provides the converse
implications
if
we start with the even weaker Local Strict non-Minorization or non-Majorization Properties.
\end{proof}

\begin{prp}\label{prop:whatifequinon-unique}
	Suppose that the kernel functions $K_0,\dots,K_n$ are strictly concave.
	 Let $S=S_\sigma$ be a fixed simplex and let $\ee, \ff \in \overline{S}$ be  two different equioscillation points.
	\begin{abc}
\item
Then we have $M(S)<m(S)$, and the Sandwich Property (see Definition \ref{def:sandwich} and Remark \ref{rem:sandwich}) fails.
\item If $\MM(\ee)\leq \MM(\ff)$ and $\ee\in S$, then the
Local Strict  non-Majorization (b)
and
 all the non-Minorization Properties
fail to hold at $\ee$.
\item If the kernels {either all
satisfy \eqref{eq:kernsing'},  or are all in $\Ce^1(0,2\pi)$}, then   the Comparison Property (A) fails
  (see Definition \ref{def:majorization}).
\end{abc}
\end{prp}
\begin{proof}
 For definiteness assume,
as we may, that $\MM(\ee)\leq \MM(\ff)$.
	
\medskip\noindent (a) If $\MM(\ee)<\MM(\ff)$, then we obviously have $M(S)\leq \MM(\ee)<\MM(\ff)=\mm(\ff)\leq m(S)$. If, on the other hand, $\MM(\ee)=\MM(\ff)$, then for the point $\vcg:=\frac12 (\ee+\ff)\in \overline{S}$ by the strict concavity we find
$m_j(\vcg) >\frac12 (m_j(\ee)+m_j(\ff)) = \MM(\ee)$
for all $j=0,\dots,n$, hence also $\mm(\vc{g})>\MM(\ee)$ and thus also $m(S)\geq \mm(\vc{g})>\MM(\ee)\geq M(S)$. In both cases the Sandwich Property  must fail, because by Remark \ref{rem:sandwich} this property is equivalent to $M(S)\geq m(S)$.

\medskip\noindent (b)  If $\MM(\ee)\leq \MM(\ff)$, then $\ff$ majorizes $\ee$, so Corollary \ref{cor:majcomp} (a) finishes the proof.

\medskip\noindent (c) Under the conditions we have $\ee,\ff\in S$ in view of Corollary \ref{cor:equiosciinside}. According to the previous part (b), we find that the the
 Local Strict  non-Majorization (b)
and
non-Minorization Properties (c), (D) and (G)
fail to hold at $\ee$. However, it has already been noted in Remark \ref{rem:hierarchy} that in this case the Comparison Property (A) must fail as well.
\end{proof}

\begin{corollary}\label{l:WeakLocalM}
	Suppose the kernels $K_0,\dots, K_n$ are strictly concave.
Let $S:=S_\sigma$ be a simplex and let  $\yS \in S$ be a  local minimum point of $\MM$, see \eqref{cond:locmin}.
\begin{abc}
	\item Then there exists no other point different from $\yS$ in  $\overline{S}$  majorizing $\yS$.
\item Suppose the kernels {either all
satisfy \eqref{eq:kernsing'}, or all are in $\Ce^1(0,2\pi)$.}
 Then there exists no other local minimum point of $\MM$ in the sense \eqref{cond:locmin}
in the closure $\overline{S}$ of $S$.
\end{abc}
\end{corollary}
\begin{proof}
	\medskip\noindent (a) Suppose $\xx\in \overline{S}$ majorizes $\yS$ and $\xx\ne\yS$. Then by Corollary \ref{cor:majcomp} (a) there are $\vca\in \RR^n$ and  $\delta>0$ with  $m_j(\yS-t\vca)<m_j(\yS)$ for every $t\in (0,\delta)$ and $j=0,\dots,n$. Hence $\yS$ cannot be a local minimum point for $\MM$.
	
\medskip\noindent (b)  By Proposition \ref{prop:minmaxnplusone}, under the conditions on the kernels the local minimum points of $\MM$ are also equioscillation points. Therefore, if $\yy\in \overline{S}$, $\yy\neq \yS$ is another local minimum point of $\MM$, then one of $\yy$ and $\yS$ majorizes the other. But then by part (a) the two points must be equal.
\end{proof}

To sum up our findings we can state:

\begin{prp}\label{prop:Mmextremal}
	Suppose the kernels $K_0,\dots, K_n$ are strictly concave and {either all
satisfy \eqref{eq:kernsing'}, or all belong to $\Ce^1(0,2\pi)$.}
	Let $S:=S_\sigma$ be a simplex.
	 If $\MM$ has a local minimum point $\yS \in S$, then $\yS$ is a unique point of equioscillation in $\overline{S}$, and $\mm$ has there its (unique, global) maximum. In particular, then $M(S)=m(S)$. Moreover, the Sandwich Property holds true in $S$. Furthermore,  the Singular non-Majorization and non-Minorization Properties hold on $S$.
 \end{prp}
\begin{proof}
	Let $\ys\in\overline{S}$ be the (unique, global) maximum point of $\mm$,
see Corollary \ref{cor:muniquenomaj} (a).
    Obviously,
	\[\min_{j=0,\dots,n} m_j(\ys)=\mm(\ys) \geq \mm(\yS).\] By assumption we can apply Proposition \ref{prop:minmaxnplusone} to conclude that $\yS$ is an equioscillation point, i.e., $\mm(\yS) = \MM(\yS) =m_j(\yS)$ for $j=0,\dots,n$. Thus we find that $\ys$ majorizes the point $\yS$. According to Corollary \ref{l:WeakLocalM} (a) this is not possible unless $\ys=\yS$. Therefore we obtain $M(S)=m(S)$, and Remark \ref{rem:sandwich} yields the Sandwich Property. If $\ee\in \overline{S}$ is another equioscillation point, then $\MM(\ee)\geq \MM(\yS)$ (since $\yS$ is a minimum point). By Proposition \ref{prop:whatifequinon-unique} (a) this would imply  $M(S)<m(S)$, which would be a contradiction.
Therefore, there exists no other equioscillation point in $\overline{S}$ than $\yS$ itself.
	Since $\yS\in S$ is a local minimum point of $\MM$, by Corollary \ref{l:WeakLocalM} (a) there is no point majorizing it. But also $\yS$ is the unique global minimum point  of $\MM$, so there is no point in $\overline{S}$ minorizing it.
	\end{proof}

\section{The Difference Jacobi Property}\label{sec:equi}
{\piros
In this section we show that if the kernels are 
in $\Ce^2(0,2\pi)$ with strictly negative second derivative, 
then we have the Difference Jacobi Property on any simplex. 
This will result in a global homeomorphism result (Corollary \ref{cor:diffhomeo}) and in 
the uniqueness of equioscillation points 
(in a fixed simplex) under the 
condition \eqref{eq:kernsing}, see Corollary \ref{cor:equiuniqueC2}.}

\begin{prp}
Suppose that $K_0,\dots, K_n$ are in $\Ce^2(0,2\pi)$ with $K''_j<0$ ($j=0,\dots,n$), and let $S=S_\sigma$ be a simplex.
For $j=0,\dots, n$ the functions $m_j(\yy)$ are continuously differentiable in $S$ and
\label{prop:jacobielemei0}
\begin{equation}
\label{eq:mudiff}
\frac{\partial m_j}{\partial y_r} (\yy)
=-K_r'\big(z_j(\yy)- y_r\big)\quad\mbox{for $r=1,\dots,n$.}
\end{equation}
\end{prp}
\begin{proof}
	Let $\yy\in S$ be fixed.	Recall that $t=z_j(\yy)$ is the unique maximum point in $I_j(\yy)$, i.e., with $F'(\yy,t)=0$. Since
	\[
	F''(\yy,t)=K_0''(t)+\sum_{j=1}^nK_j''(t-y_j)<0
	\]
	by the implicit function theorem, for a suitable neighborhood $U\times V\subseteq  S\times I_j(\yy)$ we have that  $z_j:U\to V$ is continuously differentiable. Since  $m_j(\yy)=F(\yy,z_j(\yy))$ we obtain that $m_j$, too is continuously differentiable and
\begin{align*}
\frac{\partial m_j}{\partial y_r}(\yy)
&=
\frac{\partial }{\partial y_r}\Bigl(F\big( \yy, z_j(\yy)\big)\Bigr)
=
\frac{\partial F} {\partial y_r} \big( \yy, z_j(\yy)\big)
+
\frac{\partial}{\partial t}  F\big( \yy, t \big)\vert_{t=z_j(\yy)}
\frac{\partial}{\partial y_r} z_j\big(\yy\big)
\\
&=
-K_r'\big(z_j(\yy)- y_r\big).\qedhere
\end{align*}
\end{proof}

As a consequence,
the Jacobian matrix
$D\vc{m}$ of $\vc{m}=(m_0,\dots, m_{n})^\trp$ is
\begin{equation}
\label{mujacobi}
D\vc{m} =
\bordermatrix{
~ & ~ & \overset{r}{\downarrow} & ~ \cr
~ & ~ & \vdots & ~ \cr
j\rightarrow & \cdots & -K_r'\big(z_j(\yy)- y_r\big) & \cdots \cr
~ & ~ & \vdots & ~
}
\end{equation}
where $j=0,\ldots,n$ and $r=1,\ldots,n$ .

For a given permutation $\sigma$ of $\{1,\dots, n\}$ let us consider the mapping $\Delta_\sigma$ defined by
\begin{equation}
\label{def:delta}
\Delta_\sigma(\yy) :=
(m_{\sigma(1)}(\yy)-m_{\sigma(0)}(\yy),\ldots,m_{\sigma(n)}(\yy)-m_{\sigma(n-1)}(\yy))^\trp.
\end{equation}
Its
Jacobian matrix
$D\Delta_\sigma$ is
\begin{equation}
\label{deltajacobi}
D\Delta_\sigma (\yy) =
\bordermatrix{
~ & ~ & \overset{r}{\downarrow} & ~ \cr
~ & ~ & \vdots & ~ \cr
j\rightarrow & \cdots & -K_r'\big(z_{\sigma(j)}(\yy)- y_r\big) + K_r'\big(z_{\sigma(j-1)}(\yy)- y_r\big) & \cdots \cr
~ & ~ & \vdots & ~
}
\end{equation}
where   $j=1,\ldots,n$ and $r=1,\ldots,n$.

\begin{prp}
Suppose that for each $j=0,\dots,n$ the kernel $K_j$ belongs to $\Ce^2(0,2\pi)$ with $K_j''<0$.
Let $S=S_\sigma$ be a simplex and let $\yy\in S$ be such that for each $j=0,1,\dots,n$ we have $z_j(\yy)\in \intt I_j(\yy)$.
 Then, the
Jacobian matrix of $\Delta_\sigma(\yy)$
is non-singular. That is, on $S$, we have the Difference Jacobi Property.
\label{prop:mmatrix3}
\end{prp}
\begin{proof}
For the sake of brevity we may suppose $\sigma=\id$, i.e., $\sigma(j)=j$, otherwise we can relabel the kernels $K_j$ accordingly.
We abbreviate $z_j:=z_j(\yy)$ and have
 according to the assumption
\[
 z_{j-1}<y_{j}<z_{j}\quad\mbox{ for } j=1,\dots,n.
\]
Write $A:=-D \Delta_\sigma(\yy)$.
First, we show that $A$ is a so-called Z-matrix, that is,
the entries are non-negative on the diagonal and are non-positive off the diagonal (see e.g. \cite{bermanplemmons}, p. 132 and p. 279).

On the diagonal the entries are
{\lila $K_r'(z_{r}- y_r) - K_r'(z_{{r-1}}- y_r)$}, $r=1,\ldots,n$.
Since $0<z_{{r-1}}<y_r<z_{r}<2\pi$ 
we obtain $z_{{r-1}}-y_r<0<z_{r}-y_r$ and 
{\lila $2\pi+z_{{r-1}}-y_r<2\pi$, furthermore, $0<z_{r}-y_r<2\pi+z_{{r-1}}-y_r<2\pi$.} 
Now, using the $2\pi$ periodicity of $K_r'$ and that $K_r'$ is strictly monotone decreasing {\lila in $(0,2\pi)$},
we obtain
$K_r'(z_{{r-1}}- y_r) < K_r'(z_{r}- y_r)$, that is,
$K_r'(z_{r}- y_r) - K_r'(z_{{r-1}}- y_r) > 0$.

\medskip\noindent
For
 $j<r$
we have
 $z_{{j-1}}<z_{j}\leq z_{{r-1}}<y_r$.
Therefore,
$-2\pi<z_{{j-1}}-y_r<z_{j}-y_r <0$
and
using that $K_r'$ is strictly monotone decreasing and $2\pi$ periodic,
we can write
\[
 K_r'(z_{j}- y_r) - K_r'(z_{{j-1}}- y_r)<0.
\]
Therefore the elements above the diagonal of $A$ are strictly negative.

If
$j>r$,
then
 $y_r<z_{r}\leq z_{{j-1}}<z_j$.
As above,
 $0<z_{{j-1}}-y_r<z_{j}-y_r <2\pi$
and using that $K_r'$ is strictly monotone decreasing,
we can write
\[
 K_r'(z_{j}- y_r) - K_r'(z_{{j-1}}- y_r) < 0,
\]
meaning that the entries below the diagonal of $A$ are strictly negative, too.
 So we have seen that $A$ is a Z-matrix.

\medskip\noindent We now show that the column sums of $A$  are strictly positive.
Indeed, the sum of the $r^{\text{th}}$ column of $A$ is telescopic
\begin{align*}
\sum_{l=1}^n \Bigl(K_r'(z_{l}-y_r) - K_r'(z_{{l-1}}-y_r)\Bigr)=K_r'(z_{n}-y_r)-K'_r(z_{0}-y_r).
\end{align*}
Since  $0<z_{0}<y_r<z_{n}<2\pi$, we have $0<z_{n}-y_r<2\pi+z_{0}-y_r<2\pi$. Since $K'_r$ is strictly decreasing and $2\pi$ periodic, it follows $K_r'(z_{n}-y_r)-K'_r(z_{0}-y_r)>0$.

Therefore, with  $\xx=(1,1,\ldots,1)^\trp\in \RR^n$ we have  $A^\trp \xx$ is a strictly positive vector. This means that $A^\trp$ satisfies condition I27 in \cite{bermanplemmons} (see page 136). Hence by  Theorem 2.3 on pp.{} 134--138 in \cite{bermanplemmons} it follows that $A^\trp$ is an M-matrix and is non-singular, this yielding also the non-singularity of $-A$. The proof is hence complete.
\end{proof}

\begin{corollary}\label{cor:diffhomeo}
	 Suppose that for each $j=0,\dots,n$ the kernel $K_j$ belongs to $\Ce^2(0,2\pi)$ with $K_j''<0$ and satisfies \eqref{eq:kernsing}.
	Let $S=S_\sigma$ be a simplex.
The mapping $\Delta_\sigma:S\to \RR^n$ is then a homeomorphism.
\end{corollary}
\begin{proof}
By Proposition \ref{prop:mmatrix3} the mapping $\Delta_\sigma$ is locally a homeomorphism (onto its image), and by Proposition \ref{prop:bdry} it carries the boundary $\partial S$ onto the boundary of the one-point compactified $\RR^n$. By a well-known result---
see e.g.{} \cite[p.{} 105, Lemma 3.24]{SzV},  \cite[Corollary 4.3]{palais}, or \cite[pp.{} 136--137, Theorem 5.3.8]{ortegarheinboldt}---$\Delta_\sigma$ is a homeomorphism.
\end{proof}

Here is a proof of existence (and even uniqueness) of equioscillation points in a given simplex under the special conditions of this section.
\begin{corollary}\label{cor:equiuniqueC2}
Suppose that for each $j=0,\dots,n$ the kernel $K_j$ belongs to $\Ce^2(0,2\pi)$ with $K_j''<0$ and satisfies \eqref{eq:kernsing}.
Then all equioscillation points belong to some (open) simplex, and in each simplex $S=S_\sigma$ there is a unique equioscillation point.
\end{corollary}
\begin{proof}
An equioscillation point must belong to $X$ according to Corollary \ref{cor:equiosciinside}.
In a fixed simplex $S_\sigma$, an equioscillation point is the inverse image of $\vc{0}\in \RR^n$ under the homeomorphism $\Delta_\sigma$ from Corollary \ref{cor:diffhomeo}.
\end{proof}

\section{Equioscillation points}\label{sec:equioscillation}

In this section we prove the existence of equioscillation points in each simplex $S=S_\sigma$, and discuss the uniqueness of such points. The main tool will be the approximation of kernels by a sequence of kernel functions having special properties, so the arguments rely on the results of Section \ref{sec:approx}.

\begin{lemma}\label{l:equilimit} Suppose that $K_0,\dots, K_n$ are strictly concave kernel functions and that a sequence of strictly concave kernel functions $(K_j^{(k)})_{k\in\NN}$ converges uniformly (e.s.) to $K_j$ as $k\to \infty$, $j=1,\dots,n$. {\piros Let $S=S_\sigma$ be a simplex.}
For each $k\in \NN$ let
$\ee^{(k)}\in \overline{S}$
be an equioscillation point for the system of kernels $K_j^{(k)}$, $j=0,\dots,n$. Then any  accumulation point  $\ee\in \overline{S}$ of the sequence $(\ee^{(k)})_{k\in \NN}$ is an equioscillation point of the system $K_j$, $j=0,\dots,n$.
\end{lemma}
\begin{proof} By passing to a subsequence we may assume that $\ee^{(k)}\to \ee \in \overline{S}$. By assumption and by Proposition \ref{prp:unifconvm}  $m_j^{(k)}\to m_j$
uniformly (e.s.)
on $\overline{S}$ as $k\to \infty$. It follows that $m_j^{(k)}(\ee_k)\to m_j(\ee)$ as $k\to \infty$, so $\ee\in \overline{S}$ is an equioscillation point.
\end{proof}

We need another lemma, similar to \cite[Theorem 1]{Azagra}, in order to be able to apply the previous result.
\begin{lemma}\label{l:convexC2approx} Let $f:[0,1)\to \RR$ be a strictly concave,
non-increasing
function.
Then for each $\varepsilon>0$ there exists another strictly concave
 decreasing
function $g:[0,1)\to \RR$
such that $g\in \Ce^\infty[0,1)$, $g''<0$ on $[0,1)$, and $f(x)-\varepsilon\leq g(x) \leq  f(x)$ for each $x\in [0,1)$.
\end{lemma}
\begin{proof}
This lemma is fairly standard, but for sake of completeness, we include a proof.

Assume, without loss of generality,  that $f(0)=0$.
Let us consider the right (hence right continuous) derivative
$f'_{+}$
of $f$ for our construction:
We can write $f(x)=\int_0^x \phi (t) dt$, where
$\phi(t):=f'_{+}(t)$
and  $\phi: [0,1)\to (-\infty,0]$.

It suffices to construct a $\Ce^\infty$-approximation $\gamma:[0,1)\to (-\infty,0]$ to the
non-increasing function $\phi$, which has non-positive, continuous derivative $\gamma' \in \Ce^\infty[0,1)$, and which satisfies $\gamma(x) \le \phi(x)$ on $[0,1)$ and $\int_0^1 (\phi(x)-\gamma(x))dx <\varepsilon$.
Indeed, then $g(x):=\int_0^x \gamma(t) dt$ is a suitable approximant to $f$.
(If needed, we can easily achieve $g''<0$ by adding $-\eta \cdot (x+1)^2$ to $g$ where $\eta>0$ is small enough,
still satisfying $f(x)-\varepsilon-4\eta\leq g(x) - \eta \cdot (x+1)^2 \leq f(x)$).

Write $\phi(x)= \alpha(x) + \beta(x)$, where $\alpha(x)$ is a pure jump function and $\beta(x)$ is continuous. Both $\alpha$ and $\beta$ are non-increasing.

Approximate $\beta$ with a pure jump function $\beta_1$ such that $\beta_1$ is non-increasing and $\beta(x)-\varepsilon/2 \le \beta_1(x) \le \beta(x)$ for all $x\in[0,1)$.

Consider $\alpha(x)+\beta_1(x)=\sum_{j=1}^\infty s_j H(x-r_j)$,   where $H(x)$ is the usual Heaviside function, $H(x)=1$ for $x\geq 0$ and otherwise zero.
Here  $s_j<0$, $r_j\in [0,1)$, $\sum_{i~:r_i<x} |s_i| <\infty$ (for all $x<1$).
By construction, $\phi(x)- \varepsilon/2  \le \sum_{j=1}^\infty s_j H(x-r_j)\le \phi(x)$.

Take $\psi\in\Ce^{\infty}(\RR)$ with $\psi\ge 0$, $\supp \psi=[-1,0]$, $\int_\RR\psi(t)dt=1$ and define  $\theta(x):=\int_{-\infty}^x \psi(t)dt$.
Consider the translated and dilated versions $\tau_{r,h}(x):=\theta((x-r)/h)$ of $\theta$.
Then  $\tau_{r,h} \in \Ce^\infty[0,1)$ for any $h>0$, and these functions are
non-decreasing, and  $ H(x-r)\le \tau_{r,h}(x) $ with strict inequality  holding precisely for $x\in (r-h,r)$. As a result, we have $\int_0^1 |\tau_{r,h}(x)-H(x-r)| dx\leq h$. Approximate now the constructed pure jump function from below as follows:
\begin{equation*}
\sum_{i=1}^\infty s_i H(t-r_i) \ge \sum_{i=1}^\infty s_i \tau_{r_i,h_i} (t),
\end{equation*}
where both sums are absolutely and uniformly convergent for all $t \in [0,x]$ for any fixed $x<1$, if only we assume $h_i \le \tfrac{1}2(1-{r_i})$. (Indeed, this follows for the first sum by $\sum_{i~:r_i<x} |s_i| <\infty$, while the assumption entails that $r_i-h_i<x \Rightarrow r_i < x + \tfrac{1}2(1-{r_i}) \Leftrightarrow r_i < \frac{2x+1}3 (<1)$, whence the sum $\sum_{i~:r_i-h_i<x} |s_i| \le \sum_{i~:r_i<(2x+1)/3} |s_i|$ also converges.) Furthermore, we also have
\begin{align*}
0&\le \int_0^x \sum_{i=1}^\infty s_i H(t-r_i) - \sum_{i=1}^\infty s_i \tau_{r_i,h_i} (t)\:dt
= \sum_{i=1}^\infty s_i  \int_0^x H(t-r_i) - \tau_{r_i,h_i} (t) \:dt \\
  &\le \sum_{i :~ r_i- h_i<x} \left|s_i\right| h_i <\frac{\varepsilon}{2},
\end{align*}
if we also know that $h_i$ are so small that $\sum_{i=1}^\infty \left| s_i \right| h_i <\varepsilon/2 $. Here we can choose $h_i:= \min \bigl(\tfrac{1}2-\tfrac {r_i}2,2^{-i} \tfrac{\varepsilon}{4 \left|s_i\right|} \bigr)$.

Finally, let $\gamma(x):=\sum_{i=1}^\infty s_i \tau_{r_i,h_i} (t)$.
Then $\phi(x)\ge \gamma(x)$ and $\int_0^1 \gamma(x)-\phi(x)dx < \varepsilon $.
This finishes the proof of this lemma.
\end{proof}
\begin{lemma}
 \label{l:convexC2kernapprox} Let $K$ be a strictly concave kernel function.
Then for each $\varepsilon>0$ there exists {\piros another strictly concave
function} $k\in \Ce^2(0,2\pi)$, $k''<0$ on $(0,2\pi)$, and $K(x)-\varepsilon\leq k(x) \leq  K(x)$ for each $x\in (0,2\pi)$.
\end{lemma}
\begin{proof}
This approximation is indeed possible, for given $\varepsilon>0$ and a given (strictly) concave function $K:(0,2\pi)\to \RR$ satisfying \eqref{eq:kernsing}, we can
choose
the
maximum point $c\in (0,2\pi)$, and consider the intervals 
$[c,2\pi)$ and $(0,c]$ separately: 
applying Lemma \ref{l:convexC2approx}  for $-K((x-c)/(2\pi-c))$ 
and $-K((c-x)/c)$ separately provides an approximating strictly 
concave kernel function $k\in \Ce^2((0,2\pi)\setminus\{c\})$ 
with $k''<0$ and $K-\varepsilon<k<K$. 
By a modification of this kernel function even a smooth 
approximating kernel function, 
as in the assertion, can be easily found.
\end{proof}

\begin{thm}\label{thm:equiexists}
Suppose that for each $j=0,\dots,n$ the kernels $K_j$ are strictly concave. Then for each simplex $S=S_\sigma$ there exists an equioscillation point in $\overline{S}$.

Moreover, if the kernels are either all in $\Ce^1(0,2\pi)$ or at least $n$ of them satisfy \eqref{eq:kernsing'}, then any equioscillation point is in the open simplex $S$.
\end{thm}
\begin{proof}
We split the proof into several steps.

\medskip\noindent{\it Step 1.}  First, let us suppose that all the kernel functions $K_0, \dots, K_n$  satisfy \eqref{eq:kernsing}. By Lemma \ref{l:convexC2kernapprox} we can take  a sequence  $(K_i^{(k)})_{k\in \NN}$ of strictly concave functions in $\Ce^2(0,2\pi)$  satisfying $\frac{d^2}{dt^2} K_i^{(k)}(t)<0$ and converging strongly uniformly (and therefore locally uniformly, too) to the functions $K_i$. Note that hence we also require that $K_j^{(k)}$ satisfy \eqref{eq:kernsing}. 

According to Corollary \ref{cor:equiuniqueC2} each system $K_j^{(k)}$, $j=0,\dots,n$, has a unique equioscillation point $\ee^{(k)}$. By Lemma \ref{l:equilimit} any accumulation point $\ee$ of this sequence (and, by compactness,  there is one) is an equioscillation point. Finally, by Corollary \ref{cor:equiosciinside} an equioscillation point is necessarily inside $S$.
This concludes the proof for the special case when all the kernels satisfy \eqref{eq:kernsing}.

\medskip\noindent {\it Step 2.} Now, let us consider the  case when the kernels are strictly concave but satisfy \eqref{eq:kernsing'pm} only. Let us fix the auxiliary functions $L_k(x):= \log_{-}(k|x|)$, which are concave, even, non-positive functions on $(-\pi,0) \cup (0,\pi)$ with singularity at $0$. We extend these functions to $\RR$ periodically. For $k\in \NN$ and $j=0,\dots, n$
define
 $K^{(k)}_j:=L_k+K_j$.
Then  $K_j^{(k)}\upto K_j$ on $\TT\setminus \{0\}$.
By Step 1, for each $k\in \NN$ there is an equioscillation point $\ee^{(k)}$ for the system $K_j^{(k)}$, $j=0,\dots,n$. By passing to a subsequence we can assume $\ee^{(k)}\to \ee\in \overline{S}$. For $j\in \{0,\dots,n\}$ we have
\begin{equation*}
m_j^{(k)}(\ek) = \max_{t\in I_j(\ek)} F^{(k)}(\ek,t) \leq \max_{t\in I_j(\ek)} F(\ek,t) =m_j(\ek).
\end{equation*}
Since $m_j$ is continuous on $\overline{S}$, we obtain
\begin{equation}\label{eq:mjkekleqmj}
\limsup_{k\to\infty} m_j^{(k)}(\ek)\leq m_j(\ee).
\end{equation}

Suppose first that the arc $I_j(\ee)$ is non-degenerate for all $j=0,1,\dots,n$, i.e., assume $\ee\in S$. Then Proposition \ref{prop:zjinter} (d) yields $z_j(\ee)\in\intt I_j(\ee)=(e_j,e_r)$, so for sufficiently large $k$ we have $z_j(\ee)\in\intt I_j(\ee^{(k)})$, too; furthermore,
since by construction $K_j(t)=K_j^{(k)}(t)$ for $t\not\in[-\frac1k,\frac1k]$, for sufficiently large $k$ we even have $e^{(k)}_j+1/k < z_j(\ee) < e^{(k)}_r-1/k$, whence $F^{(k)}(\ek,z_j(\ee)) = F(\ek,z_j(\ee))$, too. Therefore we obtain
\begin{equation*}
m_j^{(k)}(\ek) = \max_{t\in I_j(\ek)} F^{(k)}(\ek,t)\geq F^{(k)}(\ek,z_j(\ee)) = F(\ek,z_j(\ee)).
\end{equation*}
This implies
\begin{equation}\label{eq:mjkekgeqmj}
\liminf_{k\to\infty}m_j^{(k)}(\ek) \geq \liminf_{k\to\infty} F(\ek,z_j(\ee))=F(\ee,z_j(\ee))=m_j(\ee).
\end{equation}
So the proof of Step 2 is complete if $\ee\in S$.

Finally, we show that $\ee\in \partial S$ is impossible.
Indeed, if there is a degenerate arc $I_j(\ee)$, then by Corollary \ref{cor:neighborjump} there is a neighboring non-degenerate arc $I_i(\ee)$ such that $m_i(\ee)>m_j(\ee)$. But then we are led to a contradiction, because using \eqref{eq:mjkekleqmj} and \eqref{eq:mjkekgeqmj} we also have
\[
m_j(\ee)\geq \limsup_{k\to\infty} m_j^{(k)}(\ek)
\geq 
\liminf_{k\to\infty} m_j^{(k)}(\ek)
=
\liminf_{k\to\infty} m_i^{(k)}(\ek)
\geq 
m_i(\ee),
\]
taking into account the equioscillation of $m^{(k)}$ at $\ek$. 

\medskip\noindent
{\it Step 3.} Finally, we suppose only that $K_0,\dots, K_n$  are strictly concave kernel functions.
 We now take the functions $L_k(x):=(\sqrt{|x|}-1/k)_{-}$, which are negative only for $-1/k^2<x<1/k^2$ and zero otherwise, and converge uniformly to zero. Restricting $L_k$ to $[-\pi,\pi)$ and then extending it periodically we thus obtain a function on $\TT$ which is concave on $(0,2\pi)$ and converges to
 $0$ uniformly on $[0,2\pi]$.
 Note that
 $\lim_{x\to 0\pm 0} L'_k(x) = \pm \infty$,
 hence the perturbed kernels $\Kkj:=K_j+L_k$, $j=0,\dots,n$, satisfy \eqref{eq:kernsing'pm}.
Again, in view of the already proven case in Step 2,
there exist some equioscillation points $\ek$ for the system $\Kkj$, $j=0,\dots,n$, and by compactness, there exists an accumulation point $\ee \in \overline{S}$ of the sequence $(\ek)_{k\in \NN}$.
By uniform convergence of the kernels we can apply  Lemma \ref{l:equilimit} to conclude that $\ee$ is an equioscillation point of the system $K_j$, $j=0,\dots,n$.

\medskip\noindent
It remains to prove that  $\ee\in S$ if the additional assumptions are fulfilled, but this
has already been done
in Corollary \ref{cor:equiosciinside}.\end{proof}

\begin{corollary}\label{cor:Mleqm} Let the kernel functions $K_0,\dots,K_n$ be strictly concave. Then in any simplex $S=S_\sigma$ the Equioscillation Property holds, and we have $M(S)\leq m(S)$.
\end{corollary}

\begin{corollary}\label{cor:Meqm} Let the kernel functions $K_0,\dots,K_n$ be strictly concave and let $S=S_\sigma$ be a simplex.
	Suppose that $M(S)=m(S)$. Then there is $\ws\in \overline{S}$ with $m(S)=\mm(\ws)$ and $\ws$ is the unique equioscillation point in $\overline{S}$.
\end{corollary}
\begin{proof}
Let $\ee\in \overline{S}$ be an equioscillation point (see Corollary \ref{cor:Mleqm}), and let $\ws\in \overline{S}$ be such that $\mm(\ws)=m(S)$ (see  Proposition \ref{prop:exists}). 
Because $\mm(\ee)=\MM(\ee)\geq M(S)=m(S)=\mm(\ws)$,
we  find
that $\ee$ is also a maximum point of $\mm$, and that $\mm(\ee)=M(S)$. By Corollary \ref{cor:muniquenomaj}
(a),
$\ee=\ws$, and by  $M(S)=m(S)$ and in view of
Proposition \ref{prop:whatifequinon-unique}
(a),
the equioscillation point is unique.
\end{proof}

\section{{\piros Proof of Theorem \ref{thm:0mainspecialcase}, some consequences and conclusions}}
\label{sec:summary}
{\piros For the sake of better legibility we recall here Theorem \ref{thm:0mainspecialcase} from Section \ref{sec:intro}
by using the terminology introduced in the previous sections. Then we discuss the sharpness of the result and draw some further consequences.}

\begin{thm}\label{thm:mainspecialcase}
	Suppose the kernel functions $K_0,K_1,\dots,K_n$ are strictly concave and either all
satisfy \eqref{eq:kernsing'}, or all belong to $\Ce^1(0,2\pi)$.
Then there is $\wS\in \TT^n$, $\wS=(w_1,\dots,w_n)$ with
\begin{align*}
M:=\inf_{\yy\in\TT^n}\sup_{t\in \TT}F(\yy,t)=\sup_{t\in \TT}F(\wS,t).
\end{align*}
Moreover, we have the following:
\begin{abc}
	\item $\wS$ is an equioscillation point, i.e., $m_0(\wS)=\cdots=m_n(\wS)$.
	\item $\wS\in S:=S_\sigma$ for some simplex, i.e., the nodes in $\wS$ are different, and
\begin{align*}
M(S)=\inf_{\yy\in S}\max_{j=0,\dots,n}\sup_{t\in I_j(\yy)} F(\yy,t)=M=\sup_{\yy\in S}\min_{j=0,\dots,n}\sup_{t\in I_j(\yy)} F(\yy,t)=m(S).
\end{align*}
\item We have the Sandwich Property on $S$, i.e., for each $\xx,\yy\in S$
\[
\mm(\xx)\leq M\leq \MM(\yy).
\]
\end{abc}
\end{thm}

\begin{proof}
In view of Corollary \ref{cor:minmaxexist}, a global minimum point $\ww^*$ of $\MM$ must exist.
Next, Corollary \ref{cor:ezkell} furnishes part (a) and $\ww^*\in X$, i.e. the first half of (b).
Finally, Proposition \ref{prop:Mmextremal}
implies the second half of (b) and the assertion in (c).
\end{proof}

\begin{example}\label{examp:diffmvlaues}\label{examp:last2}
We present an example showing that on different simplexes we may have different
 values of $M$. {\lila This will be done in several steps, and we begin with considering the functions }
\begin{align*}
K(x)&:=\pi-|x-\pi| &&\quad \text{for $x\in [0,2\pi]$},\\
Q(x)&:=x(2\pi-x)&&\quad \text{for $x\in [0,2\pi]$},
\end{align*}
and extend them periodically to $\RR$.
We take $K_0=K_1=K$ and $K_2=K_3=\varepsilon Q$ where
$\varepsilon\in (0,\frac14)$
is fixed arbitrarily.  {\lila This is not yet the system of kernels that we are looking for, but they will serve as a basis for the construction. }

Note that this system of kernels almost satisfies
the conditions of Theorem \ref{thm:mainspecialcase}:
two kernels satisfy \eqref{eq:kernsing'pm} and all the  kernels are in $\Ce^1((0,2\pi)\setminus\{\pi\})$,
and the two not satisfying \eqref{eq:kernsing'pm} are even in
$\Ce^1(0,2\pi)$
(which, again, could have been enough if satisfied by all).

We consider two simplexes $S=S_\sigma$ for $\sigma=(2,1,3)$ and $S'=S_{\sigma'}$ with $\sigma'=(3,2,1)$. We prove that there is an equioscillation point $\ee\in S$ and for this equioscillation point we have
$\MM(\ee)>\MM(S')$.
This will be done
first in
two steps below, then in Step 3 we shall take an appropriate sequence of kernel functions $K_j^{(k)}$ converging to $K_j$ ($j=0,1,\dots,n$) and obtain
\[
M^{(k)}(S)>M^{(k)}(S')
\]
{\lila as required.}

\subsubsection*{Step 1.} We take the node system
$\ee:$
$e_0=0$, $e_1=\pi$, $e_2=\frac{\pi}2$, $e_3=\frac{3\pi}2$.
Then we have $\ee\in S$ and
\[
F(\ee,t)
=K_0(t)+
K_1(t-e_1)+K_2(t-e_2)+K_3(t-e_3)
=\pi+\varepsilon Q(t-\tfrac\pi 2)+\varepsilon Q(t-\tfrac{3\pi} 2).
\]
It is easy to see that
\[
m_0(\ee)=F(\ee,0)=\max_{t\in [0,\frac\pi 2]}F(\ee,t)=\pi+3\varepsilon \tfrac{\pi^2}{2},
\]
and by symmetry $m_0(\ee)=m_1(\ee)=m_2(\ee)=m_3(\ee)$, i.e., $\ee$ is an equioscillation point.

\subsubsection*{Step 2.} Consider the node system $x_0=0$, $x_1=\pi+(3-2\sqrt{2})\varepsilon \pi^2$, $x_2=(2\sqrt{2}-2)\pi$, $x_3=0$. Then of course $\xx\in\overline{S'}\cap \overline{S}$.
It is easy to see that
\begin{equation*}
F(\xx,t)=\begin{cases}
-2\varepsilon t^{2}+2(1+\varepsilon x_{2})t -\varepsilon x_2^2+2 \pi  (\varepsilon x_2+1)-x_1, &
\mbox{ if }0\le t\le x_1-\pi,\\
-2\varepsilon t^{2}+2\varepsilon x_2 t-\varepsilon x_{2}\left(-2\pi+x_{2}
\right)+ x_{1},
&
\mbox{ if }x_1-\pi\le t\le x_2,\\
-2\varepsilon t^{2}+2\varepsilon\left(2\pi+x_2\right)t-\varepsilon x_{2}\left(2\pi+x_{2}\right)+x_{1}, &
\mbox{ if }x_2\le t\le\pi,\\
-2\varepsilon t^{2}+2(\varepsilon x_{2}+2\varepsilon\pi-1)t - \varepsilon x_{2}(2\pi+x_2) + x_1+2\pi, &
\mbox{ if }\pi\le t\le x_1,\\
-2\varepsilon t^{2}+2\varepsilon\left(2\pi+x_{2}\right)t-\varepsilon x_2 (2\pi +x_2)-x_1+2\pi, &
\mbox{ if }x_{1}\le t\le 2\pi.
\end{cases}
\end{equation*}
 For definiteness of indexing, let us consider the node system $\xx$ as an element of the simplex $S'$ where $\sigma'=(3,2,1)$.

Now, an easy but tedious computation leads to the following.
The maximum of $F(\xx,\cdot)$ on
 $I_0(\xx)=[x_0,x_3]=[0,0]$
is
\begin{align*}
m_0(\xx)&=F(\xx,0)=\pi+\varepsilon\pi^2(14\sqrt{2}-19),\\
\intertext{the maximum of  $F(\xx,\cdot)$ on $I_1(\xx)=[x_1,2\pi]$ is attained at $z_1(\xx)=\pi+\tfrac{x_2}{2}$ and }
m_1(\xx)&=F(\xx,\pi+\tfrac{w_2}{2})=\pi+\varepsilon\pi^2(6\sqrt{2}-7),
\intertext{the maximum of  $F(\xx,\cdot)$ on $I_2(\xx)=[x_2,x_1]$ is attained at $z_2(\xx)=\pi$ and }
m_2(\xx)&=F(\xx,\pi)=\pi+\varepsilon\pi^2(6\sqrt{2}-7),
\intertext{the maximum of  $F(\xx,\cdot)$ on $I_3(\xx)=[x_3,x_2]=[0,x_2]$ is attained at $z_3(\xx)=\tfrac{x_2}{2}$ and }
m_3(\xx)&=F(\xx,\tfrac{x_2}{2})=\pi+\varepsilon\pi^2(6\sqrt{2}-7).
\end{align*}
From this we conclude
\[
\MM(\xx)=\pi+\varepsilon\pi^2(6\sqrt{2}-7)<\pi+3\varepsilon \tfrac{\pi^2}{2}=\MM(\ee),
\]
and hence
\[
 M(S),\:\:
M(S')\leq \MM(\xx)<\MM(\ee).
\]
Note that the equioscillation point $\ee\in S$ thus cannot be a minimum point of $\MM$ on the simplex $S$,
while $\xx\in\overline{S}\cap\overline{S'}$ is a weak equioscillation point on the boundary of both simplexes.

\subsubsection*{Step 3.} Now, let
\[
K_j^{(k)}(x):=K_j(x)+\frac{1}{k}\sqrt{\pi^2-(x-\pi)^2},
\]
for $j=0,1,2,3$. Then
$K_0^{(k)}$, $K_1^{(k)}$, $K_2^{(k)}$, $K_3^{(k)}$ are strictly concave,  symmetric, satisfying the condition \eqref{eq:kernsing'pm} and
\[
K_j^{(k)}\to  K_j\quad \text{uniformly as $k\to \infty$ for $j=0,1,2,3$}.
\]
Since the configuration of the kernel functions for the simplex $S$ is symmetric and the node system $\ee$ is symmetric, it is easy to see that $\ee$ is an equioscillation point in $S$ also in the case of  the kernels $K_j^{(k)}$.
By Proposition \ref{prp:unifconvm} we have $M^{(k)}(S)\to M(S)$, $m^{(k)}(S)\to m(S)$ and $m_j^{(k)}(\ee)\to m_j(\ee)$ as $k\to \infty$.
Let ${\wS}^{(k)}\in \overline{S}$ be such that $M^{(k)}(S)=\MM({\wS}^{(k)})$.

Now if for some $k\in \NN$ we have $\MM^{(k)}(\ee)\neq M^{(k)}(S)$, then
${\wS}^{(k)}\in \partial S$  (by Proposition \ref{prop:Mmextremal})
and $m^{(k)}(S)\geq\mm^{(k)}(\ee)>M^{(k)}(S)$.
By Corollary \ref{cor:SS'better} (d) we have then $M^{(k)}(S'')<M^{(k)}(S)$ for some neighboring simplex $S''$.
Since by symmetry there are basically two simplexes, we must have $M^{(k)}(S'')=M^{(k)}(S')$ (recall $S'=S_{\sigma'}$ for $\sigma'=(3,2,1)$).
Therefore
\[
M^{(k)}(S)>M^{(k)}(S').
\]

{\piros On the other hand, we cannot have $\MM^{(k)}(\ee)=
M^{(k)}( S )$ for all $k\in \NN$, because then for all large $k$
\[
M^{(k)}(S)=\MM^{(k)}(\ee)>\MM^{(k)}(\xx)
\]
would hold, and that is impossible by $\xx\in \overline{S}$.}

We sum up what has been found in this example: There are strictly concave kernel functions $K_j^{(k)}$, $j=0,1,2,3$ satisfying \eqref{eq:kernsing'pm},  and there are two simplexes $S$ and $S'$ such that $M^{(k)}(S)>M^{(k)}(S')$.
\end{example}

	The phenomenon observed in the previous example can be present also for strictly concave kernels with the \eqref{eq:kernsing} property.
\begin{example}\label{examp:last}
Consider some symmetric kernel functions $K_0$, $K_1$, $K_2$, $K_3$ satisfying \eqref{eq:kernsing'pm} with $M(S_{\sigma})> M(S_{\sigma'})$ (see the previous Example \ref{examp:last2}). Let $L$ be a strictly concave, symmetric kernel function with \eqref{eq:kernsing}, and consider $K_j^{(k)}:=\frac1kL+K_j$, $j=0,\dots,3$.
Then, as in Example \ref{examp:last2}, by means of Proposition \ref{prp:unifconvm} we obtain 
\begin{equation*}
M^{(k)}(S_{\sigma})> M^{(k)}(S_{\sigma'})
\end{equation*}
for large $k$.
\end{example}

\begin{example} It can happen that $M(\TT^3)<m(\TT^3)$.

Indeed, let $K_0$, $K_1$, $K_2$, $K_3\in \Ce^2(0,2\pi)$  be strictly concave symmetric kernel functions  satisfying \eqref{eq:kernsing} with
\[
M(S_{\sigma})> M(S_{\sigma'})
\]
for different simplexes $S_\sigma$ and $S_{\sigma'}$.
Consider, for example, the situation of the preceding Example \ref{examp:last}.

Let $\wS\in \TT^3$ be a global minimum point of $\MM$ on $\TT^3$. Let $S_{\sigma''}$ denote the simplex in which $\wS$ lies. We then have
\[
M(\TT^3)=m(S_{\sigma''})=M(S_{\sigma''})\leq M(S_{\sigma'})<M(S_{\sigma})\leq m(S_{\sigma})
\]
by Theorem \ref{thm:mainspecialcase} (b) and by Corollary \ref{cor:Mleqm}.
This implies $  M(\TT^3)<m(\TT^3) $.
\end{example}

Next, let us discuss the case when all but one kernel functions are the same. This is  analogous to the setting of Fenton \cite{Fenton} in the interval case. Under these circumstances the phenomenon in the previous example is not present anymore. We first need the next lemma, whose similar versions have appeared already in \cite{Fenton} and \cite{Saff}.
\begin{lemma}\label{lem:szethuzas}
Let $K$ be strictly concave and let $a,b>0$, 
{\lila $0<x\le y <2\pi $} 
be given. {\lila Then for  $0<\delta<\min\{\frac{x}{b},\frac{2\pi-y}{a}\}$ we have}
		\[
	\frac{1}{a}K(t-(y+ah))+\frac1bK(t-(x-bh))<\frac1aK(t-y)+\frac1b K(t-x)
	\]
for each  $t\in (0,x-b\delta)\cup (y+a\delta, 2\pi)$ and each $0<h<\delta$.
	\end{lemma}
\begin{proof}
		By strict concavity the difference quotients of  $K$ are strictly decreasing in both variables, so that for all $h\in (0,\delta)$ and $t\in (0,x-b\delta)$ or $t\in (y+a\delta,2\pi)$
\begin{equation*}
\frac{K(t-x+bh)-K(t-x)}{bh}
<
\frac{K(t-y)-K(t-y-ah)}{ah}.
\end{equation*}
			But this inequality is equivalent to the assertion.
		\end{proof}

\begin{thm}\label{thm:mainspecialcase2}
Suppose the kernel functions $L,K$ are strictly concave and either  $K$ satisfies \eqref{eq:kernsing'} or both $K$ and $L$ belong to $\Ce^1(0,2\pi)$. Set
\[
F(\yy,t):= L(t)+ \sum_{j=1}^n K(t-y_j).
\]
Then there is an up to permutation unique $\wS\in \TT^n$, $\wS=(w_1,\dots,w_n)$ with
\begin{align*}
M:=\inf_{\yy\in\TT^n}\sup_{t\in \TT}F(\yy,t)=\sup_{t\in \TT}F(\wS,t).
\end{align*}
Moreover, we have the following:
\begin{abc}
	\item The nodes $w_0,\dots,w_n$ are different and $\wS$ is an equioscillation point, i.e.,
	\[
	m_0(\wS)=\cdots=m_n(\wS).
	\]
	\item We have
\begin{align*}
M=\inf_{\yy\in \TT^n}\max_{j=0,\dots,n}\sup_{t\in I_j(\yy)} F(\yy,t)=\sup_{\yy\in\TT^n}\min_{j=0,\dots,n}\sup_{t\in I_j(\yy)} F(\yy,t)=m.
\end{align*}
{\piros (Here it is immaterial that for a given $\yy$ which permutation $\sigma$ is taken with $\yy\in \overline{S_\sigma}$.)}
\item We have the Sandwich Property on $\TT^n$, i.e.,
for each $\xx,\yy\in \TT^{n}$
\[
\mm(\xx)\leq m=M \leq \MM(\yy).
\]
\item If $K$ is as in the above and $L=K$, then a permutation of the points
$w_0=0,w_1,\dots,w_n$
lies equidistantly in $\TT$.
\end{abc}
\end{thm}
\begin{proof} First of all, notice that assertion (d) is obvious by the complete symmetry of the setup. Furthermore, again by the cyclic symmetry of the situation, even if $K \ne L$, we still have for any two simplexes $S_\sigma$ and $S_{\sigma'}$ that $M(S_\sigma)=M(S_{\sigma'})=M$ and $m(S_\sigma)=m(S_{\sigma'})=m$. Thus, if $L$ and $K$ satisfies \eqref{eq:kernsing'}, or if both belong to $\Ce^1(0,2\pi)$, existence, uniqueness, and the assertions (a), (b) and (c) are contained in Theorem \ref{thm:mainspecialcase}.
	
It remains to prove parts (a), (b) and (c) in the case when $K$ satisfies \eqref{eq:kernsing'} while $L$ does not, so that $L$ is a real-valued continuous function on $\TT$. Without loss of generality we may assume that $K$ satisfies \eqref{eq:kernsing'm}.

Let $\wS=(w_1,\dots,w_n)$ be a global minimum point of $\MM$ in $\TT^{n}$ (Corollary \ref{cor:minmaxexist}). We first show that $\wS\in X$, i.e., $\wS\in S$ for some simplex $S$. 
{\piros We argue by contradiction and assume that $\wS\in \TT^n\setminus X$, i.e. $\wS\in \partial S_\sigma$ for some permutation $\sigma$, which is now fixed for the numbering of the nodes.}

As the kernels $K_i=K$ satisfy {\lila \eqref{eq:kernsing'm}}
for $i=1,\dots,n$, Lemma \ref{lem:new} (b) 
immediately provides $M > F(\wS, w_i)$ for each $w_i$, $i=1,\dots,n$. Now if $\wS \in \partial S_\sigma$ is such that $w_i=w_0=0$ for some $i\in \{1,2,\dots,n\}$, then we also have $M > F(\wS, w_0)$, and so for any maximum point $z$ of $F(\wS,\cdot)$ we necessarily have $z \in \TT\setminus \{w_0,w_1,w_2,\dots,w_n\}$. That is, for the unique local maximum points 
{\piros $z_{j_i}(\wS) \in I_{j_i}(\wS)$} with 
{\piros $M=m_{j_i}(\wS)=F(\wS,z_{j_i}(\wS))$}, where $i=1,\dots,k$, neither of these points can be endpoints of the respective 
{\piros $I_{j_i}(\wS)$}, and so they are all located in the interior of the respective arcs. Note that by assumption $\wS \in \partial S_\sigma$, hence there are at most $n$ non-degenerate arcs, so $k\le n$ and the Perturbation Lemma \ref{lem:pertnodes00} (c) applies. This provides us some perturbation of the node system $\wS$ to a new node system $\ww'$ with all the maxima $m_{j_i}(\ww')<M$ ($i=1,\dots,k$). As the other arcs had maxima strictly below $M$, and in view of continuity (Proposition \ref{prop:cont}), altogether we would get $\MM(\ww')<M$, a contradiction.

Therefore, it remains to settle the case when there is no $i\ge 1$ with $w_i=w_0$ (but we still have $\wS \in \partial S_\sigma$). 
So assume that $(0=w_0<) w_{\sigma(j)}=\dots=w_{\sigma(j+k)}(<2\pi)$ 
is a complete list of $k+1$ 
coinciding nodes within $(0,2\pi)$. 
As before, {\lila in view of  condition \eqref{eq:kernsing'm}  }
Lemma \ref{lem:new} (b) applies 
providing $M > F(\wS,w_{\sigma(j)})$. 
Consider now the perturbed system $\ww'$ 
obtained from $\wS$ by means of slightly pulling apart $w_{\sigma(j)}$ and $w_{\sigma(j+k)}$, 
i.e. taking $w'_{\sigma(j)}:=w_{\sigma(j)}-h$ 
and $w'_{\sigma(j+k)}=w_{\sigma(j+k)}+h$ (and leaving the other nodes unchanged).
Referring to Lemma \ref{lem:szethuzas} with $a=b=1$, we obtain for small enough $h>0$ that $F$ is strictly decreased in $\TT\setminus (w_{\sigma(j)}-h,w_{\sigma(j)}+h))$, whence even $\max_{\TT\setminus (w_{\sigma(j)}-h,w_{\sigma(j)}+h)} F(\ww',t) <M$, while in the missing interval of length $2h$ continuity of $F$ and $M > F(\wS,w_{\sigma(j)})$ entails $\max_{[w_{\sigma(j)}-h,w_{\sigma(j)}+h]} F(\ww',t) <M$. 
Altogether, we are led to $\MM(\ww') < M$, a contradiction again. 
{\piros This proves that $\wS \in X$, i.e., belongs to the interior of some simplex.}

\medskip\noindent Now, by Theorem \ref{thm:equiexists} there is an equioscillation point $\ee\in S$, which certainly majorizes $\wS$. By Corollary \ref{l:WeakLocalM} {\lila(a)} we obtain $\wS=\ee$. This proves (a). Let $\ws$ be a maximum point of $\mm$ in $\overline{S}$. Then, $\ws$ majorizes the equioscillation point $\wS$, so again Corollary \ref{l:WeakLocalM} {\lila(a)} yields $\wS=\ws$. This proves (b) and (c).
\end{proof}

\section{An application: A minimax problem on the torus}\label{sec:HKSmini}

The aim of this section is to prove the next result, which generalizes Theorem \ref{th:ABE-HKS} of Hardin, Kendall and Saff from \cite{Saff} in the extent that we do not assume the kernels to be even. We also add some extra information about the extremal node system: It is the unique solution of the dual maximin problem.

\begin{corollary}\label{cor:HKSgen}
	Let $K$ be any concave kernel function, and let $0=e_0<e_1<\cdots<e_n$ be the equidistant node system
	in $\TT$.
	Consider $F(\yy,t)= K(t)+ \sum_{j=1}^n K(t-y_j)$.
	\begin{abc}
		\item For $\ee=(e_1,\dots,e_n)$ we have
		\[
		\max_{t\in \TT} F(\ee,t)=M=\inf_{\yy\in\TT^n}\max_{t\in \TT} F(\yy,t),
		\]
		i.e., $\ee$ is a minimum point of $\MM$. Moreover,
		\[
		\inf_{\yy\in\TT^n}\max_{j=0,\dots,n} m_j(\yy)=
		 M=m=\sup_{\yy\in\TT^n}\min_{j=0,\dots,n} m_j(\yy).
		 \]
		
		\item If $K$ is strictly concave, then $\ee$ is the unique (up to permutation of the 
nodes) maximum point of $\mm$ and the unique minimum point of $\MM$.
		\end{abc}
\end{corollary}
\begin{proof}
Since the permutation of the nodes is irrelevant we may restrict the consideration to the simplex $S:=S_{\id}$, where $\id$ is the identical permutation. We have $M=M(S)$ and $m=m(S)$.
	
	\medskip\noindent
	(a)
Approximate $K$ uniformly by
strictly concave
kernel functions $K^{(k)}$ satisfying \eqref{eq:kernsing'pm} (cf.{} Example \ref{examp:last}).
By Theorem \ref{thm:mainspecialcase2},  $M^{(k)}=\MM^{(k)}(\mathbf{e})$  and $M^{(k)}=m^{(k)}$
and obviously we have $M^{(k)}=M^{(k)}(S)$, $m^{(k)}=m^{(k)}(S)$.
By Proposition \ref{prp:unifconvm} we have $M^{(k)}(S)\to M(S)=M$, $m^{(k)}(S)\to m(S)=m$, $\mm^{(k)}(\ee)\to \mm(\ee)$ and  $\MM^{(k)}(\ee)\to \MM(\ee)$. So $\MM(\ee)=M=M(S)=m(S)=m$.

\medskip\noindent (b) Let $\wS\in \overline{S}$ be a minimum point of $\MM$.
 If $m_j(\wS)<\MM(\wS)=M(S)$ held for some $j\in \{0,1,\dots,n\}$, then by an application of Lemma \ref{lem:szethuzas}
 (with $a=b=1$ there)  and Corollary \ref{cor:mjsimplexcont} we could arrive at a new node system $\ww'$ with $\MM(\ww')<\MM(\wS)$, which is impossible. We conclude therefore that $\wS$ is an equioscillation point. Since by part (a) we have $m(S)=M(S)$, the equioscillation point is unique by
Proposition \ref{prop:whatifequinon-unique} (a).
 Hence $\wS=\ee$, and uniqueness follows.
\end{proof}

\section{An application: Generalized polynomials and Bojanov's result}
\label{sec:bojanovext}
{\piros
In this section we present two  applications of the previously developed theory to Chebyshev type problems for generalized polynomials and generalized trigonometric polynomials, thereby refining some results of Bojanov \cite{Bojanov} in the polynomial situation (see Theorem \ref{thm:bojanov} below), and proving the analogue of this generalization in the trigonometric situation.}

We shall use the following form of our main theorem.
\begin{thm}\label{thm:mainspecialcase3}
Suppose the kernel function $K$ is strictly concave and either satisfies {\eqref{eq:kernsing'}, or is in $\Ce^1(0,2\pi)$}.
Let
$r_0,r_1,\ldots,r_n>0$,
set
$K_j:=r_j K$
and
\[
F(\yy,t):= K_0(t)+ \sum_{j=1}^n  K_j(t-y_j) =
r_0 K(t)+ \sum_{j=1}^n r_j K(t-y_j).
\]
Let $S=S_\sigma$ be a simplex.
Then there is a unique $\wS\in S$, $\wS=(w_1,\dots,w_n)$ with
\begin{align*}
M(S):=\inf_{\yy\in S}\sup_{t\in \TT}F(\yy,t)=\sup_{t\in \TT}F(\wS,t).
\end{align*}
Moreover, we have the following:
\begin{abc}
	\item The nodes $w_0,\dots,w_n$ are different and $\wS$ is an equioscillation point, i.e.,
	\[
	m_0(\wS)=\cdots=m_n(\wS).
	\]
	\item We have
\begin{align*}
\inf_{\yy\in S}
\max_{j=0,\dots,n}\sup_{t\in I_j(\yy)} F(\yy,t)=M(S)=m(S)=
\sup_{\yy\in S}
\min_{j=0,\dots,n}\sup_{t\in I_j(\yy)} F(\yy,t).
\end{align*}
\item We have the Sandwich Property
 in  $\overline{S}$,
i.e., for each $\xx,\yy\in \overline{S}$
\[
\mm(\xx)\leq
M(S)
\leq \MM(\yy).
\]
\end{abc}
\end{thm}

\begin{proof}
	There is $\ww\in\overline{S}$ with $M(S)=\sup_{t\in \TT}F(\ww,t)$. 	
	By Proposition \ref{prop:Mmextremal} we only need to prove
that $\ww$ belongs
to the interior of the simplex,  i.e., $\ww\in S$. Suppose by contradiction that
$w_{\sigma(k-1)} \le
w_{\sigma(k)}=w_{\sigma(k+1)}=\cdots=w_{\sigma(\ell)}<
 2\pi=
w_{\sigma{(n+1})}$ with $k\neq \ell$, $k\in\{1,\dots,n\}$
(the case $k=0$ will be  considered below separately).
Then we can apply Lemma \ref{lem:szethuzas}
with
$a=\frac1{r_{\sigma(\ell)}}$, $b=\frac1{r_{\sigma(k)}}$
and
$x=w_{\sigma(k)}$,   $y=w_{\sigma(\ell)}$,
and move the two nodes
$w_{\sigma(k)}$ and $w_{\sigma(\ell)}$
away from each other, such that
 the
new node system $\ww'$ still belongs to {\lila $\overline{S}$}.  We conclude
\begin{align*}
&F(\ww',t)-F(\ww,t)\\
&=K_{\sigma(k)}(t-w_{\sigma(k)}')+K_{\sigma(\ell)}(t-w_{\sigma(\ell)}')-K_{\sigma(k)}(t-w_{\sigma(k)})-K_{\sigma(\ell)}(t-w_{\sigma(\ell)})<0 
\end{align*}
for all
$t\in \TT\setminus [w_{\sigma(k)}',w_{\sigma(\ell)}']$.
Hence  we obtain
\begin{equation}
\label{cond:mjwpmjw}
m_j(\ww')<m_j(\ww)\quad\text{for each } j\in \{0,\dots,n\}\setminus
 \{\sigma(k),\dots, \sigma(\ell-1)\}.
\end{equation}
Since by Corollary \ref{cor:neighborjump}
$m_{\sigma(k)}(\ww)=m_{\sigma(k+1)}(\ww)=\cdots=m_{\sigma(\ell-1)}(\ww)<\MM(\ww)$,
if we move the two nodes
$w_{\sigma(k)}$ and $w_{\sigma(\ell)}$
by
a sufficiently small amount, by Corollary \ref{cor:mjsimplexcont} we can achieve
\begin{equation}
\label{cond:mjcsokk}
m_{\sigma(k)}(\ww'),\
m_{\sigma(k+1)}(\ww'),\
\cdots,\
m_{\sigma(\ell-1)}(\ww')
<\MM(\ww).
\end{equation}
Putting together \eqref{cond:mjwpmjw} and \eqref{cond:mjcsokk},
we would obtain $\MM(\ww')<\MM(\ww)$, which is in contradiction with the choice of $\ww$.

If
finally, {\lila $k=0$, that is }
$w_0$ happens to coincide with some
$w_{\sigma(\ell)}$,
then we can move $w_0$ and
$w_{\sigma(\ell)}$
away from each other as above and obtain a new node system $w_0'\in \TT$,  $\ww'=(w_1',\dots, w_n')$ with $\MM(\ww')<\MM(\ww)$, and then we need to rotate back {\lila all the nodes} by $w'_0$.

 We have seen that $\wS:=\ww\in S$, therefore the proof is complete.
	\end{proof}

Bojanov proved in \cite{Bojanov} the following.
\begin{thm}[(Bojanov)]\label{thm:bojanov}
Let $\nu_{1},\ldots,\nu_{n}$ be fixed positive integers. Fix $[a,b]\subset\RR$.
Then, there exists a unique system of points
$a<x_{1}<\ldots<x_{n}<b$ such that
\[
\| (x-x_{1})^{\nu_{1}}\ldots(x-x_{n})^{\nu_{n}}\| =\inf_{a\leq y_{1}<\ldots<y_{n}\leq b}\| (x-y_{1})^{\nu_{1}}\ldots(x-y_{n})^{\nu_{n}}\|
\]
where $\|\cdot\| $ denotes the sup-norm over $[a,b]$.
The extremal polynomial
\[
P^{*}(x):=(x-x_{1})^{\nu_{1}}\ldots(x-x_{n})^{\nu_{n}}
\]
is uniquely characterized by the property that there exist $a=s_{0}<s_{1}<\ldots<s_{n-1}<s_{n}=b$
such that $|P^{*}(s_{j})|=\| P^{*}\| $  for
$j=0,1,\ldots,n$.
 Moreover, in this situation
 \[
 P^{*}(s_{j+1})=(-1)^{\nu_{j+1}}P^{*}(s_{j})\quad \text{for $j=0,1,\ldots,n-1$.}
 \]
\end{thm}
Now, we are going to establish a similar result for trigonometric polynomials
and relate this new result to Bojanov's theorem.

It is well known (see e.g.{} \cite{BorweinErdelyi} p.{} 10) that a trigonometric
polynomial \[T(t)=a_{0}+\sum_{k=1}^{m}a_{k}\cos(kt)+b_{k}\sin(kt)\]
where $|a_{m}|+|b_{m}|>0$,  can be written in the form
$T(t)=c\prod_{j=1}^{2m}\sin\frac{t-t_{j}}{2}$ where $c$,
$t_{1},\ldots,t_{2m}$ are numbers. More precisely, if $T(t')=0$,
$t'\in\CC$, $\Re t'\in[0,2\pi)$, then $t'$ appears in
$t_{1},\ldots,t_{2m}$ and if $a_{0},a_{1},b_{1},\ldots,a_{m},b_{m}\in\RR$
and $T(t')=0$, $t'\in\CC\setminus\RR$, $\Re t'\in[0,2\pi)$,
then the conjugate of $t'$ is also a zero, $T(\overline{t'})=0$
and both appear among $t_{1},\ldots,t_{2m}$.

Functions of the form
\[
a\prod_{j=1}^{m}\Bigl|\sin\frac{t-t_{j}}{2}\Bigr|^{r_{j}},
\]
where $a,r_{j}>0$, $t_{j}\in\CC$ for all $j=1,\ldots,m$, are called \emph{generalized trigonometric
polynomials} (GTP for short), see, e.g.
\cite{BorweinErdelyi} Appendix 4. The number $\frac{1}{2}\sum_{j=1}^{m}r_{j}$ is usually
called the degree of this GTP.

In the next theorem, we describe {\lila Chebyshev type  extremal GTPs 
(having minimal sup
norm and fixed leading coefficient)}
when the multiplicities of
the zeros are fixed and the zeros are real.
Let us mention a  {\lila related } result 
of Kristiansen (see \cite[Thm.{} 2]{Kristiansen},  which is also mentioned in \cite{BojanovTuran} as Theorem B) concerning trigonometric polynomials with prescribed multiplicities of zeros.
{\lila However, the paper \cite{Kristiansen} does not concern extremal (minimax or maximin) problems but gives an existence and uniqueness result for trigonometric polynomials when the local extrema are also prescribed.}
\begin{thm}
Let $r_{0},r_{1},\ldots,r_{n}>0$ be fixed.
Then, there exists a unique system of points $0=w_{0}<w_{1}<\ldots<w_{n}<2\pi$ such that
\[
\Bigl\| \Bigl|\sin\frac{t-w_{0}}{2}\Bigr|^{r_{0}}\cdots\Bigl|\sin\frac{t-w_{n}}{2}\Bigr|^{r_{n}}\Bigr\| 
={\lila \inf_{0=y_{0}\leq y_{1}\leq \ldots\leq y_{n}\leq 2\pi} }
\Bigl\| \Bigl|\sin\frac{t-y_{0}}{2}\Bigr|^{r_{0}}\cdots\Bigl|\sin\frac{t-y_{n}}{2}\Bigr|^{r_{n}}\Bigr\| 
\]
where $\| \cdot\| $ denotes the sup-norm over $[0,2\pi]$.
The extremal GTP
\[T^{*}(t):=\Bigl|\sin\frac{t-w_{0}}{2}\Bigr|^{r_{0}}\cdots\Bigl|\sin\frac{t-w_{n}}{2}\Bigr|^{r_{n}}\]
is uniquely determined by properties that there exist $0<z_{0}<z_{1}<z_{2}<\ldots<z_{n}<2\pi$
such that $w_{j}$'s and $z_{j}$'s interlace, i.e., $0=w_{0}<z_{0}<w_{1}<\ldots<w_{n}<z_{n}<w_{0}+2\pi=2\pi$,
and $T^{*}(z_{j})=\| T^{*}\| $ for $j=0,1,\ldots,n$.\end{thm}
\begin{proof}
{\lila Let $K(x):=\log|\sin(x/2)|$
for $-\pi\leq x\leq\pi$, and extend it $2\pi$-periodically to $\RR$.}
Then $K$ is a kernel in $\Ce^2(0,2\pi)$ with $K''<0$.
Let $K_{j}(x):=r_{j}K(x)$, $j=0,1,2,\ldots,n$
be the kernels and consider the simplex $S:=S_{\id}$.
Further, let $T(\yy,t):=\prod_{j=0}^{n}|\sin\frac{t-y_{j}}{2}|^{r_{j}}$
where $\yy\in S$ and $F(\yy,t):=\log|T(\yy,t)|$.
{\lila Then $F(\yy,t)$ is a
sum of translates function, because}
\[
F(\yy,t)=K_{0}(t)+\sum_{j=1}^{n}K_{j}(t-y_{j})
=\sum_{j=0}^{n}r_{j}K(x-y_{j}).
\]
Applying Theorem  \ref{thm:mainspecialcase3},
we obtain that $M(S)=\inf_{\yy\in S}\sup_{t\in[0,2\pi)}F(\yy,t)$
is attained at exactly one point $\wS=(w_{1},\ldots,w_{n})\in S$, i.e.,
\[
M(S)=\sup_{t\in[0,2\pi)}F(\wS,t)\quad \text{and}\quad
\sup_{t\in[0,2\pi)}F(\yy,t)>M(S)\quad\text{when $\yy\ne\wS$.}
\]
   Moreover, there exist $0<z_{0}<z_{1}<z_{2}<\ldots<z_{n}<2\pi$
such that $F(\wS,z_{j})=M(S)$, that is,  $\wS$ is an equioscillation point.  The
interlacing property obviously follows. Rewriting these properties for
$T^{*}(t):=\exp F(\wS,t)$, we
obtain the assertions.
\end{proof}

We turn to the interval case. Suppose the $n$  positive real numbers $r_{1},r_{2},\ldots,r_{n}>0$ are fixed,
and consider $P(x):=|x-y_{1}|^{r_{1}}\ldots|x-y_{n}|^{r_{n}}$.
Such functions are sometimes called \emph{generalized algebraic polynomials}
(GAP, see, for instance, \cite{BorweinErdelyi} Appendix 4). Now, fix $[a,b]\subset\RR$
and consider the following minimization problem
\begin{equation}
\inf_{a\leq y_{1}<\ldots<y_{n}\leq b}\sup_{x\in [a,b]} \bigl| |x-y_{1}|^{r_{1}}\ldots|x-y_{n}|^{r_{n}}\bigr|. \label{prob:algmin}
\end{equation}

{\lila In order to solve this,} 
we will investigate the problem
\begin{equation}
{\lila \inf_{\mathbf{t}} }\sup_{t\in [0,2\pi]}\:\:\Biggl| \Bigl|\sin\frac{t-t_{1}}{2}\Bigr|^{r_{n}}\ldots\Bigl|\sin\frac{t-t_{n}}{2}\Bigr|^{r_{1}}\Bigl|\sin\frac{t-t_{n+1}}{2}\Bigr|^{r_{1}}\ldots\Bigl|\sin\frac{t-t_{2n}}{2}\Bigr|^{r_{n}}\Biggr| \label{prob:doubledtrigmin0}
\end{equation}
where the infimum is taken for {\lila $\mathbf{t}:=(t_1,\ldots,t_{2n})$ with $0\leq t_{1}\le \ldots\le t_{n}\le t_{n+1}\le\ldots\le t_{2n}\le 2\pi$ with $t_1+t_{2n}=2\pi$, the latter normalizition being natural in view of the periodicity of the occurring sine functions.} 
{\lila Note that in the original Bojanov problem the $y_j$'s are different, while we allow the $t_j$'s to coincide; this apparently larger generality leads to the same problem actually.}

\begin{thm}
\label{thm:doubled_trig_prob}
With the previous notation, the infimum in
{\piros \eqref{prob:doubledtrigmin0}} is attained at a unique point $\wS=(w_{1},w_{2},\ldots,w_{2n})$
with $w_{1}+(w_{2n}-2\pi)=0$ and $0<w_{1}<\ldots<w_{2n}<2\pi$.
Furthermore, $\wS$ is symmetric: $w_{k}=2\pi-w_{2n+1-k}$
for $k=1,2,\ldots,n$.
\end{thm}
As a consequence the minimization in \eqref{prob:doubledtrigmin0} has the same (unique) solution as
\begin{equation}
{\lila \inf_{\mathbf{t}} }
\sup_{t\in [0,2\pi]}\:\:\Biggl| \Bigl|\sin\frac{t-t_{1}}{2}\Bigr|^{r_{n}}\ldots\Bigl|\sin\frac{t-t_{n}}{2}\Bigr|^{r_{1}}\Bigl|\sin\frac{t-t_{n+1}}{2}\Bigr|^{r_{1}}\ldots\Bigl|\sin\frac{t-t_{2n}}{2}\Bigr|^{r_{n}}\Biggr|, \label{prob:doubledtrigmin}
\end{equation}
where the infimum is taken for 
{\lila $\mathbf{t}=(t_1,\ldots,t_{2n})$ and $0\leq t_{1}\le \ldots\le t_{n}\le \pi$} 
satisfying  $ t_j=2\pi-t_{2n+1-j}$, for all $j=1,\ldots,n$.

The previous theorem follows from the next, more general, symmetry theorem.
\begin{thm}
\label{thm:symm_minmax_prob}Let $K_{1},\ldots,K_{n}$ be strictly
concave kernels such that $K_{j}$ is even for all $j=1,\ldots,n$.
Assume that the kernels are either all in $\Ce^1(0,2\pi)$ or {all satisfy \eqref{eq:kernsing'}.}
Take the simplex $S:=\{ 0\leq y_{1}<y_{2}<\ldots<y_{2n}<2\pi\} $.
Define the symmetric
sum of translates function
\begin{multline}
F_{\symm}(\yy,t):=K_{1}(t-y_{1})+\ldots+K_{n-1}(t-y_{n-1})+K_{n}(t-y_{n})\\
+K_{n}(t-y_{n+1})+K_{n-1}(t-y_{n+2})+\ldots+K_{1}(t-y_{2n})\label{def:symm_pot}
\end{multline}
and consider the
``doubled'' problem
\begin{equation}
M_{\symm}:={\lila \inf_{\yy\in \overline{S} }} \sup_{t\in[0,2\pi)}F_{\symm}(\yy,t).\label{def:symm_M}
\end{equation}
Then there is a unique minimum point $\wS=(w_{1},w_{2},\ldots,w_{2n})\in S$
with $w_{1}+(w_{2n}-2\pi)=0$. Furthermore, $\wS$
is symmetric: $w_{k}=2\pi-w_{2n+1-k}$ ($k=1,2,\ldots,n$) and there
are exactly $2n$ points: $0=z_{1}<z_{2}<\ldots<z_{n+1}=\pi<\ldots<z_{2n}$
where $F_{\symm}(\wS,\cdot)$ attains its supremum.
Moreover, $z_{j}$'s and $w_{j}$'s interlace and $z_{j}$'s are symmetric
too: $z_{k}=2\pi-z_{2n+1-k}$ ($k=1,2,\ldots,n$).\end{thm}
\begin{proof}
{\piros Following the symmetric definition, we let $K_{n+k}(t):=K_{n+1-k}(-t)$
where $k=1,2,\ldots,n$.  By symmetry we have
\begin{equation}
K_{n+k}(t)=K_{n+1-k}(t)\quad \mbox{for }k=-n+1,\ldots,n.\label{eq:kerndoubled}
\end{equation}}
Hence $F_{\symm}(\yy,t)=\sum_{j=1}^{2n}K_{j}(t-y_{j})$.

The existence and uniqueness follow from Theorem \ref{thm:mainspecialcase3}.
That is, there exists a unique $\wS=(w_{1},w_{2},\ldots,w_{2n})\in S$ (unique
with $w_{1}=0$) such that $M(S)=\MM(\wS)$).
Furthermore, $M(S)=m(S)$ and
$F(\wS,\cdot)$ equioscillates, hence
$m(S)=\mm(\wS)$.
Using rotation, we can assume that $w_{1}>0$ is such that $w_{1}+(w_{2n}-2\pi)=0$.

Now, we establish $w_{k}=2\pi-w_{2n+1-k}$ ($k=1,2,\ldots,n$). By
the assumption, it holds for $k=1$, i.e., $w_{1}=2\pi-w_{2n}$.
Reflect the $w_{k}$'s: $v_{k}:=2\pi-w_{2n+1-k}$, $k=1,\ldots,2n$
and write $\vv:=(v_{1},\ldots,v_{2n})$. Then $v_{1}=w_{1}$
and $v_{2n}=w_{2n}$. Furthermore, put $L_{k}(t):=K_{2n+1-k}(-t)$
and consider
\[
\tilde{F}(\vv,t):={\lila \sum_{k=1}^{2n}L_{k}(t-v_{k}) }
\]
the sum of translates function
of the reflected configuration. We obtain, by using \eqref{eq:kerndoubled}
and the {\kek symmetry} of the kernels, that
{\lila
\begin{align*}
L_{k}(t-v_{k})&=K_{2n+1-k}(v_{k}-t)=K_{2n+1-k}(t-v_{k})\\
&=K_{2n+1-k}(t-2\pi+w_{2n+1-k})=K_{2n+1-k}(t-w_{2n+1-k})
\end{align*} }
for all {\lila $k=1,\ldots,2n$}. Hence
\begin{align*}
\tilde{F}(\vv,t)
&={\lila \sum_{k=1}^{2n}L_{k}(t-v_{k})=\sum_{k=1}^{2n}K_{2n+1-k}((2\pi-t)-w_{2n+1-k}) }\\
&=F_{\symm}(\wS,2\pi-t)=F_{\symm}(\wS,-t).
\end{align*}
Obviously $\vv\in S$.
By definition, $m_0{(\wS)}=m_{2n}{(\wS)}=\sup\{F_{\symm}(\wS,t): w_{2n}-2\pi\le t\le w_{1}\}$ and $m_j{(\wS)}=\sup\{F_{\symm}(\wS,t): w_j\le t\le w_{j+1}\}$,
$j=1,\ldots,2n-1$,
and similarly for $\vv$,  $m_j{(\vv)}=\sup\{\tilde{F}(\vv,t): v_j\le t\le v_{j+1}\}$,
$j=1,\ldots,2n-1$ and
\[m_0{(\vv)}=m_{2n}{(\vv)}=\sup\{\tilde{F}(\vv,t): v_{2n}-2\pi\le t\le v_{1}\}.\]
Hence, we also have for $j=1,\ldots,2n-1$
\begin{align*}
m_{j}{(\wS)}
&=\sup\{F_{\symm}(\wS,t): w_j\le t \le w_{j+1}\}\\
&=\sup\{F_{\symm}(\wS,-t): -w_{j+1}\le t \le -w_{j}\}=\sup\{\tilde{F}(\vv,t): -w_{j+1}\le t \le -w_{j}\}\\
&=\sup\{\tilde{F}(\vv,t): 2\pi-w_{j+1}\le t \le 2\pi-w_{j}\}=\sup\{\tilde{F}(\vv,t): v_{2n-j}\le t \le v_{2n+1-j}\}\\
&=m_{2n-j}{(\vv)},
\end{align*}
and obviously $m_0{(\vv)}=m_{2n}{(\vv)}=m_0{(\wS)}=m_{2n}{(\wS)}$.
This implies that {\lila together with $m_j(\wS)$, also $m_j(\vv)$ provides $\MM(\wS)=\MM(\vv)$, whence by uniqueness  $\vv=\wS$.}
{\lila Therefore, 
$w_{k}=2\pi-w_{2n+1-k}$
($k=1,2,\ldots,n$), too.}
The symmetry of the $w_{k}$'s implies the remaining assertions (interlacing
and symmetry of the $z_{j}$'s).
\end{proof}
We connect the ``algebraic'' problem \eqref{prob:algmin} and the ``trigonometric'' problem \eqref{prob:doubledtrigmin}
by using a classical idea of transferring between these situations with $x=\cos t$ (see, e.g., \cite{Szego}).
\begin{lemma}
	Let $L(x):=\frac{b-a}{2}x+\frac{b+a}{2}$. The identities
	\begin{equation}\label{eq:asserted}
	y_{j}=L(\cos t_{n+1-j}),\:\:
	t_{n+1-j}=\arccos L^{-1}(y_j), \:\: t_{n+j}=2\pi-\arccos L^{-1}(y_j)
	\end{equation}
	for $j=1,\dots, n$
	provide a one-to-one correspondence between generalized algebraic polynomials in \eqref{prob:algmin} and generalized trigonometric polynomials in \eqref{prob:doubledtrigmin}.
	Similarly, for the corresponding interlacing points of maxima we have $s_{j}=L(\cos z_{n+1-j})$, $z_{n+1-j}=\arccos L^{-1}(s_j)$ and  $z_{n+j}=2\pi-\arccos L^{-1}(s_j)$ for $j=0,\dots, n$.
\end{lemma}
\begin{proof}
For simplicity, assume that $a=-1$, $b=1$, hence $L(x)=x$.
Recall
\begin{equation}
\sin\frac{t-\alpha}{2}\sin\frac{t+\alpha-2\pi}{2}=\frac{1}{2}(\cos t-\cos \alpha)\label{eq:cos_id}
\end{equation}
hence
\begin{align}
&\Bigl|\sin\frac{t-t_{1}}{2}\Bigr|^{r_{n}}\cdots\Bigl|\sin\frac{t-t_{n}}{2}\Bigr|^{r_{1}}\Bigl|\sin\frac{t+t_{n}-2\pi}{2}\Bigr|^{r_{1}}\cdots\Bigl|\sin\frac{t+t_{1}-2\pi}{2}\Bigr|^{r_{n}}\notag\\
&\qquad=\frac{1}{2^{\sum_{j=1}^{n}r_{j}}}|\cos t-\cos t_{1}|^{r_{n}}\cdots|\cos t-\cos t_{n}|^{r_{1}}\label{eq:cos_mult}.
\end{align}
Therefore, for every GAP $P(x)=|x-y_{1}|^{r_{1}}\ldots|x-y_{n}|^{r_{n}}$
there is a {\lila symmetric GTP $T(t)$ (of the form as in \eqref{prob:doubledtrigmin})} 
such that $P(\cos t)=2^{-\sum_{j=1}^{n}r_{j}}T(t)$.  Also to every GTP $T(t)$ as appearing in \eqref{prob:doubledtrigmin}, there is a corresponding GAP as in \eqref{prob:algmin} (modulo a constant factor), where between the zeros $t_j$, $t_{n+1-j}$ and $y_j$ ($j=1,\dots,n$) the asserted relations \eqref{eq:asserted} hold and $P(\cos t)=2^{-\sum_{j=1}^{n}r_{j}}T(t)$.  The statement about the points of maxima is now obvious.
\end{proof}
From this the following generalization of Bojanov's result can be deduced immediately:
\begin{thm}\label{thm:bojanovgen}
Let $\nu_{1},\ldots,\nu_{n}>0$ be fixed, and let $[a,b]\subset\RR$.
Then, there exists a unique system of points
$a<x_{1}<\ldots<x_{n}<b$ such that
\[
\| |x-x_{1}|^{\nu_{1}}\ldots|x-x_{n}|^{\nu_{n}}\| =\inf_{a\leq y_{1}<\ldots<y_{n}\leq b}\| |x-y_{1}|^{\nu_{1}}\ldots|x-y_{n}|^{\nu_{n}}\|
\]
where $\|\cdot\| $ denotes the sup-norm over $[a,b]$.
The extremal generalized polynomial \[
P^{*}(x):=|x-x_{1}|^{\nu_{1}}\ldots|x-x_{n}|^{\nu_{n}}\]
is uniquely characterized by the existence of $a=s_{0}<s_{1}<\ldots<s_{n-1}<s_{n}=b$ with
 $|P^{*}(s_{j})|=\| P^{*}\| $ for
$j=0,1,\ldots,n$.
\end{thm}

{\lila\begin{remark}
In retrospect, we see here that considering the
(in general, different) extremal quantities and problems
on each simplex separately provides us a more precise
result than just considering $M$ and $m$ as in \eqref{eq:minmax} and
\eqref{eq:maxmin}.
To obtain Bojanov's theorem for each fixed ordered $n$-tuples
$(\nu_1,\ldots,\nu_n)$ one needs this more precise version.
\end{remark}}

\providecommand{\bysame}{\leavevmode\hbox to3em{\hrulefill}\thinspace}
\providecommand{\MR}{\relax\ifhmode\unskip\space\fi MR }

\providecommand{\MRhref}[2]{
  \href{http://www.ams.org/mathscinet-getitem?mr=#1}{#2}
}
\providecommand{\href}[2]{#2}

\medskip

 B.{} Farkas\\
 School of Mathematics and 
  Natural Sciences,\\ University of Wuppertal\\
   Gau{\ss}stra{\ss}e 20\\
  42119 Wuppertal, Germany\\
\email{farkas@math.uni-wuppertal.de}

\smallskip

B.{} Nagy\\
 MTA-SZTE Analysis and 
 Stochastics Research Group,\\
  Bolyai Institute, University of Szeged\\
  Aradi v\'ertanuk tere 1\\
   6720 Szeged, Hungary\\
\email{nbela@math.u-szeged.hu}

\smallskip

Sz.{} Gy.{} R\'ev\'esz\\
 Institute of Mathematics and Informatics, 
 Faculty of Sciences,\\
University of Pécs\\
 Vasvári Pál utca 4\\
7622 Pécs,  Hungary; and \\[1ex]
 Alfr\'ed R\'enyi Institute of Mathematics\\
 Re\'altanoda utca 13-15\\ 
 1053 Budapest, Hungary \\
 \email{revesz.szilard@renyi.mta.hu}

\end{document}